\documentclass[draft]{amsart}

\usepackage{amsmath,amssymb,amsthm,enumerate} 
\usepackage{amsfonts} 
\usepackage{color}

\theoremstyle{plain}
\newtheorem{pretheo}{Theorem}[section]
\newtheorem{preassu}[pretheo]{Assumption}
\newtheorem{precoro}[pretheo]{Corollary}
\newtheorem{predefi}[pretheo]{Definition}
\newtheorem{preexam}[pretheo]{Example}
\newtheorem{prelemm}[pretheo]{Lemma}
\newtheorem{preprop}[pretheo]{Proposition}
\newtheorem{prerema}[pretheo]{Remark}

\newenvironment{theo}{\begin{pretheo}}{\end{pretheo}}

\newenvironment{coro}{\begin{precoro}}{\end{precoro}}
\newenvironment{defi}{\begin{predefi}}{\end{predefi}}

\newenvironment{lemm}{\begin{prelemm}}{\end{prelemm}}

\newenvironment{rema}{\begin{prerema}\rm}{\end{prerema}}

\DeclareMathOperator{\dv}{div}
\DeclareMathOperator{\hlm}{Hol}

\DeclareMathOperator{\Dv}{Div}

\newcommand{\intd}{\,d}

\newcommand{\pd}{\partial}
\newcommand{\wht}[1]{\widehat{#1}}
\newcommand{\wtd}[1]{\widetilde{#1}}

\newcommand{\BBM}{\mathbb{M}}

\newcommand{\Ba}{\mathbf{a}}
\newcommand{\Bb}{\mathbf{b}}

\newcommand{\Bf}{\mathbf{f}}

\newcommand{\Bh}{\mathbf{h}}

\newcommand{\Bn}{\mathbf{n}}

\newcommand{\Bu}{\mathbf{u}}
\newcommand{\Bv}{\mathbf{v}}

\newcommand{\BC}{\mathbf{C}}
\newcommand{\BD}{\mathbf{D}}
\newcommand{\BE}{\mathbf{E}}
\newcommand{\BF}{\mathbf{F}}
\newcommand{\BG}{\mathbf{G}}
\newcommand{\BH}{\mathbf{H}}
\newcommand{\BI}{\mathbf{I}}

\newcommand{\BK}{\mathbf{K}}
\newcommand{\BL}{\mathbf{L}}
\newcommand{\BM}{\mathbf{M}}
\newcommand{\BN}{\mathbf{N}}

\newcommand{\BR}{\mathbf{R}}
\newcommand{\BS}{\mathbf{S}}
\newcommand{\BT}{\mathbf{T}}
\newcommand{\BU}{\mathbf{U}}
\newcommand{\BV}{\mathbf{V}}

\newcommand{\CA}{\mathcal{A}}
\newcommand{\CB}{\mathcal{B}}

\newcommand{\CF}{\mathcal{F}}
\newcommand{\CG}{\mathcal{G}}

\newcommand{\CK}{\mathcal{K}}
\newcommand{\CL}{\mathcal{L}}
\newcommand{\CM}{\mathcal{M}}

\newcommand{\CP}{\mathcal{P}}

\newcommand{\CR}{\mathcal{R}}
\newcommand{\CS}{\mathcal{S}}
\newcommand{\CT}{\mathcal{T}}

\newcommand{\CX}{\mathcal{X}}
\newcommand{\CY}{\mathcal{Y}}
\newcommand{\CZ}{\mathcal{Z}}

\newcommand{\Fm}{\mathfrak{m}}
\newcommand{\Fn}{\mathfrak{n}}

\newcommand{\Fp}{\mathfrak{p}}
\newcommand{\Fq}{\mathfrak{q}}

\newcommand{\FA}{\mathfrak{A}}
\newcommand{\FB}{\mathfrak{B}}

\newcommand{\FX}{\mathfrak{X}}
\newcommand{\FY}{\mathfrak{Y}}

\newcommand{\SSN}{\mathsf{N}}

\newcommand{\SST}{\mathsf{T}}

\newcommand{\vps}{\varepsilon}

\newcommand{\vph}{\varphi}

\numberwithin{equation}{section} 
\allowdisplaybreaks[4]

\begin{document}
\title[Compressible fluid model of Korteweg type]
{Existence of $\CR$-bounded solution operator families
for a compressible fluid model of Korteweg type on the half-space}

\author[Hirokazu Saito]{Hirokazu Saito}
\address{Faculty of Industrial Science and Technology,
Tokyo University of Science,
102-1 Tomino, Oshamambe-cho, Yamakoshi-gun,
Hokkaido 049-3514, Japan}
\email{hsaito@rs.tus.ac.jp}



\thanks{This research was partly supported by JSPS Grant-in-aid for Young Scientists (B) \#17K14224.}



\begin{abstract}
The aim of this paper is to show the existence of $\CR$-bounded solution operator families
for a generalized resolvent problem on the half-space 
arising from a compressible fluid model of Korteweg type. 
Such a compressible fluid model was derived by Dunn and Serrin (1985)
and studied by Kotschote (2008) as an initial-boundary value problem.
\end{abstract}

\maketitle

\section{Introduction and main results}\label{sec1}

\subsection{Introduction}\label{subsec1-1}
Let $\BR_+^N$ and $\BR_0^N$ be respectively the half-space and the boundary of $\BR_+^N$ for $N\geq 2$,
i.e.
\begin{align*}
\BR_+^N
&=\{x=(x',x_N) \mid x'=(x_1,\dots,x_{N-1})\in\BR^{N-1}, x_N>0\}, \\
\BR_0^N
&=\{x=(x',x_N) \mid x'=(x_1,\dots,x_{N-1})\in\BR^{N-1},x_N=0\}.
\end{align*}

This paper is concerned with the existence of $\CR$-bounded solution operator families
for a generalized resolvent problem on $\BR_+^N$ 
arising from a compressible fluid model of Korteweg type as follows:
\begin{equation}\label{eq1:gen-res}
\left\{\begin{aligned}
\lambda\rho+\dv\Bu&=d && \text{in $\BR_+^N$,} \\
\lambda\Bu-\mu_*\Delta\Bu-\nu_*\nabla\dv\Bu-\kappa_*\nabla\Delta\rho&=\Bf && \text{in $\BR_+^N$,} \\
\Bn\cdot\nabla\rho=g, \quad \Bu&=0 && \text{on $\BR_0^N$.}
\end{aligned}\right.
\end{equation}

Here $\lambda$ is the resolvent parameter varying in 
\begin{equation}\label{defi:sector}
\Sigma_\vps=\{\lambda\in\BC\setminus\{0\} \mid |\arg\lambda|<\pi-\vps\}
\end{equation}
for $\vps\in(0,\pi/2)$;
$\rho=\rho(x)$ and $\Bu=\Bu(x)=(u_1(x),\dots,u_N(x))^\SST$ are respectively the fluid density and the fluid velocity that are unknown functions;
$d=d(x)$, $\Bf=\Bf(x)=(f_1(x),\dots,f_N(x))^\SST$, and $g=g(x)$ are given functions;
$\Bn=(0,\dots,0,-1)^\SST$ is the outward unit normal vector on $\BR_0^N$;
$\Ba\cdot\Bb=\sum_{j=1}^N a_j b_j$ for any $N$-vectors $a=(a_1,\dots,a_N)^\SST$ and $\Bb=(b_1,\dots,b_N)^\SST$.
Here and subsequently, one uses the following notation for differentiations:
Let $u=u(x)$, $\Bv=(v_1(x),\dots,v_N(x))^\SST$, and $\BM=(M_{ij}(x))$ be a scalar-, a vector-, and an $N\times N$ matrix-valued function
defined on a domain of $\BR^N$, and then for $\pd_j=\pd/\pd x_j$
\begin{align*}
&\nabla u = (\pd_1 u,\dots,\pd_N u)^\SST, \quad \Delta u=\sum_{j=1}^N\pd_j^2 u, \quad \Delta\Bv=(\Delta v_1,\dots,\Delta v_N)^\SST, \\
&\dv\Bv=\sum_{j=1}^N\pd_j v_j, \quad 
\nabla\Bv=\{\pd_j v_k \mid j,k=1,\dots,N\}, \\
&\nabla^2\Bv=\{\pd_j\pd_k v_l \mid j,k,l=1,\dots,N\}, \quad
\Dv\BM=\bigg(\sum_{j=1}^N\pd_j M_{1j},\dots,\sum_{j=1}^N\pd_j M_{Nj}\bigg)^\SST.
\end{align*}

The $\mu_*$ and $\nu_*$ are positive constants describing the viscosity coefficients,
while $\kappa_*$ is a positive constant describing the capillarity coefficient (cf. e.g. \cite{Saito1} for more detail).
In \cite{Saito1}, the author assumes 
\begin{equation}\label{defi:eta_*}
\eta_*:=\left(\frac{\mu_*+\nu_*}{2\kappa_*}\right)^2-\frac{1}{\kappa_*}\neq 0  \quad \text{and} \quad  \kappa_*\neq \mu_*\nu_*
\end{equation}
in order to show the existence of $\CR$-bounded solution operator families 
for a generalized problem with free boundary condition.
On the other hand, this paper shows the existence of $\CR$-bounded solutions operator families associated with \eqref{eq1:gen-res}
for arbitrary positive constants $\mu_*$, $\nu_*$, and $\kappa_*$. 
To this end, we divide in the present paper the proof into five cases as follows:
\begin{align*}
\begin{aligned}
&\text{{\bf Case I}: } && \eta_*<0; \\
&\text{{\bf Case II}: } && \eta_*>0 \quad \text{and} \quad \kappa_*\neq \mu_*\nu_*; \\
&\text{{\bf Case III}: } && \eta_*>0 \quad \text{and} \quad \kappa_*= \mu_*\nu_*; \\ 
&\text{{\bf Case IV}: } && \eta_*=0 \quad \text{and} \quad \kappa_*\neq \mu_*\nu_*; \\
&\text{{\bf Case V}: } && \eta_*=0 \quad \text{and} \quad \kappa_*= \mu_*\nu_*.
\end{aligned}
\end{align*}
Throughout this paper, Cases I, II, III, IV, and V mean the above five cases.

The motion of barotropic compressible fluids is governed by 
\begin{equation*}
\begin{aligned}
\pd_t\rho+\dv(\rho\Bu)&=0 && \text{(mass conservation),} \\
\rho(\pd_t\Bu+\Bu\cdot\nabla\Bu)&=\Dv(\BT-P(\rho)\BI) && \text{(momentum conservation),}
\end{aligned}
\end{equation*}
subject to the initial condition and suitable boundary conditions,
where $\BI$ is the $N\times N$ identity matrix and
$P:[0,\infty)\to\BR$ is a given smooth function describing the pressure.
Dunn and Serrin have shown in \cite{DS85} that for a special material of Korteweg type 
the stress tensor $\BT$ is given by
\begin{align*}
&\BT = \BS(\Bu)+\BK(\rho), \quad \BS(\Bu)=\mu\BD(\Bu)+(\nu-\mu)\dv\Bu\BI, \\
&\BK(\rho)=\frac{\kappa}{2}(\Delta^2\rho-|\nabla\rho|^2)\BI-\kappa\nabla\rho\otimes\nabla\rho,
\end{align*}
where $\BS(\Bu)$ is the usual viscous stress tensor with viscosity coefficients $\mu$, $\nu$
and $\BD(\Bu)$ is the doubled strain tensor whose $(i,j)$-component is given by $\pd_i u_j+\pd_j u_i$;
$\BK(\rho)$ is called the Korteweg tensor with a capillarity coefficient $\kappa$.
Such a Korteweg-type model was mathematically studied by Kotschote \cite{Kotschote08, Kotschote10, Kotschote12, Kotschote14}
as initial-boundary value problems.
In a forthcoming paper \cite{Saito2},
we prove a maximal regularity for some linearized system of an initial-boundary value problem
of the Korteweg-type model on general domains by using results obtained in this paper. 
We also refer to \cite{Saito2} in order to see a brief history of mathematical studies of Korteweg-type models.

Throughout this paper, the letter $C$ denotes generic constants and
$C_{a,b,c,\dots}$ means that the constant depends on the quantities $a,b,c,\dots$.
The values of constants $C$ and $C_{a,b,c,\dots}$ may change from line to line.

\subsection{Main results}\label{subsec1-2}
To state our main results, we first introduce the notation and the definition of the $\CR$-boundedness of operator families.

The set of all natural numbers is denoted by $\BN$ and $\BN_0=\BN\cup\{0\}$,
while the set of all complex numbers is denoted by $\BC$ and $\BC_+=\{z\in\BC \mid \Re z>0\}$.
Let $q\in[1,\infty]$ and $G$ be a domain of $\BR^N$.
Then $L_q(G)$ and $H_q^m(G)$, $m\in\BN$, denote respectively the usual Lebesgue spaces on $G$
and the usual Sobolev spaces on $G$. 
One sets $H_q^0(G)=L_q(G)$ and denotes the norm of $H_q^n(G)$, $n\in\BN_0$, by $\|\cdot\|_{H_q^n(G)}$.

Let $X$ and $Y$ be Banach spaces. Then $X^m$, $m\in\BN$, denotes the $m$-product space of $X$,
while the norm of $X^m$ is usually denoted by $\|\cdot\|_X$ for short.
The symbol $\CL(X,Y)$ stands for the Banach space of all bounded linear operators from $X$ to $Y$,
and $\CL(X)$ is the abbreviation of $\CL(X,X)$.
For a domain $U$ of $\BC$,
$\hlm(U,\CL(X,Y))$ is the set of all $\CL(X,Y)$-valued holomorphic functions defined on $U$.

For the right member $(d,\Bf,g)$ of \eqref{eq1:gen-res}, we set
$$
\CX_q^1(G) = H_q^1(G)\times L_q(G)^N, \quad
\CX_q^2(G)=H_q^1(G)\times L_q(G)^N\times H_q^2(G).
$$
Let $\BF^1=(d,\Bf)\in\CX_q^1(G)$ and $\BF^2=(d,\Bf,g)\in \CX_q^2(G)$.
Symbols $\FX_q^j(G)$ and $\CF_\lambda^j$ $(j=1,2)$ are then defined as follows:
\begin{align*}
&\FX_q^1(G)=L_q(G)^{\SSN_1}, \quad \SSN_1=N+1+N, \quad
\CF_\lambda^1\BF^1 = (\nabla d, \lambda^{1/2}d, \Bf)\in\FX_q^1(G); \\
&\FX_q^2(G)=L_q(G)^{\SSN_2}, \quad \SSN_2=N+1+N+N^2+N+1, \\
&\CF_\lambda^2\BF^2 = (\nabla d, \lambda^{1/2}d, \Bf, \nabla^2 g,\lambda^{1/2}\nabla g, \lambda g)\in\FX_q^2(G).
\end{align*}
One also sets for solutions of \eqref{eq1:gen-res}
\begin{alignat}{2}\label{defi:AB}
\FA_q^0(G) &= L_q(G)^{N^3+N^2+N+1}, \quad
&\CS_\lambda^0 \rho &= (\nabla^3 \rho, \lambda^{1/2}\nabla^2\rho,\lambda\nabla\rho,\lambda^{3/2}\rho); \\
\FB_q(G) &= L_q(G)^{N^3+N^2+N}, \quad
&\CT_\lambda\Bu &= (\nabla^2\Bu,\lambda^{1/2}\nabla\Bu,\lambda\Bu). \notag
\end{alignat}

\begin{rema}
The above symbols $\CX_q^1(G)$, $\CX_q^2(G)$, $\FX_q^1(G)$, $\CF_\lambda^1$, $\FX_q^2(G)$, $\CF_\lambda^2$,
$\FA_q^0(G)$, $\CS_\lambda^0$, $\FB_q(G)$, and $\CT_\lambda$ are defined in the same manner as \cite{Saito2}.
\end{rema}

The definition of the $\CR$-boundedness of operator families is as follows:

\begin{defi}\label{defi:R}
Let $X$ and $Y$ be Banach spaces.
A family of operators $\CT\subset\CL(X,Y)$ is called $\CR$-bounded on $\CL(X,Y)$,
if there exist constants $p\in[1,\infty)$ and $C>0$ such that the following assertion holds true:
For each $m\in\BN$, $\{T_j\}_{j=1}^m\subset \CT$, $\{f_j\}_{j=1}^m\subset X$
and for all sequences $\{r_j(u)\}_{j=1}^m$ of independent, symmetric, $\{-1,1\}$-valued random variables on $[0,1]$,
there holds 
\begin{equation*}
\Bigg(\int_0^1\Big\|\sum_{j=1}^m r_j(u)T_j f_j\Big\|_Y^p\intd u\Bigg)^{1/p}
\leq C\Bigg(\int_0^1\Big\|\sum_{j=1}^m r_j(u)f_j\Big\|_X^p\intd u\Bigg)^{1/p}.
\end{equation*}
The smallest such $C$ is called $\CR$-bound of $\CT$ on $\CL(X,Y)$
and denoted by $\CR_{\CL(X,Y)}(\CT)$.
\end{defi}

\begin{rema}\label{rema:Kahane}
\begin{enumerate}[(1)]
\item
The constant $C$ in Definition \ref{defi:R} may depend on $p$.
\item
It is known that $\CT$ is $\CR$-bounded for any
$p\in[1,\infty)$, provided that $\CT$ is $\CR$-bounded for some $p\in[1,\infty)$.
This fact follows from Kahane's inequality (cf. e.g. \cite[Theorem 2.4]{KW04}).
\end{enumerate}
\end{rema}

One now states the main result of this paper.

\begin{theo}\label{theo:main}
Let $q\in(1,\infty)$ and assume that $\mu_*$, $\nu_*$, and $\kappa_*$ are positive constants.
Then, for any $\lambda\in\BC_+$, there exist operators $\CA(\lambda)$ and $\CB(\lambda)$, with
\begin{align*}
\CA(\lambda)&\in\hlm(\BC_+,\CL(\FX_q^2(\BR_+^N),H_q^3(\BR_+^N)), \\
\CB(\lambda)&\in \hlm(\BC_+,\CL(\FX_q^2(\BR_+^N),H_q^2(\BR_+^N)^N)),
\end{align*}
such that, for any $\BF=(d,\Bf,g)\in\CX_q^2(\BR_+^N)$, 
$(\rho,\Bu)=(\CA(\lambda)\CF_\lambda^2\BF,\CB(\lambda)\CF_\lambda^2\BF)$ is a unique solution to the system \eqref{eq1:gen-res}.
In addition, for $n=0,1$,
\begin{align*}
\CR_{\CL(\FX_q^2(\BR_+^N),\FA_q^0(\BR_+^N))}\left(\left\{\left.\left(\lambda\frac{d}{d\lambda}\right)^n
\left(\CS_\lambda^0\CA(\lambda)\right) \right| \lambda\in\BC_+\right\}\right)
&\leq C_{N,q,\mu_*,\nu_*,\kappa_*}, \\
\CR_{\CL(\FX_q^2(\BR_+^N),\FB_q(\BR_+^N))}\left(\left\{\left.\left(\lambda\frac{d}{d\lambda}\right)^n
\left(\CT_\lambda\CB(\lambda)\right) \right| \lambda\in\BC_+\right\}\right)
&\leq C_{N,q,\mu_*,\nu_*,\kappa_*}, 
\end{align*}
where $C_{N,q,\mu_*,\nu_*,\kappa_*}$ is a positive constant depending on at most $N$, $q$, $\mu_*$, $\nu_*$, and $\kappa_*$.
\end{theo}

This paper is organized as follows:
The next section first introduces some classes of symbols and their fundamental properties.
Secondly, one computes characteristic roots that appear in analysis of a system of ordinary differential equations
in the Fourier space in Sections 3, 4, 5, and 6.
Thirdly, we introduce technical lemmas that play an important role in proving
the existence of $\CR$-bounded solution operator families.
Fourthly, the system \eqref{eq1:gen-res} is reduced to the case where $(d,\Bf)=(0,0)$,
and the main theorem (i.e. the existence of $\CR$-bounded solution operator families) is stated for the reduced system.
Furthermore, one proves Theorem \ref{theo:main}, assuming the main theorem of the reduced system holds.
Section 3 proves the main theorem of the reduced system for Cases I and II.
Section 4, 5, and 6 treat respectively Case III, IV, and V, and prove the matin theorem of the reduced system.

\section{Preliminaries}\label{sec2}
\subsection{Classes of symbols}\label{subsec2-1}
Recall that $\Sigma_\vps$ is given in \eqref{defi:sector} for $\vps\in(0,\pi/2)$.
Let $\Lambda=\Sigma_\vps$ or $\Lambda=\BC_+$,
and let $m(\xi',\lambda)$ be a function, defined on $(\BR^{N-1}\setminus\{0\})\times \Lambda$,
that is infinitely many times differentiable with respect to $\xi'=(\xi_1,\dots,\xi_{N-1})$ and holomorphic with respect to $\lambda$.
For any multi-index $\alpha'=(\alpha_1,\dots,\alpha_{N-1})\in\BN_0^{N-1}$,
let us define $\pd_{\xi'}^{\alpha'}$ by
$$
\pd_{\xi'}^{\alpha'} = \frac{\pd^{|\alpha'|}}{\pd\xi_1^{\alpha_1}\dots\pd\xi_{N-1}^{\alpha_{N-1}}}, \quad 
|\alpha'|=\alpha_1+\dots+\alpha_{N-1}.
$$
If there exists a real number $r$ such that
for any multi-index $\alpha'=(\alpha_1,\dots,\alpha_{N-1})\in\BN_0^{N-1}$ 
and $(\xi',\lambda)\in(\BR^{N-1}\setminus\{0\})\times\Lambda$
\begin{equation*}
\left|\pd_{\xi'}^{\alpha'}\left(\left(\lambda\frac{d}{d\lambda}\right)^n m(\xi',\lambda)\right)\right|
\leq C(|\lambda|^{1/2}+|\xi'|)^{r-|\alpha'|} \quad (n=0,1)
\end{equation*}
with some positive constant $C$ depending solely on $N$, $r$, $\alpha'$, and $\vps$,  
then $m(\xi',\lambda)$ is called a multiplier of order $r$ with type $1$.
If there exists a real number $r$ such that
for any multi-index $\alpha'=(\alpha_1,\dots,\alpha_{N-1})\in\BN_0^{N-1}$ 
and $(\xi',\lambda)\in(\BR^{N-1}\setminus\{0\})\times\Lambda$
\begin{equation*}
\left|\pd_{\xi'}^{\alpha'}\left(\left(\lambda\frac{d}{d\lambda}\right)^n m(\xi',\lambda)\right)\right|
\leq C(|\lambda|^{1/2}+|\xi'|)^r|\xi'|^{-|\alpha'|} \quad (n=0,1)
\end{equation*}
with some positive constant $C$ depending solely on $N$, $r$, $\alpha'$, and $\vps$,
then $m(\xi',\lambda)$ is called a multiplier of order $r$ with type $2$.

Here and subsequently, we denote the set of all symbols of order $r$ with type $j$ on  
$(\BR^{N-1}\setminus\{0\})\times\Lambda$ by $\BBM_{r,j}(\Lambda)$.
For instance,
\begin{equation*}
\xi_k/|\xi'|\in \BBM_{0,2}(\Lambda), \quad \xi_k,\lambda^{1/2}\in\BBM_{1,1}(\Lambda) \quad (k=1,\dots,N-1),
\end{equation*}
and also $|\xi'|^2,\lambda \in\BBM_{2,1}(\Lambda)$.
One notes for $r\in\BR$ and $j=1,2$ that $\BBM_{r,j}(\Lambda)$ are vector spaces on $\BC$
and that $\BBM_{r,j}(\Sigma_\vps)\subset \BBM_{r,j}(\BC_+)$ for any $\vps\in(0,\pi/2)$.
In addition, we know the following fundamental properties of $\BBM_{r,j}(\Lambda)$ (cf. \cite[Lemma 5.1]{SS12}).

\begin{lemm}\label{lemm:algebra}
Let $r_1,r_2\in\BR$, and let $\Lambda=\Sigma_\vps$ for some $\vps\in(0,\pi/2)$ or $\Lambda=\BC_+$. 
Then the following assertions hold true:
\begin{enumerate}[$(1)$]
\item
Given $l_j\in\BBM_{r_j,1}(\Lambda)$ $(j=1,2)$,
we have $l_1l_2\in \BBM_{r_1+r_2,1}(\Lambda)$.
\item
Given $m_j\in\BBM_{r_j,j}(\Lambda)$ $(j=1,2)$,
we have $m_1m_2\in\BBM_{r_1+r_2,2}(\Lambda)$.
\item
Given $n_j\in\BBM_{r_j,2}(\Lambda)$ $(j=1,2)$,
we have $n_1n_2\in\BBM_{r_1+r_2,2}(\Lambda)$.
\end{enumerate}
\end{lemm}

\subsection{Characteristic roots}\label{subsec2-2}
Let $\mu_*$, $\nu_*$, and $\kappa_*$ be positive constants,
and let us define a polynomial $\CP(s)$ by
\begin{equation*}\label{defi:p_*}
\CP(s)=s^2-\frac{\mu_*+\nu_*}{\kappa_*}s+\frac{1}{\kappa_*}.
\end{equation*}
The roofs $s_\pm$ of $\CP(s)$ are then given by
\begin{equation*}
s_\pm=\left\{\begin{aligned}
&\frac{\mu_*+\nu_*}{2\kappa_*}\pm\sqrt{\eta_*} && (\eta_*\geq 0), \\
&\frac{\mu_*+\nu_*}{2\kappa_*}\pm i\sqrt{|\eta_*|} && (\eta_*<0),
\end{aligned}\right.
\end{equation*}
where $i=\sqrt{-1}$ and $\eta_*$ is given in \eqref{defi:eta_*}. 
Setting
\begin{equation*}
s_1=s_-, \quad s_2=s_+,
\end{equation*}
we have

\begin{lemm}\label{lemm:roots_s}
Let $\mu_*$, $\nu_*$, and $\kappa_*$ be positive constants.
Then the following assertions hold true:
\begin{enumerate}[$(1)$]
\item\label{lemm:roots_s1}
If $\mu_*$, $\nu_*$, and $\kappa_*$ satisfy the condition of Case I, then
$s_1$ and $s_2$ are imaginary numbers with $\Re s_j>0$ $(j=1,2)$. Especially, in Case I,
\begin{equation*}
s_1 \neq s_2, \quad s_1\neq \mu_*^{-1}, \quad s_2\neq \mu_*^{-1}.
\end{equation*}
\item\label{lemm:roots_s2}
If $\mu_*$, $\nu_*$, and $\kappa_*$ satisfy the condition of Case II, then
$s_1$ and $s_2$ are real numbers with $s_j>0$ $(j=1,2)$. Especially, in Case II,
\begin{equation*}
s_1 \neq s_2, \quad s_1\neq \mu_*^{-1}, \quad s_2\neq \mu_*^{-1}.
\end{equation*}
\item\label{lemm:roots_s3}
If $\mu_*$, $\nu_*$, and $\kappa_*$ satisfy the condition of Case III, then $\mu_*\neq \nu_*$.
Especially, in Case III,
\begin{align*}
(s_1,s_2)&=(\mu_*^{-1}, \nu_*^{-1})  \quad \text{when $\mu_*>\nu_*$}; \\
(s_1,s_2)&=(\nu_*^{-1},\mu_*^{-1}) \quad \text{when $\mu_*<\nu_*$}.
\end{align*}
\item\label{lemm:roots_s4}
If $\mu_*$, $\nu_*$, and $\kappa_*$ satisfy the condition of Case IV, then $\mu_*\neq \nu_*$.
Especially, in Case IV,
\begin{equation*}
s_1 = s_2=\frac{\mu_*+\nu_*}{2\kappa_*}, \quad s_2\neq \mu_*^{-1}.
\end{equation*}
\item\label{lemm:roots_s5}
If $\mu_*$, $\nu_*$, and $\kappa_*$ satisfy the condition of Case V, then $\mu_*=\nu_*$. 
Especially, in Case V,
\begin{equation*}
s_1 = s_2=\frac{\mu_*+\nu_*}{2\kappa_*}=\mu_*^{-1}.
\end{equation*}
\end{enumerate}
\end{lemm}

Let us define a positive number $\vps_*$ by
\begin{equation*}
\vps_*=\arg s_2 = \arg s_+\in [0,\pi/2).
\end{equation*}
One then sets
\begin{equation}\label{defi:t}
t_j=\sqrt{|\xi'|^2+ s_j\lambda} \quad \text{for $j=1,2$,}
\end{equation}
where $(\xi',\lambda)\in\BR^{N-1}\times \Sigma_\vps$ for $\vps\in(\vps_*,\pi/2)$.
Here we have chosen a branch cut along the negative real axis and a branch of the square root so that
$\Re\sqrt z>0$ for $z\in\BC\setminus(-\infty,0]$.
In addition, we set
\begin{equation}\label{defi:omega}
\omega_\lambda=\sqrt{|\xi'|^2+\mu_*^{-1}\lambda},
\end{equation}
where $(\xi',\lambda)\in\BR^{N-1}\times\Sigma_\vps$ for $\vps\in(0,\pi/2)$.
By Lemma \ref{lemm:roots_s}, we have

\begin{lemm}\label{lemm:roots_t}
Let $\mu_*$, $\nu_*$, and $\kappa_*$ be positive constants, and let $\xi'\in\BR^{N-1}$ and
$\lambda\in\Sigma_\vps$ for $\vps\in(\vps_*,\pi/2)$.
Then the following assertions hold true:
\begin{enumerate}[$(1)$]
\item\label{coro:roots_s12}
If $\mu_*$, $\nu_*$, and $\kappa_*$ satisfy the conditions of Case I or Case II, then
\begin{equation*}
t_1 \neq t_2, \quad t_1\neq \omega_\lambda, \quad t_2\neq \omega_\lambda.
\end{equation*}
\item\label{coro:roots_s3}
If $\mu_*$, $\nu_*$, and $\kappa_*$ satisfy the condition of Case III, then
\begin{align*}
t_1&=\omega_\lambda, \quad t_2=\sqrt{|\xi'|^2+\nu_*^{-1}\lambda} \quad \text{when $\mu_*>\nu_*$}; \\
t_2&=\omega_\lambda, \quad t_1=\sqrt{|\xi'|^2+\nu_*^{-1}\lambda} \quad \text{when $\mu_*<\nu_*$}.
\end{align*}
\item\label{coro:roots_s4}
If $\mu_*$, $\nu_*$, and $\kappa_*$ satisfy the condition of Case IV, then
\begin{equation*}
t_1 = t_2=\sqrt{|\xi'|^2+\left(\frac{\mu_*+\nu_*}{2\kappa_*}\right)\lambda}, \quad t_2\neq \omega_\lambda. 
\end{equation*}
\item\label{coro:roots_s5}
If $\mu_*$, $\nu_*$, and $\kappa_*$ satisfy the condition of Case V, then
\begin{equation*}
t_1 = t_2=\omega_\lambda.
\end{equation*}
\end{enumerate}
\end{lemm}

Let us define a polynomial $P_\lambda(t)$ by
\begin{equation}\label{eq:chara}
P_\lambda(t) 
=\lambda^2 -\lambda(\mu_*+\nu_*)(t^2-|\xi'|^2)
+\kappa_*(t^2-|\xi'|^2)^2,
\end{equation}
where $(\xi',\lambda)\in\times\BR^{N-1}\times \Sigma_\vps$ for $\vps\in(\vps_*,\pi/2)$.
Since 
\begin{align*}
P_\lambda(t)
&=\kappa_*\lambda^2\left\{\frac{1}{\kappa_*}-\left(\frac{\mu_*+\nu_*}{\kappa_*}\right)
\left(\frac{t^2-|\xi'|^2}{\lambda}\right)+\left(\frac{t^2-|\xi'|^2}{\lambda}\right)^2\right\} \\
&=\kappa_*\lambda^2 \CP\left(\frac{t^2-|\xi'|^2}{\lambda}\right),
\end{align*}
the four roots of $P_\lambda(t)$ are given by $\pm t_1$ and $\pm t_2$.
Hence, one has

\begin{lemm}\label{lemm:roots_P}
Let $\mu_*$, $\nu_*$, and $\kappa_*$ be positive constants, and let $\xi'\in\BR^{N-1}$ and
$\lambda\in\Sigma_\vps$ for $\vps\in(\vps_*,\pi/2)$.
Then the four roots of $P_\lambda(t)$ are given by $\pm t_1$ and $\pm t_2$.
Especially, $\Re t_1>0$ and $\Re t_2>0$.
\end{lemm}

Similarly to \cite{Saito1}, we can prove

\begin{lemm}\label{lemm:symbol0}
Assume that $\mu_*$, $\nu_*$, and $\kappa_*$ are positive constants.
Let $\vps_1\in(\vps_*,\pi/2)$ and $\vps_2\in(0,\pi/2)$.
Then the following assertions holds true:
\begin{enumerate}[$(1)$]
\item
There exists a positive constant $C_{\vps_1,\mu_*,\nu_*,\kappa_*}$ such that
$$
\Re t_j \geq C_{\vps_1,\mu_*,\nu_*,\kappa_*}(|\lambda|^{1/2}+|\xi'|) \quad (j=1,2)
$$
for any $(\xi',\lambda)\in\BR^{N-1}\times \Sigma_{\vps_1}$.
\item
There exists a positive constant $C_{\vps_2,\mu_*}$ such that
$$
\Re\omega_\lambda \geq C_{\vps_2, \mu_*}(|\lambda|^{1/2}+|\xi'|)
$$
for any $(\xi',\lambda)\in\BR^{N-1}\times \Sigma_{\vps_2}$.
\item
Let $r\in\BR$ and $a> 0$. Then
\begin{align*}
t_1^r, t_2^r, (t_1+\omega_\lambda)^r, (t_2+\omega_\lambda)^r &\in\BBM_{r,1}(\Sigma_{\vps_1}), \\
\omega_\lambda^r&\in\BBM_{r,1}(\Sigma_{\vps_2}), \\
(|\xi'|^2+a\lambda)^r&\in\BBM_{2r,1}(\Sigma_{\vps_2}).
\end{align*}
\end{enumerate}
\end{lemm}

\subsection{Technical lemmas}\label{subsec2-3}
Let us introduce the following notation: for $x_N>0$,
\begin{equation}\label{stkernel:1}
\CM_0(x_N)=\frac{e^{-t_2 x_N}-e^{-t_1 x_N}}{t_2-t_1}, \quad 
\CM_j(x_N)=\frac{e^{-t_j x_N}-e^{-\omega_\lambda x_N}}{t_2-t_1} \quad (j=1,2)
\end{equation}
when $t_2\neq t_1$; 
\begin{equation}\label{stkernel:2}
\CM(x_N)=\frac{e^{-t_2 x_N}-e^{-\omega_\lambda x_N}}{t_2-\omega_\lambda}
\end{equation}
when $t_2\neq \omega_\lambda$.
In addition, we define the partial Fourier transform with respect to $x'=(x_1,\dots,x_{N-1})$ and its inverse transform by
\begin{align}
&\wht u = \wht u(x_N) =\wht u(\xi',x_N) =\int_{\BR^{N-1}}e^{-ix'\cdot\xi'} u(x',x_N)\intd x', \label{defi:PFT} \\
&\CF_{\xi'}^{-1}[v(\xi',x_N)](x') = \frac{1}{(2\pi)^{N-1}}\int_{\BR^{N-1}}e^{ix'\cdot\xi'} v(\xi',x_N)\intd \xi', \label{defi:IPFT}
\end{align}
respectively. 
Similarly to \cite{Saito1}, one then has the following three lemmas:

\begin{lemm}\label{lemm:multiplier1}
Let $q\in(1,\infty)$, and let $t_j$ $(j=1,2)$ and $\omega_\lambda$ be respectively given by \eqref{defi:t} and \eqref{defi:omega}
for positive constants $\mu_*$, $\nu_*$, and $\kappa_*$. 
Assume
$$
k(\xi',\lambda)\in\BBM_{-1,1}(\BC_+), \quad l(\xi',\lambda)\in\BBM_{0,1}(\BC_+),
$$
and set for $x=(x',x_N)\in\BR_+^N$
\begin{align*}
[K_0(\lambda) f](x)&=\CF_{\xi'}^{-1}\left[k(\xi',\lambda) e^{-\omega_\lambda x_N}\wht f(\xi',0)\right](x'), \\
[L_0(\lambda) f](x)&=\CF_{\xi'}^{-1}\left[l(\xi',\lambda) e^{-\omega_\lambda x_N}\wht f(\xi',0)\right](x'), \\
[K_j(\lambda) f](x)&=\CF_{\xi'}^{-1}\left[k(\xi',\lambda) e^{-t_j x_N}\wht f(\xi',0)\right](x') \quad (j=1,2), \\
[L_j(\lambda) f](x)&=\CF_{\xi'}^{-1}\left[l(\xi',\lambda) e^{-t_j x_N}\wht f(\xi',0)\right](x') \quad (j=1,2), 
\end{align*}
with $\lambda\in\BC_+$ and $ f\in H_q^2(\BR_+^N)$. Then the following assertions hold true:
\begin{enumerate}[$(1)$]
\item
For $j=0,1,2$ and $\lambda\in\BC_+$, there are operators $\wtd K_j(\lambda)$, with
$$
\wtd K_j(\lambda)\in\hlm(\BC_+,\CL(L_q(\BR_+^N)^{N^2+N+1},H_q^3(\BR_+^N))),
$$
such that for any $f\in H_q^2(\BR_+^N)$
$$
K_j(\lambda) f=\wtd K_j(\lambda)(\nabla^2 f,\lambda^{1/2}\nabla f,\lambda f).
$$
In addition, for $j=0,1,2$ and $n=0,1$,
$$
\CR_{\CL(L_q(\BR_+^N)^{N^2+N+1},\FA_q^0(\BR_+^N))}
\left(\left\{\left.\left(\lambda \frac{d}{d\lambda}\right)^n\left(\CS_\lambda^0\wtd K_j(\lambda)\right)\right|\lambda\in\BC_+\right\}\right) \leq C,
$$
with some positive constant $C$ depending solely on $N$, $q$, $\mu_*$, $\nu_*$, and $\kappa_*$. 
Here, $\FA_q^0(\BR_+^N)$ and $\CS_\lambda^0$ are given in \eqref{defi:AB} for $G=\BR_+^N$.
\item
For $j=0,1,2$ and $\lambda\in\BC_+$, there are operators $\wtd L_j(\lambda)$, with
$$
\wtd L_j(\lambda)\in\hlm(\BC_+,\CL(L_q(\BR_+^N)^{N^2+N+1},H_q^2(\BR_+^N))),
$$
such that for any $f\in H_q^2(\BR_+^N)$
$$
L_j(\lambda)f =\wtd L_j(\lambda)(\nabla^2 f,\lambda^{1/2}\nabla f,\lambda f).
$$
In addition, for $j=0,1,2$ and $n=0,1$,
$$
\CR_{\CL(L_q(\BR_+^N)^{N^2+N+1})}
\left(\left\{\left.\left(\lambda \frac{d}{d\lambda}\right)^n\left(\CT_\lambda\wtd L_j(\lambda)\right)\right|\lambda\in\BC_+\right\}\right) \leq C,
$$
with some positive constant $C$ depending solely on $N$, $q$, $\mu_*$, $\nu_*$, and $\kappa_*$.
Here, $\CT_\lambda$ are given in \eqref{defi:AB}. 
\end{enumerate}
\end{lemm}

\begin{lemm}\label{lemm:multiplier2}
Let $q\in(1,\infty)$, and let $t_j$ $(j=1,2)$ and $\omega_\lambda$ be respectively given by \eqref{defi:t} and \eqref{defi:omega}
for positive constants $\mu_*$, $\nu_*$, and $\kappa_*$ satisfying the conditions of Case I or Case II. 
Assume
$$
m_0(\xi',\lambda)\in\BBM_{0,1}(\BC_+), \quad m_1(\xi',\lambda),m_2(\xi',\lambda)\in\BBM_{1,1}(\BC_+),
$$
and set for $x=(x',x_N)\in\BR_+^N$
$$
[M_k(\lambda) f](x)=\CF_{\xi'}^{-1}\left[m_k(\xi',\lambda) \CM_k(x_N)\wht f(\xi',0)\right](x'), 
$$
with $\lambda\in\BC_+$ and $ f\in H_q^2(\BR_+^N)$. Then the following assertions hold true:
\begin{enumerate}[$(1)$]
\item
For $\lambda\in\BC_+$, there is an operator $\wtd M_0(\lambda)$, with
$$
\wtd M_0(\lambda)\in\hlm(\BC_+,L_q(\BR_+^N)^{N^2+N+1},H_q^3(\BR_+^N)),
$$
such that for any $f\in H_q^2(\BR_+^N)$
$$
M_0(\lambda) f=\wtd M_0(\lambda)(\nabla^2 f,\lambda^{1/2}\nabla f,\lambda f).
$$
In addition, for $n=0,1$,
$$
\CR_{\CL(L_q(\BR_+^N)^{N^2+N+1},\FA_q^0(\BR_+^N))}
\left(\left\{\left.\left(\lambda \frac{d}{d\lambda}\right)^n\left(\CS_\lambda^0\wtd M_0(\lambda)\right)\right|\lambda\in\BC_+\right\}\right) \leq C,
$$
with some positive constant $C$ depending solely on $N$, $q$, $\mu_*$, $\nu_*$, and $\kappa_*$.
\item
For $j=1,2$ and $\lambda\in\BC_+$, there are operators $\wtd M_j(\lambda)$, with
$$
\wtd M_j(\lambda)\in\hlm(\BC_+,\CL(L_q(\BR_+^N)^{N^2+N+1},H_q^2(\BR_+^N))),
$$
such that for any $f\in H_q^2(\BR_+^N)$
$$
M_j(\lambda)f = \wtd M_j(\lambda)(\nabla^2 f,\lambda^{1/2}\nabla f,\lambda f).
$$
In addition, for $j=1,2$ and $n=0,1$,
$$
\CR_{\CL(L_q(\BR_+^N)^{N^2+N+1})}
\left(\left\{\left.\left(\lambda \frac{d}{d\lambda}\right)^n\left(\CT_\lambda\wtd M_j(\lambda)\right)\right|\lambda\in\BC_+\right\}\right) \leq C,
$$
with some positive constant $C$ depending solely on $N$, $q$, $\mu_*$, $\nu_*$, and $\kappa_*$.
\end{enumerate}
\end{lemm}

\begin{lemm}\label{lemm:multiplier3}
Let $q\in(1,\infty)$, and let $t_j$ $(j=1,2)$ and $\omega_\lambda$ be respectively given by \eqref{defi:t} and \eqref{defi:omega}
for positive constants $\mu_*$, $\nu_*$, and $\kappa_*$ satisfying one of the following conditions:
\begin{itemize}
\item
the condition of Case III and $\mu_*>\nu_*$;
\item
the condition of Case IV.
\end{itemize}
Assume
$$
u(\xi',\lambda)\in\BBM_{0,1}(\BC_+), \quad v(\xi',\lambda)\in\BBM_{1,1}(\BC_+),
$$
and set for $x=(x',x_N)\in\BR_+^N$
\begin{align*}
[U(\lambda) f](x)&=\CF_{\xi'}^{-1}\left[u(\xi',\lambda) \CM(x_N)\wht f(\xi',0)\right](x'),  \\
[V(\lambda) f](x)&=\CF_{\xi'}^{-1}\left[v(\xi',\lambda) \CM(x_N)\wht f(\xi',0)\right](x'), 
\end{align*}
with $\lambda\in\BC_+$ and $ f\in H_q^2(\BR_+^N)$. Then the following assertions hold true:
\begin{enumerate}[$(1)$]
\item
For $\lambda\in\BC_+$, there is an operator $\wtd U(\lambda)$, with
$$
\wtd U(\lambda)\in\hlm(\BC_+,L_q(\BR_+^N)^{N^2+N+1},H_q^3(\BR_+^N)),
$$
such that for any $f\in H_q^2(\BR_+^N)$
$$
U(\lambda) f=\wtd U(\lambda)(\nabla^2 f,\lambda^{1/2}\nabla f,\lambda f).
$$
In addition, for $n=0,1$,
$$
\CR_{\CL(L_q(\BR_+^N)^{N^2+N+1},\FA_q^0(\BR_+^N))}
\left(\left\{\left.\left(\lambda \frac{d}{d\lambda}\right)^n\left(\CS_\lambda^0\wtd U(\lambda)\right)\right|\lambda\in\BC_+\right\}\right) \leq C,
$$
with some positive constant $C$ depending solely on $N$, $q$, $\mu_*$, $\nu_*$, and $\kappa_*$.
\item
For $\lambda\in\BC_+$, there is an operator $\wtd V(\lambda)$, with
$$
\wtd V(\lambda)\in\hlm(\BC_+,\CL(L_q(\BR_+^N)^{N^2+N+1},H_q^2(\BR_+^N))),
$$
such that for any $f\in H_q^2(\BR_+^N)$
$$
V(\lambda)f = \wtd V(\lambda)(\nabla^2 f,\lambda^{1/2}\nabla f,\lambda f).
$$
In addition, for $n=0,1$,
$$
\CR_{\CL(L_q(\BR_+^N)^{N^2+N+1})}
\left(\left\{\left.\left(\lambda \frac{d}{d\lambda}\right)^n\left(\CT_\lambda\wtd V(\lambda)\right)\right|\lambda\in\BC_+\right\}\right) \leq C,
$$
with some positive constant $C$ depending solely on $N$, $q$, $\mu_*$, $\nu_*$, and $\kappa_*$.
\end{enumerate}
\end{lemm}

In the last part of this subsection,
we introduce fundamental properties of the $\CR$-boundedness as follows (cf. e.g. \cite[Proposition 3.4]{DHP03}):

\begin{lemm}\label{prop:R}
Let $X$, $Y$, and $Z$ be Banach spaces.
Then the following assertions hold true:
\begin{enumerate}[$(1)$]
\item
Let $\CT$ and $\CS$ be $\CR$-bounded families on $\CL(X,Y)$.
Then $\CT+\CS=\{T+S \mid T\in\CT, S\in\CS\}$ is also $\CR$-bounded on $\CL(X,Y)$,
and 
$$\CR_{\CL(X,Y)}(\CT+\CS)\leq \CR_{\CL(X,Y)}(\CT)+\CR_{\CL(X,Y)}(\CS).$$
\item
Let $\CT$ and $\CS$ be $\CR$-bounded families
on $\CL(X,Y)$ and on $\CL(Y,Z)$, respectively.
Then $\CS\CT=\{ST \mid S\in\CS, T\in\CT\}$ is also $\CR$-bounded on $\CL(X,Z)$,
and 
$$\CR_{\CL(X,Z)}(\CS\CT)\leq \CR_{\CL(X,Y)}(\CT)\CR_{\CL(Y,Z)}(\CS).$$
\end{enumerate}
\end{lemm}

\subsection{A reduced system of \eqref{eq1:gen-res}}\label{subsec2-4}
This subsection reduces the system \eqref{eq1:gen-res} to $(d,\Bf)=(0,0)$.
To this end, we start with the following problem on the whole space:
\begin{equation}\label{eq:whole}
\left\{\begin{aligned}
\lambda\rho + \dv\Bu &= d && \text{in $\BR^N$,} \\
\lambda\Bu - \mu_*\Delta\Bu -\nu_*\nabla\dv\Bu
-\kappa_*\nabla\Delta\rho &= \Bf && \text{in $\BR^N$.}
\end{aligned}\right.
\end{equation}
Recall that $\CX_q^1(\BR^N)$, $\FX_q^1(\BR^N)$, $\CF_\lambda^1$,
$\FA_q^0(\BR^N)$, $\CS_\lambda^0$, $\FB_q(\BR^N)$, and $\CT_\lambda$
are symbols defined in Subsection \ref{subsec1-2} for $G=\BR^N$.
Concerning the system \eqref{eq:whole}, one has

\begin{theo}\label{theo:whole}
Let $q\in(1,\infty)$ and assume that $\mu_*$, $\nu_*$, and $\kappa_*$ are positive constants.
Then, for any $\lambda\in\BC_+$, there exist operators $\CA^1(\lambda)$ and $\CB^1(\lambda)$, with
\begin{align*}
\CA^1(\lambda)&\in\hlm(\BC_+,\CL(\FX_q^1(\BR^N),H_q^3(\BR^N))), \\
\CB^1(\lambda)&\in\hlm(\BC_+,\CL(\FX_q^2(\BR^N),H_q^2(\BR^N)^N)),
\end{align*}
such that, for any $\BF^1=(d,\Bf)\in\CX_q^1(\BR^N)$,
$(\rho,\Bu)=(\CA^1(\lambda)\CF_\lambda^1\BF^1,\CB^1(\lambda)\CF_\lambda^1\BF^1)$ is a unique solution to the system \eqref{eq:whole}.
In addition, for $n=0,1$,
\begin{align*}
\CR_{\CL(\FX_q^1(\BR^N),\FA_q^0(\BR^N))}\left(\left\{\left.\left(\lambda\frac{d}{d\lambda}\right)^n
\left(\CS_\lambda^0\CA^1(\lambda)\right) \right| \lambda\in\BC_+\right\}\right)
&\leq C_{N,q,\mu_*,\nu_*,\kappa_*}, \\
\CR_{\CL(\FX_q^1(\BR^N),\FB_q(\BR^N))}\left(\left\{\left.\left(\lambda\frac{d}{d\lambda}\right)^n
\left(\CT_\lambda\CB^1(\lambda)\right) \right| \lambda\in\BC_+\right\}\right)
&\leq C_{N,q,\mu_*,\nu_*,\kappa_*}, 
\end{align*}
with a positive constant $C_{N,q,\mu_*,\nu_*,\kappa_*}$.
\end{theo}

\begin{proof}
The proof is similar to \cite{Saito1}, so that the detailed proof may be omitted.
\end{proof}

One now reduces the system \eqref{eq1:gen-res} to $(d,\Bf)=(0,0)$ by using Theorem \ref{theo:whole}.
For functions $f=f(x)$ with $x=(x',x_N)\in\BR_+^N$,
let $E^e f$ and $E^o f$ be respectively the even extension of $f$ and the odd extension of $f$, i.e.
\begin{align*}
E^e f&=(E^e f)(x)=
\left\{\begin{aligned}
&f(x',x_N)  && (x_N>0), \\
&f(x',-x_N) && (x_N<0),
\end{aligned}\right. \\
E^o f&=(E^o f)(x)=
\left\{\begin{aligned}
&f(x',x_N)  && (x_N>0), \\
&-f(x',-x_N) && (x_N<0).
\end{aligned}\right.
\end{align*}
In addition, we set for $\Bf=(f_1,\dots,f_N)^\SST$ defined on $\BR_+^N$
\begin{equation*}
\BE\Bf=(E^e f_1,\dots, E^e f_{N-1},E^o f_N)^\SST.
\end{equation*}
Note that  $E^e\in \CL(H_q^1(\BR_+^N),H_q^1(\BR^N))$ and $\BE\in\CL(L_q(\BR_+^N)^N,L_q(\BR^N)^N)$.

Let $\CA^1(\lambda)$ and $\CB^1(\lambda)$ be the operators constructed in Theorem \ref{theo:whole},
and set for $(d,\Bf)\in H_q^1(\BR_+^N)\times L_q(\BR_+^N)^N$ 
\begin{equation*}
R=\CA^1(\lambda)\CF_\lambda^1(E^e d,\BE\Bf), \quad
\BU=\CB^1(\lambda)\CF_\lambda^1(E^e d,\BE\Bf).
\end{equation*}
Furthermore, let us define $S=S(x',x_N)$ and $\BV=\BV(x',x_N)$ as
\begin{equation*}
S = R(x',-x_N), \quad
\BV =(U_1(x',-x_N),\dots,U_{N-1}(x',-x_N),-U_N(x',-x_N))^\SST.
\end{equation*}
Here and subsequently,
$U_J$ and $V_J$ denote for $J=1,\dots,N$ the $J$th component of $\BU$
and the $J$th component of $\BV$, respectively.
It then holds that 
\begin{align*}
&(\lambda S +\dv\BV)(x',x_N) \\
&=(\lambda R+\dv\BU)(x',-x_N) = (E^e d)(x',-x_N)=(E^e d)(x',x_N).
\end{align*}
Analogously, for $j=1,\dots,N-1$,
\begin{align*}
(\lambda V_j-\mu_*\Delta V_j-\nu_*\pd_j\dv\BV-\kappa_*\pd_j\Delta S)(x',x_N) 
&=(E^e f_j)(x',x_N), \\
(\lambda V_N-\mu_*\Delta V_N-\nu_*\pd_N\dv\BV-\kappa_*\pd_N\Delta S)(x',x_N) 
&=(E^o f_N)(x',x_N).
\end{align*}
The uniqueness of solutions of \eqref{eq:whole} then implies $\BU(x',x_N)=\BV(x',x_N)$.
Setting $x_N=0$ in this equality yields $U_N(x',0)=0$.

Let $\rho=R+\wtd\rho$ and $\Bu=\BU+\wtd\Bu$ in \eqref{eq1:gen-res}.
One then achieves, by $U_N=0$ on $\BR_0^N$ mentioned above, the following reduced system:
\begin{equation}\label{eq1:half-red}
\left\{\begin{aligned}
\lambda\wtd\rho+\dv\wtd\Bu&=0 && \text{in $\BR_+^N$,} \\
\lambda\wtd\Bu-\mu_*\Delta\wtd\Bu-\nu_*\nabla\dv\wtd\Bu-\kappa_*\nabla\Delta\wtd\rho&=0 && \text{in $\BR_+^N$,} \\
\Bn\cdot\nabla\wtd\rho=\wtd g, \quad \wtd u_j=\wtd h_j, \quad \wtd u_N&=0 && \text{on $\BR_0^N$,} 
\end{aligned}\right.
\end{equation}
for $j=1,\dots,N-1$ and 
\begin{align}\label{defi:tilde-g}
&\wtd g =g-\Bn\cdot\nabla\CA^1(\lambda)\CF_\lambda^1(E^e d,\BE\Bf)=g+\pd_N\CA^1(\lambda)\CF_\lambda^1(E^e d,\BE\Bf), \\
&\wtd h_j =-(\CB^1(\lambda)\CF_\lambda^1(E^e d,\BE\Bf))_j, \notag 
\end{align}
where $(\Bv)_j$ denotes the $j$th component of $\Bv$.

In view of \eqref{eq1:half-red}, let us consider
\begin{equation}\label{eq2:half-red}
\left\{\begin{aligned}
\lambda\rho+\dv\Bu&=0 && \text{in $\BR_+^N$,} \\
\lambda u_J-\mu_*\Delta u_J-\nu_*\pd_J\dv\Bu-\kappa_*\pd_J\Delta\rho&=0 && \text{in $\BR_+^N$,} \\
\Bn\cdot\nabla\rho= g, \quad u_j= h_j, \quad u_N&=0 && \text{on $\BR_0^N$,} 
\end{aligned}\right.
\end{equation}
where $J=1,\dots,N$ and $j=1,\dots,N-1$,
for given $g\in H_q^2(\BR_+^N)$ and $h_j\in H_q^2(\BR_+^N)$.
The main part of the proof of Theorem \ref{theo:main} is to show
\begin{theo}\label{theo:half-red}
Let $q\in(1,\infty)$ and set
\begin{align*}
&\CY_q(\BR_+^N)=H_q^2(\BR_+^N)^N, \quad \FY_q(\BR_+^N)=L_q(\BR_+^N)^{N(N^2+N+1)}, \\
&\CG_\lambda(g, h_1,\dots, h_{N-1})=(\CT_\lambda  g,\CT_\lambda  h_1,\dots,\CT_\lambda h_{N-1}),
\end{align*}
for $(g, h_1,\dots, h_{N-1})\in\CY_q(\BR_+^N)$.
Assume that $\mu_*$, $\nu_*$, and $\kappa_*$ are positive constants.
Then, for any $\lambda\in\BC_+$, there are operators $\CA^2(\lambda)$ and $\CB^2(\lambda)$, with
\begin{align*}
\CA^2(\lambda) & \in \hlm(\BC_+,\CL(\FY_q(\BR_+^N),H_q^3(\BR_+^N))), \\
\CB^2(\lambda) & \in \hlm(\BC_+,\CL(\FY_q(\BR_+^N),H_q^2(\BR_+^N)^N)),
\end{align*}
such that, for any $\BG=(g,h_1,\dots,h_{N-1})\in\CY_q(\BR_+^N)$,
$$
(\rho,\Bu)=(\CA^2(\lambda)\CF_\lambda^2\BG,\CB^2(\lambda)\CF_\lambda^2\BG)
$$
is a unique solution to the system \eqref{eq2:half-red}.
In addition, for $n=0,1$,
\begin{align*}
\CR_{\CL(\FY_q(\BR_+^N),\FA_q^0(\BR_+^N))}
\left(\left\{\left.\left(\lambda\frac{d}{d\lambda}\right)^n
\left(\CS_\lambda^0\CA^2(\lambda)\right) \right| \lambda\in\BC_+\right\}\right)
&\leq C_{N,q,\mu_*,\nu_*,\kappa_*}, \\
\CR_{\CL(\FY_q(\BR_+^N),\FB_q(\BR_+^N))}
\left(\left\{\left.\left(\lambda\frac{d}{d\lambda}\right)^n\left(\CT_\lambda\CB^2(\lambda)\right) \right| \lambda\in\BC_+\right\}\right)
&\leq C_{N,q,\mu_*,\nu_*,\kappa_*}, 
\end{align*}
with a positive constant $C_{N,q,\mu_*,\nu_*,\kappa_*}$.
\end{theo}

In the remaining part of this subsection,
we prove Theorem \ref{theo:main}, assuming that Theorem \ref{theo:half-red} holds.

Let $\BF=(d,\Bf,g)\in \CX_q(\BR_+^N)$.
Noting $\nabla E^e d = \BE\nabla d$, we see $\CF_\lambda^1(E^e d,\BE\Bf) = (\BE\nabla d,E^e(\lambda^{1/2}d),\BE\Bf)$.
In view of this relation and \eqref{defi:tilde-g}, one sets,
for\footnote{$H_1$, $H_2$, $H_3$, $H_4$, $H_5$, and $H_6$ are variables 
corresponding to $\nabla d$, $\lambda^{1/2}d$, $\Bf$, $\nabla^2 g$, $\lambda^{1/2} \nabla g$, and $\lambda g$, respectively.
} $\BH=(H_1,\dots,H_6)\in\FX_q^2(\BR_+^N)$ and $\CZ\in\{\CA,\CB\}$,
\begin{align*}
\CZ(\lambda)\BH 
&=\CZ^1(\lambda)(\BE H_1,E^e H_2, \BE H_3) \\
& + \CZ^2(\lambda)\Big(H_4+\nabla^2\pd_N\CA^1(\lambda)(\BE H_1,E^e H_2, \BE H_3), \\
&H_5+\lambda^{1/2}\nabla\pd_N\CA^1(\lambda)(\BE H_1,E^e H_2, \BE H_3), \\
&H_6+\lambda\pd_N\CA^1(\lambda)(\BE H_1,E^e H_2, \BE H_3), \\
&-\CT_\lambda(\CB^1(\lambda)(\BE H_1,E^e H_2, \BE H_3))_1,\dots, \\
&-\CT_\lambda(\CB^1(\lambda)(\BE H_1,E^e H_2, \BE H_3))_{N-1}\Big).
\end{align*}
It is then clear that 
$(\rho,\Bu)=(\CA(\lambda)\CF_\lambda^2\BF,\CB(\lambda)\CF_\lambda^2\BF)$ is a solution to \eqref{eq1:gen-res},
and also $\CA(\lambda)$ and $\CB(\lambda)$ satisfy the estimates required in Theorem \ref{theo:main} 
by Lemma \ref{prop:R} and Theorems \ref{theo:whole} and \ref{theo:half-red}.
The uniqueness of solutions of \eqref{eq1:gen-res} can be proved similarly to \cite{Saito1}.
This completes the proof of Theorem \ref{theo:main}.

\begin{rema}
The following sections are devoted to the proof of Theorem \ref{theo:half-red},
and the proof is divided into Cases I, II, III, IV, and V.
\end{rema}

\section{Proof of Theorem \ref{theo:half-red} for Cases I and II}\label{sec3}
This section proves Theorem \ref{theo:half-red} for Cases I and II.
Throughout this section, we assume that $\mu_*$, $\nu_*$, and $\kappa_*$ are positive constants
satisfying the conditions of Case I or Case II.
One then recalls Lemma \ref{lemm:roots_t} \eqref{coro:roots_s12}, i.e.
\begin{equation*}\label{condi:I-II}
t_1\neq t_2, \quad t_1 \neq \omega_\lambda, \quad t_2\neq \omega_\lambda,
\end{equation*}
which are often used in the following computations. 
Let $J=1,\dots,N$ and $j=1,\dots,N-1$ in this section.

\subsection{Solution formulas}\label{subsec3-1}

Set $\vph=\dv\Bu$.
Applying the partial Fourier transform given by \eqref{defi:PFT} to the system \eqref{eq2:half-red} yields
the ordinary differential equations: 
\begin{align}
\lambda \wht \rho +\wht \vph &=0, \text{ $x_N>0$,} \label{eq:1} \\ 
\lambda \wht u_j-\mu_*(\pd_N^2-|\xi'|^2) \wht u_j-\nu_* i\xi_j \wht\vph
-\kappa_*i\xi_j(\pd_N^2-|\xi'|^2)\wht\rho&=0, \text{ $x_N>0$,} \label{eq:2} \\
\lambda \wht u_N-\mu_*(\pd_N^2-|\xi'|^2)\wht u_N-\nu_*\pd_N\wht\vph
-\kappa_*\pd_N(\pd_N^2-|\xi'|^2)\wht\rho&=0,  \text{ $x_N>0$,} \label{eq:3}
\end{align}
with the boundary conditions: 
\begin{align}
&\pd_N\wht \rho(0)=-\wht g(0), \label{eq:5} \\
&\wht u_j(0) =\wht h_j(0), \quad \wht u_N(0)=0. \label{eq:4} 
\end{align}
One inserts \eqref{eq:1} into \eqref{eq:2}, \eqref{eq:3}, and \eqref{eq:5}, and then 
\begin{align}
\lambda^2\wht u_j-\lambda\mu_*(\pd_N^2-|\xi'|^2)\wht u_j
-i\xi_j\{\lambda\nu_* -\kappa_*(\pd_N^2-|\xi'|^2)\}\wht \vph &=0,\text{ $x_N>0$,} \label{eq:41} \\
\lambda^2\wht u_N-\lambda\mu_*(\pd_N^2-|\xi'|^2)\wht u_N
-\pd_N\{\lambda\nu_* -\kappa_*(\pd_N^2-|\xi'|^2)\}\wht \vph &=0, \text{ $x_N>0$,} \label{eq:42} \\
\pd_N\wht \vph(0) &=\lambda\wht g(0). \label{eq:44}
\end{align}
Multiplying \eqref{eq:41} by $i\xi_j$ and applying $\pd_N$ to \eqref{eq:42},
we sum the resultant equations in order to obtain
\begin{equation*}
\lambda^2\wht \vph-\lambda(\mu_*+\nu_*) (\pd_N^2-|\xi'|^2)\wht \vph 
+\kappa_*(\pd_N^2-|\xi'|^2)^2\wht \vph=0, \text{ $x_N>0$,}
\end{equation*}
where we have used the fact that $\wht\vph=\sum_{j=1}^{N-1}i\xi_j\wht u_j+\pd_N\wht u_N$.
By $P_\lambda(t)$ given in \eqref{eq:chara}, the last equation is written as
\begin{equation}\label{170821_1}
P_\lambda(\pd_N)\wht \vph =0.
\end{equation}
On the other hand, \eqref{eq:41} and \eqref{eq:42} are respectively equivalent to
\begin{align}
\mu_*\lambda(\pd_N^2-\omega_\lambda^2)\wht u_j
+i\xi_j\{\nu_*\lambda-\kappa_*(\pd_N^2-|\xi'|^2)\}\wht \vph=0, \text{ $x_N>0$,} \label{1023_18:eq5} \\
\mu_*\lambda(\pd_N^2-\omega_\lambda^2)\wht u_N
+\pd_N\{\nu_*\lambda-\kappa_*(\pd_N^2-|\xi'|^2)\}\wht \vph=0, \text{ $x_N>0$.} \label{1103_18:eq5}
\end{align}
Applying $P_\lambda(\pd_N)$ to \eqref{1023_18:eq5} and \eqref{1103_18:eq5} then furnishes by \eqref{170821_1}
\begin{equation}\label{160827_1}
(\pd_N^2-\omega_\lambda^2)P_\lambda(\pd_N) \wht u_J =0. 
\end{equation}

\begin{rema}
In what follows, we solve equations \eqref{170821_1}-\eqref{160827_1} with boundary conditions \eqref{eq:4} and \eqref{eq:44}
with respect to  $\wht u_J$ and $\wht\vph$ under the following constraint:
\begin{equation}\label{eq:div-vph}
\wht\vph=\sum_{j=1}^{N-1}i\xi_j\wht u_j+\pd_N\wht u_N, \quad x_N>0.
\end{equation}
\end{rema}

In view of \eqref{170821_1}, \eqref{160827_1}, and Lemma \ref{lemm:roots_P}, 
we look for solutions $\wht u_J$ and $\wht\vph$ of the forms:
\begin{align}
\wht u_J 
&= \alpha_J e^{-\omega_\lambda x_N} + \beta_J(e^{-t_1 x_N}-e^{-\omega_\lambda x_N})
+ \gamma_J(e^{-t_2 x_N}-e^{-\omega_\lambda x_N}), \label{eq:9_0} \\
\wht\vph
&=\sigma e^{-t_1 x_N}+\tau e^{-t_2 x_N}. \label{eq:9}
\end{align}
It then holds by \eqref{eq:div-vph} that
\begin{align}
i\xi'\cdot\alpha'-i\xi'\cdot\beta'-i\xi'\cdot\gamma'
-\omega_\lambda\alpha_N+\omega_\lambda\beta_N+\omega_\lambda\gamma_N =0, \label{eq:10} \\
\sigma=i\xi'\cdot\beta'-t_1\beta_N, \quad \tau=i\xi'\cdot\gamma'-t_2\gamma_N, \label{eq:10_2}
\end{align}
where $i\xi'\cdot a'=\sum_{j=1}^{N-1}i\xi_j a_j$ for $a\in \{\alpha,\beta,\gamma\}$.
On the other hand, inserting \eqref{eq:9_0} and \eqref{eq:9} into \eqref{1023_18:eq5} and \eqref{1103_18:eq5} yields
\begin{align*}
\mu_*\lambda\beta_j(t_1^2-\omega_\lambda^2)+i\xi_j \sigma \{\nu_*\lambda-\kappa_*(t_1^2-|\xi'|^2)\}&=0, \\
\mu_*\lambda\gamma_j(t_2^2-\omega_\lambda^2)+i\xi_j\tau\{\nu_*\lambda-\kappa_*(t_2^2-|\xi'|^2)\}&=0,  \\
\mu_*\lambda\beta_N(t_1^2-\omega_\lambda^2)-t_1\sigma \{\nu_*\lambda-\kappa_*(t_1^2-|\xi'|^2)\}&=0,  \\
\mu_*\lambda\gamma_N(t_2^2-\omega_\lambda^2)-t_2\tau\{\nu_*\lambda-\kappa_*(t_2^2-|\xi'|^2)\}&=0. 
\end{align*}
By these relations, we have
\begin{equation*}
\mu_*\lambda(t_1^2-\omega_\lambda^2)\left(\beta_j+\frac{i\xi_j}{t_1}\beta_N\right) = 0,\quad
\mu_*\lambda(t_2^2-\omega_\lambda^2)\left(\gamma_j+\frac{i\xi_j}{t_2}\gamma_N\right)=0.
\end{equation*}
Since $t_1\neq \omega _\lambda$ and $t_2\neq \omega_\lambda$, the last two equations imply 
\begin{equation}\label{eq:51}
\beta_j = -\frac{i\xi_j}{t_1}\beta_N, \quad \gamma_j = -\frac{i\xi_j}{t_2}\gamma_N.
\end{equation}
These relations yield 
\begin{equation}\label{eq:60_2}
i\xi'\cdot\beta'=\frac{|\xi'|^2}{t_1}\beta_N, \quad i\xi' \cdot\gamma'=\frac{|\xi'|^2}{t_2}\gamma_N,
\end{equation}
and thus
\begin{equation}\label{eq:60}
i\xi'\cdot\beta'-t_1\beta_N = -\left(\frac{t_1^2-|\xi'|^2}{t_1}\right)\beta_N, \quad
i\xi'\cdot\gamma'-t_2\gamma_N = -\left(\frac{t_2^2-|\xi'|^2}{t_2}\right)\gamma_N.
\end{equation}

Next, we consider the boundary conditions. By \eqref{eq:4} and \eqref{eq:9_0}, 
\begin{equation}
\alpha_j = \wht h_j(0), \quad \alpha_N =0. \label{eq:55}
\end{equation}
It then holds by the first relation of \eqref{eq:55} that 
\begin{equation}\label{eq:59}
i\xi'\cdot\alpha'=i\xi'\cdot\wht \Bh'(0), \quad \wht\Bh'(0)=(\wht h_1(0),\dots,\wht h_{N-1}(0))^\SST.
\end{equation}
On the other hand, by \eqref{eq:44} and \eqref{eq:9},
$$
-t_1\sigma-t_2\tau = \lambda\wht g(0),
$$
which, combined with \eqref{eq:10_2} and \eqref{eq:60}, furnishes
\begin{equation}\label{eq:57}
(t_1^2-|\xi'|^2)\beta_N+(t_2^2-|\xi'|^2)\gamma_N=\lambda\wht g(0).
\end{equation}

One now derives simultaneous equations with respect to $\beta_N$ and $\gamma_N$.
Inserting \eqref{eq:60_2}, \eqref{eq:59}, and $\alpha_N=0$ of \eqref{eq:55} into \eqref{eq:10} furnishes
\begin{align*}
i\xi'\cdot\wht\Bh'(0)-t_1^{-1}|\xi'|^2\beta_N-t_2^{-1}|\xi'|^2\gamma_N+\omega_\lambda\beta_N+\omega_\lambda\gamma_N=0.
\end{align*}
Hence,
\begin{equation*}
-t_2(t_1\omega_\lambda-|\xi'|^2)\beta_N-t_1(t_2\omega_\lambda-|\xi'|^2)\gamma_N=t_1t_2i\xi'\cdot\wht\Bh'(0),
\end{equation*}
which, combined with \eqref{eq:57}, yields
\begin{align}\label{SEs}
&\BL 
\begin{pmatrix}
\beta_N \\ \gamma_N
\end{pmatrix}
=
\begin{pmatrix}
\lambda \wht g(0) \\
t_1t_2i\xi'\cdot\wht\Bh'(0)
\end{pmatrix}, \\
&\BL
=\begin{pmatrix}
t_1^2-|\xi'|^2 &  t_2^2-|\xi'|^2 \\
-t_2(t_1\omega_\lambda-|\xi'|^2) & -t_1(t_2\omega_\lambda-|\xi'|^2)
\end{pmatrix}.  \notag
\end{align}

Let us solve \eqref{SEs}. 
By direct calculations, 
\begin{align}\label{def:detL}
\det \BL &= 
t_2(t_2^2-|\xi'|^2)(t_1\omega_\lambda-|\xi'|^2)-t_1(t_1^2-|\xi'|^2)(t_2\omega_\lambda-|\xi'|^2) \\
&=(t_2-t_1)\{t_1t_2\omega_\lambda(t_2+t_1)-|\xi'|^2(t_2^2+t_1t_2+t_1^2-|\xi'|^2)\}. \notag
\end{align}
One here proves

\begin{lemm}\label{lemm:detL}
There holds $\det\BL\neq 0$ for any $(\xi',\lambda)\in\BR^{N-1}\times(\overline{\BC_+}\setminus\{0\})$,
where $\overline{\BC_+}=\{z\in\BC \mid \Re z\geq 0\}$.
\end{lemm}

\begin{proof}
The lemma is proved by contradiction.
Suppose that $\det\BL=0$ for some $(\xi',\lambda)\in\BR^{N-1}\times(\overline{\BC_+}\setminus\{0\})$.
Then there is $(\beta_N,\gamma_N)\neq(0,0)$
satisfying \eqref{SEs} with $\wht g(0)=0$ and $\wht \Bh'(0)=0$.
This implies that the equations \eqref{1023_18:eq5} and \eqref{1103_18:eq5},
with $\wht\vph =\sum_{j=1}^{N-1}i\xi_j\wht u_j+\pd_N\wht u_N$ and homogeneous boundary conditions $\wht u_J(0)=0$ and $\pd_N\wht\vph(0)=0$,
admits a non-trivial solution $(\wht u_1,\dots,\wht u_N,\wht\vph)$
sufficiently smooth and decaying exponentially as $x_N\to\infty$.
Let us denote the non-trivial solution by $(u_1,\dots,u_N,\vph)$ for notational simplicity.

Recall that \eqref{1023_18:eq5} and \eqref{1103_18:eq5} are respectively equivalent to \eqref{eq:41} and \eqref{eq:42}.
One multiplies \eqref{eq:41} and \eqref{eq:42} by $\lambda^{-1}$, and then 
\begin{align}
\lambda u_j-\mu_*(\pd_N^2-|\xi'|^2) u_j -\nu_* i\xi_j\vph  
+\kappa_*\lambda^{-1} i\xi_j(\pd_N^2-|\xi'|^2)\vph&=0, \text{ $x_N>0$,}  \label{eq:101} \\
\lambda u_N-\mu_*(\pd_N^2-|\xi'|^2) u_N-\nu_* \pd_N\vph  
 +\kappa_*\lambda^{-1}\pd_N(\pd_N^2-|\xi'|^2) \vph&=0, \text{ $x_N>0$.} \label{eq:102}  
\end{align}

In this proof,  we set 
$(a,b)=\int_0^\infty a(x_N)\,\overline{b(x_N)}\intd x_N$
and $\|a\| =\sqrt{(a,a)}$
for functions $a=a(x_N)$ and $b=b(x_N)$ on $\BR_+$.

{\bf Step 1.}
Multiplying \eqref{eq:101} by $\overline{u_j(x_N)}$ and integrating the resultant formula
with respect to $x_N\in(0,\infty)$ yield
\begin{align*}
\lambda\|u_j\|^2-\mu_*((\pd_N^2 u_j,u_j)-|\xi'|^2\|u_j\|^2)-\nu_*(i\xi_j\vph,u_j)& \\
+\kappa_*\lambda^{-1}(i\xi_j(\pd_N^2-|\xi'|^2)\vph, u_j)&=0,
\end{align*}
which, combined with $(\pd_N^2 u_j,u_j)=-\|\pd_N u_j\|^2$ following from integration by parts with $u_j(0)=0$ 
and combined with the properties:
\begin{equation*}
(i\xi_j\vph,u_j)=-(\vph,i\xi_j u_j), \quad (i\xi_j(\pd_N^2-|\xi'|^2)\vph, u_j)=-((\pd_N^2-|\xi'|^2)\vph,i\xi_j u_j),
\end{equation*}
furnishes that
\begin{align}\label{170821_3}
\lambda\|u_j\|^2+\mu_*(\|\pd_N u_j\|^2+|\xi'|^2\|u_j\|^2) +\nu_*(\vph,i\xi_j u_j) &\\
-\kappa_*\lambda^{-1}((\pd_N^2-|\xi'|^2)\vph,i\xi_j u_j)&=0. \notag
\end{align}
It similarly follows from \eqref{eq:102} that
\begin{align}\label{170821_4}
\lambda\|u_N\|^2+\mu_*(\|\pd_N u_N\|^2+|\xi'|^2\|u_N\|^2)+\nu_*(\vph,\pd_N u_N) & \\
-\kappa_*\lambda^{-1}((\pd_N^2-|\xi'|^2)\vph,\pd_N u_N)&=0. \notag
\end{align}

{\bf Step 2.}
Summing \eqref{170821_3} with respect to $j=1,\dots,N-1$ and \eqref{170821_4}, we have by $\vph=\sum_{j=1}^{N-1}i\xi_j u_j+\pd_N u_N$
\begin{align}\label{170821_5}
\lambda\sum_{J=1}^N\|u_J\|^2+\mu_*\sum_{J=1}^N\left(\|\pd_N u_J\|^2+|\xi'|^2\|u_J\|^2\right)&  \\
+\nu_*\|\vph\|^2-\kappa_*\lambda^{-1}((\pd_N^2-|\xi'|^2)\vph,\vph)&=0. \notag
\end{align}
On the other hand, by integration by parts with $\pd_N\vph(0)=0$,
$$
((\pd_N^2-|\xi'|^2)\vph,\vph)=-(\|\pd_N\vph\|^2+|\xi'|^2\|\vph\|^2).
$$
Inserting this relation into \eqref{170821_5} and noting $\lambda^{-1}=\overline\lambda|\lambda|^{-2}$ furnish
\begin{align}\label{170821_6}
\lambda\sum_{J=1}^N\|u_J\|^2+\mu_*\sum_{J=1}^N\left(\|\pd_N u_J\|^2+|\xi'|^2\|u_J\|^2\right)& \\
+\nu_*\|\vph\|^2+\kappa_*\overline\lambda|\lambda|^{-2}(\|\pd_N\vph\|^2+|\xi'|^2\|\vph\|^2)&=0. \notag
\end{align}

{\bf Step 3.}
One takes the real part of \eqref{170821_6} and 
the imaginary part of \eqref{170821_6} in order to obtain
\begin{align}
(\Re\lambda)\left\{\sum_{J=1}^N\|u_J\|^2+\kappa_*|\lambda|^{-2}(\|\pd_N\vph\|^2+|\xi'|^2\|\vph\|^2)\right\}& \label{170821_7} \\
\quad+\mu_*\sum_{J=1}^N\left(\|\pd_N u_J\|^2+|\xi'|^2\|u_J\|^2\right)+\nu_*\|\vph\|^2&=0, \notag \\
(\Im\lambda)\left\{\sum_{J=1}^N\|u_J\|^2-\kappa_*|\lambda|^{-2}(\|\pd_N\vph\|^2+|\xi'|^2\|\vph\|^2)\right\}&=0. \label{170821_8}
\end{align}

It now holds by \eqref{170821_7} that $\vph=0$.
One then sees that $(u_1,\dots,u_N)=(0,\dots,0)$ by \eqref{170821_7} when $\Re\lambda>0$
and by \eqref{170821_8} when $\Re\lambda=0$.
Hence, $(u_1,\dots,u_N,\vph)=(0,0,\dots,0)$,
which contradicts the fact that $(u_1,\dots,u_N,\vph)$ is a non-trivial solution.
This completes the proof of the lemma.
\end{proof}

Let us write the inverse matrix $\BL^{-1}$ of $\BL$ as follows:
\begin{equation*}
\BL^{-1} = \frac{1}{\det\BL}
\begin{pmatrix}
L_{11} & L_{12} \\
L_{21} & L_{22}
\end{pmatrix},
\end{equation*}
where
\begin{alignat}{2}\label{eq:cofacI}
L_{11} &= -t_1(t_2\omega_\lambda-|\xi'|^2),  &&\quad
L_{12} = -(t_2^2-|\xi'|^2), \\
L_{21} &= t_2(t_1\omega_\lambda-|\xi'|^2), && \quad
L_{22} =t_1^2-|\xi'|^2. \notag
\end{alignat}
One then sees that, by solving \eqref{SEs},
\begin{align}\label{bega}
\beta_N &=
\frac{\lambda L_{11}}{\det\BL}\wht g(0)+\frac{t_1 t_2 L_{12}}{\det\BL}i\xi'\cdot\wht\Bh'(0), \\
\gamma_N &=
\frac{\lambda L_{21}}{\det\BL}\wht g(0) +\frac{t_1 t_2 L_{22}}{\det\BL}i\xi'\cdot\wht\Bh'(0). \notag
\end{align}
On the other hand, one has, by \eqref{eq:9_0}, \eqref{eq:9}, \eqref{eq:10_2}, \eqref{eq:51}, \eqref{eq:60}, and \eqref{eq:55},
\begin{align*}
\wht u_j(x_N)
&=
\wht h_j(0)e^{-\omega_\lambda x_N} 
-\frac{i\xi_j}{t_1}\beta_N(e^{-t_1 x_N}-e^{-\omega_\lambda x_N})  \\
&-\frac{i\xi_j}{t_2}\gamma_N(e^{-t_2 x_N}-e^{-\omega_\lambda x_N}), \\ 
\wht u_N(x_N)
&=
\beta_N\left(e^{-t_1 x_N}-e^{-\omega_\lambda x_N}\right)
+\gamma_N\left(e^{-t_2 x_N}-e^{-\omega_\lambda x_N}\right), \notag \\
\wht\vph(x_N)&=-\left(\frac{t_1^2-|\xi'|^2}{t_1}\right)\beta_Ne^{-t_1 x_N}
-\left(\frac{t_2^2-|\xi'|^2}{t_2}\right)\gamma_N e^{-t_2 x_N},
\end{align*}
and sets $\wht\rho(x_N)=-\lambda^{-1}\wht\vph(x_N)$ in view of \eqref{eq:1}.
Recall the inverse partial Fourier transform given in \eqref{defi:IPFT} and
set $\rho=\CF_{\xi'}^{-1}[\wht\rho(x_N)](x')$ and $u_J=\CF_{\xi'}^{-1}[\wht u_J(x_N)](x')$.
Then $\rho$ and $\Bu=(u_1,\dots, u_N)^\SST$ solve the system \eqref{eq2:half-red}.

\subsection{Analysis of symbols}\label{subsec3-2}
This subsection estimates several symbols 
arising from the representation formulas of solutions obtained in Subsection \ref{subsec3-1}.


Let us define the following symbols: 
\begin{align}\label{symbols_2}
\Fm_k(\xi',\lambda)&=\frac{t_k(t_k+\omega_\lambda)\det\BL}{\lambda(t_2-t_1)} \quad (k=1,2),  \\
\Fn_1(\xi',\lambda)&=\frac{(t_2+\omega_\lambda)L_{11}}{\lambda}, \quad
\Fn_2(\xi',\lambda)=\frac{(t_1+\omega_\lambda)L_{21}}{\lambda}, \notag \\
\Fp_1(\xi',\lambda)&=\frac{t_1+\omega_\lambda}{t_2+\omega_\lambda},
\quad \Fp_2(\xi',\lambda)=\frac{t_2+\omega_\lambda}{t_1+\omega_\lambda}. \notag
\end{align}
Subsequently, $k=1$ or $k=2$, and also
one often denotes $\Fm_k(\xi,\lambda)$, $\Fn_k(\xi',\lambda)$, and $\Fp_k(\xi',\lambda)$
by $\Fm_k$, $\Fn_k$, and $\Fp_k$, respectively, for short.
Recall by \eqref{defi:t} and \eqref{defi:omega} that 
\begin{equation}\label{recall:1}
t_k^2-|\xi'|^2=s_k\lambda, \quad t_k^2-\omega_\lambda^2=(s_k-\mu_*^{-1})\lambda, \quad
\omega_\lambda^2-|\xi'|^2 =\mu_*^{-1}\lambda.
\end{equation}
Since $\det\BL$ is written as
\begin{align*}
\det\BL 
&=(t_2-t_1)\{t_2\omega_\lambda(t_2+t_1)(t_1-\omega_\lambda)-\omega_\lambda^2(t_1^2-|\xi'|^2) \\
& +(\omega_\lambda^2-|\xi'|^2)(t_2^2+t_1t_2+t_1^2-|\xi'|^2)\} \\
&=(t_2-t_1)\{t_1\omega_\lambda(t_2+t_1)(t_2-\omega_\lambda)-\omega_\lambda^2(t_2^2-|\xi'|^2) \\
&+(\omega_\lambda^2-|\xi'|^2)(t_2^2+t_1t_2+t_1^2-|\xi'|^2)\},
\end{align*}
one has by \eqref{recall:1}
\begin{align}\label{180318_1}
\Fm_k(\xi',\lambda)
&=(s_k-\mu_*^{-1}) t_1t_2 \omega_\lambda (t_2+t_1)-s_k t_k \omega_\lambda^2 (t_k+\omega_\lambda) \\
&+\mu_*^{-1}t_k(t_k+\omega_\lambda)(t_2^2+t_1 t_2+t_1^2-|\xi'|^2). \notag
\end{align}
In addition, by \eqref{recall:1},
\begin{align*}
(t_k+\omega_\lambda)(t_k\omega_\lambda-|\xi'|^2)
&=(t_k+\omega_\lambda)\{(t_k-\omega_\lambda)\omega_\lambda+\omega_\lambda^2-|\xi'|^2\} \\
&=(s_k-\mu_*^{-1})\lambda\omega_\lambda+\mu_*^{-1}\lambda(t_k+\omega_\lambda),
\end{align*}
which furnishes
\begin{align}\label{180227_1}
\Fn_1(\xi',\lambda)&=-t_1\{(s_2-\mu_*^{-1})\omega_\lambda+\mu_*^{-1}(t_2+\omega_\lambda)\}, \\
\Fn_2(\xi',\lambda)&=t_2\{(s_1-\mu_*^{-1})\omega_\lambda+\mu_*^{-1}(t_1+\omega_\lambda)\}. \notag
\end{align}

By \eqref{eq:cofacI}, \eqref{180318_1}, \eqref{180227_1}, and Lemmas \ref{lemm:algebra} and \ref{lemm:symbol0},
we immediately obtain

\begin{lemm}\label{lemm:symbol1}
Let $k=1,2$. It then holds that 
\begin{equation*}
L_{k 1}\in\BBM_{3,1}(\BC_+), \quad
\Fm_k \in \BBM_{4,1}(\BC_+), \quad 
\Fn_k \in \BBM_{2,1}(\BC_+), \quad
\Fp_k \in\BBM_{0,1}(\BC_+).
\end{equation*}
\end{lemm}


To treat $\Fm_k(\xi',\lambda)^{-1}$, we prove

\begin{lemm}\label{lemm:lopatinski}
There exists a positive constant $C_{\mu_*,\kappa_*,\nu_*}$ such that, 
for any $(\xi',\lambda)\in\BR^{N-1}\times (\overline{\BC_+}\setminus\{0\})$, there holds the estimate:
\begin{equation}\label{180305_1}
|\Fm_k(\xi',\lambda)|\geq C_{\mu_*,\kappa_*,\nu_*}(|\lambda|^{1/2}+|\xi'|)^4 \quad (k=1,2).
\end{equation}
\end{lemm}

\begin{proof}
{\bf Case 1}. One considers the case $|\xi'|^2/|\lambda|\leq R_1$, where
$\xi'\in\BR^{N-1}$, $\lambda\in\overline{\BC_+}\setminus\{0\}$, and
some number $R_1\in(0,1)$ determined below.

Let $z=|\xi'|^2/\lambda$. Then,
\begin{equation*}
t_k=\sqrt{s_k\lambda}(1+O(z)), \quad \omega_\lambda=\sqrt{\mu_*^{-1}\lambda}(1+O(z))
\quad \text{as $|z|\to 0$,}
\end{equation*}
which yields that
$$
t_k+\omega_\lambda
=\Big(\sqrt{s_k}+\sqrt{\mu_*^{-1}}\,\Big)\sqrt{\lambda}(1+O(z)) \quad \text{as $|z|\to0$}
$$
and that by the second equality of \eqref{def:detL}
\begin{align*}
&\frac{\det\BL}{t_2-t_1}
=\sqrt{s_2}\sqrt{s_1}\sqrt{\mu_*^{-1}}\Big(\sqrt{s_2}+\sqrt{s_1}\Big)\lambda^2(1+O(z)) \quad \text{as $|z|\to0$.}
\end{align*}
One thus has by \eqref{symbols_2}
\begin{equation*}
\Fm_k(\xi',\lambda)=\sqrt{s_k}\Big(\sqrt{s_k}+\sqrt{\mu_*^{-1}}\,\Big)
\sqrt{s_2}\sqrt{s_1}\sqrt{\mu_*^{-1}}\Big(\sqrt{s_2}+\sqrt{s_1}\Big)
\lambda^2(1+O(z))
\end{equation*}
as $|z|\to 0$.
Since there is a positive constant $M$, depending on at most $\mu_*$, $\nu_*$, and $\kappa_*$, such that
$\Re\sqrt{s_k}\geq M$, one observes that
\begin{align*}
|\sqrt{s_k}|&\geq \Re\sqrt{s_k} \geq M, \\
\Big|\sqrt{s_k}+\sqrt{\mu_*^{-1}}\Big|
&\geq \Re\Big(\sqrt{s_k}+\sqrt{\mu_*^{-1}}\,\Big) \geq M+\sqrt{\mu_*^{-1}}, \\
|\sqrt{s_2}+\sqrt{s_1}|&\geq \Re(\sqrt{s_1}+\sqrt{s_2})\geq 2 M.
\end{align*}
Combining these inequalities with the last equality yields that
there is a constant $R_1\in(0,1)$ such that,
for any $(\xi',\lambda)\in\BR^{N-1}\times(\overline{\BC_+}\setminus\{0\})$ with $|\xi'|^2/|\lambda|\leq R_1$,
\begin{equation*}
|\Fm_\ell(\xi',\lambda)|\geq  C_{\mu_*,\nu_*,\kappa_*}|\lambda|^2
\end{equation*}
for some positive constant $C_{\mu_*,\nu_*,\kappa_*}$.
On the other hand, it holds by $(a+b)^2/2\leq (a^2+b^2)$ with $a,b\geq0$ that 
\begin{align*}
|\lambda|^2&=\frac{1}{2}(|\lambda|^2+|\lambda|^2)\geq 
\frac{1}{2}\left\{|\lambda|^2+\left(\frac{|\xi'|^2}{R_1}\right)^2\right\}\geq \frac{1}{4}\left(|\lambda|+\frac{|\xi'|^2}{R_1}\right)^2 \\
& \geq \frac{1}{16}\left(|\lambda|^{1/2}+\frac{|\xi'|}{\sqrt{R_1}}\right)^4.
\end{align*}
By the last two inequalities,
we have \eqref{180305_1} for $(\xi',\lambda)\in\BR^{N-1}\times(\overline{\BC_+}\setminus\{0\})$ with $|\xi'|^2/|\lambda|\leq R_1$.

{\bf Case 2}.
One considers the case $|\lambda|/|\xi'|^2\leq R_2$, where
$\xi'\in\BR^{N-1}\setminus\{0\}$, $\lambda\in\overline{\BC_+}\setminus\{0\}$, and
some number $R_2\in(0,1)$ determined below.

Let $y=\lambda/|\xi'|^2$. Then,
\begin{equation*}
t_k=|\xi'|(1+O(y)), \quad \omega_\lambda=|\xi'|(1+O(y))
\quad \text{as $|y|\to 0$,}
\end{equation*}
which, combined with \eqref{180318_1}, furnishes that
\begin{equation*}
\Fm_k(\xi',\lambda)=2\mu_*^{-1}|\xi'|^4(1+O(y)) \quad \text{as $|y|\to 0$.}
\end{equation*}
Thus there is a constant $R_2\in(0,1)$ such that,
for any $(\xi',\lambda)\in(\BR^{N-1}\setminus\{0\})\times(\overline{\BC_+}\setminus\{0\})$ with $|\lambda|/|\xi'|^2\leq R_2$,
\begin{equation*}
|\Fm_\ell(\xi',\lambda)|\geq \mu_*^{-1}|\xi'|^4.
\end{equation*}
Similarly to Step 1, this inequality yields \eqref{180305_1}
for $(\xi',\lambda)\in(\BR^{N-1}\setminus\{0\})\times(\overline{\BC_+}\setminus\{0\})$ with $|\lambda|/|\xi'|^2\leq R_2$.

{\bf Case 3}.
One considers the case $|\xi'|^2/|\lambda|\geq R_1/2$ and $|\lambda|/|\xi'|^2\geq R_2/2$,
where $\xi'\in\BR^{N-1}\setminus\{0\}$,
$\lambda\in\overline{\BC_+}\setminus\{0\}$,
and $R_1$, $R_2$ are positive constants introduced in the above two steps.
The condition of this case is equivalent to
\begin{equation}\label{180305_3}
\frac{R_2}{2}|\xi'|^2\leq |\lambda| \leq \frac{2}{R_1}|\xi'|^2.
\end{equation}

Let $\wtd\xi'$, $\wtd\lambda$, $\wtd t_k$, and $\wtd\omega_\lambda$ be given by
\begin{alignat*}{2}
\wtd\xi'&=(|\lambda|^{1/2}+|\xi'|)^{-1}\xi', \quad 
&\wtd\lambda&=(|\lambda|^{1/2}+|\xi'|)^{-2}\lambda, \\
\wtd t_k&=\sqrt{|\wtd\xi'|^2+s_k\wtd\lambda}, 
&\wtd\omega_\lambda&=\sqrt{|\wtd\xi'|^2+\mu_*^{-1}\wtd\lambda}.
\end{alignat*}
One then observes that
\begin{equation*}
\Fm_k(\xi',\lambda)=(|\lambda|^{1/2}+|\xi'|)^4\Fm_k(\wtd\xi',\wtd\lambda)
\end{equation*}
and that \eqref{180305_3} implies $r_1\leq|\wtd\xi'|\leq r_2$ and $r_3\leq|\wtd\lambda|\leq r_4$, where 
\begin{alignat*}{2}
r_1&=\left(\sqrt{\frac{2}{R_1}}+1\right)^{-1}, \quad 
&r_2&=\left(\sqrt{\frac{R_2}{2}}+1\right)^{-1}, \\
r_3&=\left(\sqrt{\frac{2}{R_2}}+1\right)^{-2}, \quad
&r_4&=\left(\sqrt{\frac{R_1}{2}}+1\right)^{-2}.
\end{alignat*}
We here define a compact set $K$ as follows:
\begin{equation*}
K=\{(\wtd\xi',\wtd\lambda)\in\BR^{N-1}\times\overline{\BC_+}
\mid r_1\leq|\wtd\xi'|\leq r_2, r_3\leq|\wtd\lambda|\leq r_4\}.
\end{equation*} 
Since $\Fm_k(\xi',\lambda)$ is continuous and $\Fm_k(\xi',\lambda)\neq 0$ on 
$\BR^{N-1}\times (\overline{\BC_+}\setminus\{0\})$
by \eqref{symbols_2} and Lemma \ref{lemm:detL},
there exists at least one minimum of $|\Fm_\ell(\wtd\xi',\wtd\lambda)|$ over $K$ such that
\begin{equation*}
m_k:=\min_{(\wtd\xi',\wtd\lambda)\in K}|\Fm_k(\wtd\xi',\wtd\lambda)|>0.
\end{equation*}
Thus, for any $(\xi',\lambda)\in(\BR^{N-1}\setminus\{0\})\times(\overline{\BC_+}\setminus\{0\})$ with \eqref{180305_3},
\begin{equation*}
|\Fm_k(\xi',\lambda)|
=(|\lambda|^{1/2}+|\xi'|)^4|\Fm_k(\wtd\xi',\wtd\lambda)|\geq m_k(|\lambda|^{1/2}+|\xi'|)^4,
\end{equation*}
which implies that \eqref{180305_1} holds for Case 3.

Summing up the above estimates, we have completed the proof of Lemma \ref{lemm:lopatinski}.
\end{proof}

\begin{coro}\label{coro:lopatinski}
Let $k=1,2$. Then $\Fm_k^{-1}\in\BBM_{-4,1}(\BC_+)$.
\end{coro}

\begin{proof}
Recall Bell's formula for derivatives of the composite function of $f(t)$ and $g(\xi')$  as follows:
for any multi-index $\alpha'\in\BN_0^{N-1}$,
\begin{equation*}
\pd_{\xi'}^{\alpha'} f(g(\xi'))=\sum_{k=1}^{|\alpha|} f^{(k)}(g(\xi'))
\sum_{\stackrel{\text{\scriptsize{$\alpha_1'+\dots+\alpha_k'=\alpha'$,}}}{|\alpha_j'|\geq 1}}
\Gamma_{\alpha_1',\dots,\alpha_k'}^{\alpha'}(\pd_{\xi'}^{\alpha_1'}g(\xi'))\dots(\pd_{\xi'}^{\alpha_k'}g(\xi'))
\end{equation*}
with suitable coefficients $\Gamma_{\alpha_1',\dots,\alpha_k'}^{\alpha'}$,
where $f^{(k)}(t)$ is the $k$th derivative of $f(t)$.

One first proves, for any $r\in \BR$ and multi-index $\alpha'\in\BN_0^{N-1}$, 
\begin{equation}\label{180305_5}
|\pd_{\xi'}^{\alpha'}\Fm_k(\xi',\lambda)^r|\leq C_{r,\alpha',\mu_*,\nu_*,\kappa_*}(|\lambda|^{1/2}+|\xi'|)^{4r-|\alpha'|},
\end{equation}
where $(\xi',\lambda)\in(\BR^{N-1}\setminus\{0\})\times\BC_+$.
Using Bell's formula with $f(t)=t^r$ and $g(\xi')=\Fm_k(\xi',\lambda)$, we have
by Lemmas \ref{lemm:symbol1} and \ref{lemm:lopatinski}
\begin{align*}
|\pd_{\xi'}^{\alpha'}\Fm_k(\xi',\lambda)^r|
&\leq C_{r,\alpha',\mu_*,\nu_*,\kappa_*}\sum_{k=1}^{|\alpha'|}|\Fm_k(\xi',\lambda)|^{r-k}(|\lambda|^{1/2}+|\xi'|)^{4k-|\alpha'|} \\
&\leq C_{r,\alpha',\mu_*,\nu_*,\kappa_*}\sum_{k=1}^{|\alpha'|}(|\lambda|^{1/2}+|\xi'|)^{4(r-k)}(|\lambda|^{1/2}+|\xi'|)^{4k-|\alpha'|} \\
&\leq C_{r,\alpha',\mu_*,\nu_*,\kappa_*}(|\lambda|^{1/2}+|\xi'|)^{4r-|\alpha'|}
\end{align*}
for any $(\xi',\lambda)\in(\BR^{N-1}\setminus\{0\})\times\BC_+$.
This implies that \eqref{180305_5} holds true.

Now it holds that
\begin{equation*}
\lambda\frac{d}{d\lambda}\Fm_k(\xi',\lambda)^{-1}=-\Fm_k(\xi',\lambda)^{-2}\left(\lambda\frac{d}{d\lambda}\Fm_k(\xi',\lambda)\right).
\end{equation*}
Combining this relation with \eqref{180305_5} for $r=-2$ and Lemma \ref{lemm:symbol1} furnishes,
by Leibniz's rule, that for any multi-index $\alpha'\in\BN_0^{N-1}$ and
$(\xi',\lambda)\in(\BR^{N-1}\setminus\{0\})\times\BC_+$
\begin{align*}
&\left|\pd_{\xi'}^{\alpha'}\left(\lambda\frac{d}{d\lambda}\Fm_k(\xi',\lambda)^{-1}\right)\right| \\
&\leq C_{\alpha'}\sum_{\beta'+\gamma'=\alpha'}\left|\pd_{\xi'}^{\beta'}\Fm_k(\xi',\lambda)^{-2}\right|
\left|\pd_{\xi'}^{\gamma'}\left(\lambda\frac{d}{d\lambda}\Fm_k(\xi',\lambda)\right)\right| \\
&\leq C_{\alpha',\mu_*,\nu_*,\kappa_*} \sum_{\beta'+\gamma'=\alpha'}
(|\lambda|^{1/2}+|\xi'|)^{-8-|\beta'|}(|\lambda|^{1/2}+|\xi'|)^{4-|\gamma'|} \\
&\leq C_{\alpha',\mu_*,\nu_*,\kappa_*}(|\lambda|^{1/2}+|\xi'|)^{-4-|\alpha'|}.
\end{align*}
This inequality and \eqref{180305_5} with $r=-1$ complete the proof of the corollary.
\end{proof}

\subsection{Proof of Theorem \ref{theo:half-red}}\label{subsec3-3}
This subsection proves Theorem \ref{theo:half-red} 
by means of results obtained in Subsections \ref{subsec3-1} and \ref{subsec3-2}.
Let $\CM_0(x_N)$, $\CM_1(x_N)$, and $\CM_2(x_N)$ be symbols given in \eqref{stkernel:1}.

It now holds that
\begin{align*}
\rho
&=
\CF_{\xi'}^{-1}\left[\left\{\left(\frac{t_1^2-|\xi'|^2}{\lambda t_1}\right)\beta_N+\left(\frac{t_2-|\xi'|^2}{\lambda t_2}\right)\gamma_N\right\} e^{-t_1 x_N}\right](x') \\
&+\CF_{\xi'}^{-1}\left[\left(\frac{t_2^2-|\xi'|^2}{\lambda t_2}\right)\gamma_N\left(e^{-t_2 x_N}-e^{-t_1 x_N}\right)\right](x') \\
&=:\rho_1+\rho_2.
\end{align*}
By \eqref{eq:cofacI}, \eqref{bega}, \eqref{symbols_2}, and \eqref{recall:1}, we see that 
\begin{align*}
\rho_1&=
\CF_{\xi'}^{-1}\left[
\left(\frac{s_1\lambda L_{11}}{t_1\det\BL}+\frac{s_2\lambda L_{21}}{t_2\det\BL}\right) \wht g(0)e^{-t_1 x_N}\right](x') \\
&+\CF_{\xi'}^{-1}\left[
\left(\frac{s_1t_2L_{12}+s_2 t_1 L_{22}}{\det\BL}\right)i\xi'\cdot\wht \Bh'(0)e^{-t_1 x_N}
\right](x') \\
&=\sum_{k=1}^2\CF_{\xi'}^{-1}\left[\frac{s_k (t_2+t_1)\Fp_k(\xi',\lambda)\Fn_k(\xi',\lambda)}{(s_2-s_1)\Fm_k(\xi',\lambda)}e^{-t_1 x_N}\wht g(0)\right](x') \\
&-\sum_{l=1}^{N-1}\CF_{\xi'}^{-1}\left[\frac{s_1s_2i\xi_l t_1(t_1+\omega_\lambda)}{\Fm_1(\xi',\lambda)}e^{-t_1 x_N}\wht h_l(0)\right](x').
\end{align*}
On the other hand, $\beta_N$ and $\gamma_N$ are written as
\begin{align}\label{171010_1}
\beta_N
&=\frac{t_1(t_1+\omega_\lambda)L_{11}}{(t_2-t_1)\Fm_1(\xi',\lambda)}\wht g(0)
-\frac{s_2 t_1^2 t_2(t_1+\omega_\lambda)}{(t_2-t_1)\Fm_1(\xi',\lambda)}i\xi'\cdot\wht\Bh'(0), \\
\gamma_N
&=\frac{t_2(t_2+\omega_\lambda)L_{21}}{(t_2-t_1)\Fm_2(\xi',\lambda)}\wht g(0)
+\frac{s_1t_1t_2^2(t_2+\omega_\lambda)}{(t_2-t_1)\Fm_2(\xi',\lambda)}i\xi'\cdot\wht\Bh'(0), \notag
\end{align}
which furnishes, together with \eqref{recall:1}, that
\begin{align*}
\rho_2&=
\CF_{\xi'}^{-1}\left[\frac{s_2(t_2+\omega_\lambda)L_{21}}{\Fm_2(\xi',\lambda)}\CM_0(x_N)\wht g(0)\right](x') \notag \\
&+\sum_{l=1}^{N-1}\CF_{\xi'}^{-1}\left[\frac{s_1s_2 i\xi_l t_1t_2(t_2+\omega_\lambda)}{\Fm_2(\xi',\lambda)}\CM_0(x_N)\wht h_l(0)\right](x'). 
\end{align*} 
Recalling $\rho=\rho_1+\rho_2$, one obtains
\begin{align}\label{rep:1}
\rho
&=\sum_{k=1}^2\CF_{\xi'}^{-1}\left[\frac{s_k (t_2+t_1)\Fp_k(\xi',\lambda)\Fn_k(\xi',\lambda)}{(s_2-s_1)\Fm_k(\xi',\lambda)}e^{-t_1 x_N}\wht g(0)\right](x') \\
&-\sum_{l=1}^{N-1}\CF_{\xi'}^{-1}\left[\frac{s_1s_2i\xi_l t_1(t_1+\omega_\lambda)}{\Fm_1(\xi',\lambda)}e^{-t_1 x_N}\wht h_l(0)\right](x') \notag \\
&+\CF_{\xi'}^{-1}\left[\frac{s_2(t_2+\omega_\lambda)L_{21}}{\Fm_2(\xi',\lambda)}\CM_0(x_N)\wht g(0)\right](x') \notag \\
&+\sum_{l=1}^{N-1}\CF_{\xi'}^{-1}\left[\frac{s_1s_2 i\xi_l t_1t_2(t_2+\omega_\lambda)}{\Fm_2(\xi',\lambda)}\CM_0(x_N)\wht h_l(0)\right](x'). \notag
\end{align}
Similarly, we observe by \eqref{171010_1} that, for $u_j$ $(j=1,\dots,N-1)$ and $u_N$, 
\begin{align}\label{rep:2}
u_j
&=\CF_{\xi'}^{-1}\left[e^{-\omega_\lambda x_N}\wht l_j(0)\right](x') \\
&-\sum_{k=1}^2\CF_{\xi'}^{-1}\left[\frac{i\xi_j(t_k+\omega_\lambda)L_{k1}}{\Fm_k(\xi',\lambda)}\CM_k(x_N)\wht g(0)\right](x') \notag \\
&+\sum_{k=1}^2\sum_{l=1}^{N-1}\CF_{\xi'}^{-1}\left[\frac{(-1)^k s_1s_2\xi_j\xi_l t_1t_2(t_k+\omega_\lambda)}{s_k\Fm_k(\xi',\lambda)}
\CM_k(x_N)\wht h_l(0)\right](x'), \notag \\
u_N
&=\sum_{k=1}^2\CF_{\xi'}^{-1}\left[\frac{t_k(t_k+\omega_\lambda)L_{k 1}}{\Fm_k(\xi',\lambda)}\CM_k(x_N)\wht g(0)\right](x') \notag \\
&+\sum_{k=1}^2\sum_{l=1}^{N-1}\CF_{\xi'}^{-1}\left[\frac{(-1)^k s_1s_2i\xi_l  t_1t_2t_k(t_k+\omega_\lambda)}
{s_k\Fm_k(\xi',\lambda)}\CM_k(x_N)\wht h_l(0)\right](x'). \notag
\end{align}

By Lemmas \ref{lemm:algebra}, \ref{lemm:symbol0}, and \ref{lemm:symbol1} and by Corollary \ref{coro:lopatinski},
the symbols of $\rho$, $u_J$ ($J=1,\dots,N$) satisfy the following conditions:
For $\rho$, there hold
\begin{align}\label{multi:rho}
\frac{s_k(t_2+t_1)\Fp_k(\xi',\lambda)\Fn_k(\xi',\lambda)}{(s_2-s_1)\Fm_k(\xi',\lambda)},
\ \frac{s_1s_2i\xi_l t_1(t_1+\omega_\lambda)}{\Fm_1(\xi',\lambda)}
&\in 
\BBM_{-1,1}(\BC_+), \\
\frac{s_2(t_2+\omega_\lambda)L_{21}}{\Fm_2(\xi',\lambda)},
\ \frac{s_1s_2i\xi_l t_1t_2(t_2+\omega_\lambda)}{\Fm_2(\xi',\lambda)}
&\in \BBM_{0,1}(\BC_+); \notag
\end{align} 
For $u_J$, there hold
\begin{align}\label{multi:u}
\frac{i\xi_j(t_k+\omega_\lambda)L_{k 1}}{\Fm_k(\xi',\lambda)},
\ \frac{(-1)^k s_1s_2\xi_j\xi_l t_1t_2(t_k+\omega_\lambda)}{s_k \Fm_k(\xi',\lambda)}
&\in\BBM_{1,1}(\BC_+), \\
\frac{t_k(t_k+\omega_\lambda)L_{k 1}}{\Fm_k(\xi',\lambda)},
\ \frac{(-1)^k s_1s_2i\xi_l  t_1 t_2t_k (t_k+\omega_\lambda)}{s_k\Fm_k(\xi',\lambda)}
&\in\BBM_{1,1}(\BC_+). \notag
\end{align}

Finally, combing \eqref{rep:1}-\eqref{multi:u} with Lemmas \ref{lemm:multiplier1}, \ref{lemm:multiplier2}, and \ref{prop:R}
shows the existence of solution operators $\CA^2(\lambda)$ and $\CB^2(\lambda)$ stated in Theorem \ref{theo:half-red}.
This completes the proof of Theorem \ref{theo:half-red} for Cases I and II.

\section{Proof of Theorem \ref{theo:half-red} for Case III}\label{sec4}
This section proves Theorem \ref{theo:half-red} for Case III. 
Throughout this section, we assume that $\mu_*$, $\nu_*$, and $\kappa_*$ are positive constants satisfying
the condition of Case III.
One then recalls Lemmas \ref{lemm:roots_s} \eqref{lemm:roots_s3} and \ref{lemm:roots_t} \eqref{coro:roots_s3},
and considers the case $\mu_*>\nu_*$ only, i.e.
\begin{equation*}
t_1=\omega_\lambda=\sqrt{|\xi'|^2+\mu_*^{-1}\lambda}, \quad t_2=\sqrt{|\xi'|^2+\nu_*^{-1}\lambda},
\end{equation*}
which are often used in the following computations.
Let $J=1,\dots,N$ and $j=1,\dots,N-1$ in this section.

\subsection{Solution formulas}\label{subsec4-1}
One first considers
\eqref{170821_1}-\eqref{160827_1} with \eqref{eq:4}, \eqref{eq:44}, and \eqref{eq:div-vph}
in order to derive solution formulas of \eqref{eq2:half-red}.
In view of \eqref{170821_1}, \eqref{160827_1}, and Lemma \ref{lemm:roots_P}, we look for solutions $\wht u_J$ and $\wht\vph$ of the forms:
\begin{align}
\wht u_J&=\alpha_J e^{-\omega_\lambda x_N}+\beta_J x_Ne^{-\omega_\lambda x_N}+\gamma_J(e^{-t_2 x_N}-e^{-\omega_\lambda x_N}), \label{eq:400} \\
\wht\vph&=\sigma e^{-\omega_\lambda x_N}+\tau e^{-t_2 x_N}. \label{eq:400_2}
\end{align}
It then holds by \eqref{eq:div-vph} that
\begin{align}
\sigma&=i\xi'\cdot\alpha'-i\xi'\cdot\gamma'-\omega_\lambda\alpha_N+\beta_N+\omega_\lambda \gamma_N, \label{eq:402_3} \\
0&=i\xi'\cdot\beta'-\omega_\lambda\beta_N, \label{eq:402} \\
\tau&=i\xi'\cdot\gamma'-t_2 \gamma_N. \label{eq:402_5}
\end{align}
On the other hand, by the assumption $\kappa_*=\mu_*\nu_*$, 
\begin{equation*}
\nu_*\lambda-\kappa_*(\pd_N^2-|\xi'|^2)=-\mu_*\nu_*(\pd_N^2-\omega_\lambda^2),
\end{equation*}
and thus \eqref{1023_18:eq5} and \eqref{1103_18:eq5} are respectively equivalent to 
\begin{align}\label{eq:420}
\lambda(\pd_N^2-\omega_\lambda^2)\wht u_j-\nu_* i \xi_j(\pd_N^2-\omega_\lambda^2)\wht\vph&=0, \\
\lambda(\pd_N^2-\omega_\lambda^2)\wht u_N-\nu_* \pd_N(\pd_N^2-\omega_\lambda^2)\wht\vph&=0. \notag
\end{align}
Here note that
$$
(\pd_N^2-\omega_\lambda^2)(x_N e^{-\omega_\lambda x_N})=-2\omega_\lambda e^{-\omega_\lambda x_N}.
$$
Inserting \eqref{eq:400} and \eqref{eq:400_2}, together with the last relation, into \eqref{eq:420} yields
\begin{alignat*}{2}
-2\lambda\omega_\lambda\beta_j&=0, \quad
(t_2^2-\omega_\lambda^2)(\lambda\gamma_j-\nu_*i\xi_j\tau)&&=0, \\
-2\lambda\omega_\lambda\beta_N&=0, \quad
(t_2^2-\omega_\lambda^2)(\lambda\gamma_N+\nu_* t_2\tau)&&=0,
\end{alignat*}
which, combined with $t_2\neq \omega_\lambda$, furnishes
\begin{align}
\beta_J&=0, \label{eq:403} \\
\lambda\gamma_j-\nu_*i\xi_j\tau&=0, \label{eq:404} \\
\lambda\gamma_N+\nu_* t_2\tau&=0. \label{eq:405}
\end{align}

\begin{rema}
The relation \eqref{eq:403} implies \eqref{eq:402}, 
while the relations \eqref{eq:404} and \eqref{eq:405} imply \eqref{eq:402_5}.
\end{rema}

One has by \eqref{eq:400}, \eqref{eq:402_3}, and \eqref{eq:403},
\begin{align}
\wht u_J&=\alpha_J e^{-\omega_\lambda x_N}+\gamma_J(e^{-t_2 x_N}-e^{-\omega_\lambda x_N}), \label{eq:406} \\ 
\sigma&=i\xi'\cdot\alpha'-i\xi'\cdot\gamma'-\omega_\lambda\alpha_N+\omega_\lambda \gamma_N, \label{eq:406_3}
\end{align}
and also
by \eqref{eq:404} and \eqref{eq:405}
\begin{equation}\label{eq:408}
\gamma_j =-\frac{i\xi_j}{t_2}\gamma_N. 
\end{equation}
Furthermore, \eqref{eq:408} yields
\begin{equation}\label{eq:410}
i\xi'\cdot\gamma'=\frac{|\xi'|^2}{t_2}\gamma_N.
\end{equation}

Next, we consider the boundary conditions. 
By \eqref{eq:4} and \eqref{eq:406},
\begin{equation}\label{eq:411}
\alpha_j=\wht h_j(0), \quad \alpha_N=0.
\end{equation}
It then holds by the first relation of \eqref{eq:411} that
\begin{equation*}
i\xi'\cdot\alpha'=i\xi'\cdot\wht\Bh'(0), \quad \wht\Bh'(0)=(\wht h_1(0),\dots,\wht h_{N-1}(0))^\SST.
\end{equation*}
Combining this relation with \eqref{eq:406_3}, \eqref{eq:410}, and $\alpha_N=0$ of \eqref{eq:411} furnishes
\begin{equation}\label{eq:412_1}
\sigma=i\xi'\cdot\wht\Bh'(0)-\frac{|\xi'|^2}{t_2}\gamma_N+\omega_\lambda\gamma_N,
\end{equation}
while \eqref{eq:405} implies
\begin{equation}\label{eq:412_2}
\tau=-\frac{\nu_*^{-1}\lambda}{t_2}\gamma_N.
\end{equation}
One the other hand, by \eqref{eq:44} and \eqref{eq:400_2},
$$
-\omega_\lambda\sigma - t_2\tau= \lambda \wht g(0),
$$
which, combined with \eqref{eq:412_1} and \eqref{eq:412_2}, furnishes
$$
-\omega_\lambda\left(i\xi'\cdot\wht\Bh'(0)-\frac{|\xi'|^2}{t_2}\gamma_N+\omega_\lambda\gamma_N\right)+\nu_*^{-1}\lambda\gamma_N
=\lambda\wht g(0).
$$
Solving this equation with respect to $\gamma_N$, we have
\begin{equation*}
\gamma_N=\frac{t_2(\lambda\wht g(0)+\omega_\lambda i\xi'\cdot\wht\Bh'(0))}{\omega_\lambda|\xi'|^2-t_2\omega_\lambda^2+t_2\nu_*^{-1}\lambda}.
\end{equation*}
Here note that by $t_2^2=|\xi'|^2+\nu_*^{-1}\lambda$
\begin{align*}
\omega_\lambda|\xi'|^2-t_2\omega_\lambda^2+t_2\nu_*^{-1}\lambda
&=\omega_\lambda(t_2^2-\nu_*^{-1}\lambda)-t_2\omega_\lambda^2+t_2\nu_*^{-1}\lambda \\
&=(t_2-\omega_\lambda)(t_2\omega_\lambda+\nu_*^{-1}\lambda).
\end{align*}
The above formula of $\gamma_N$ is thus written as\footnote{
Estimates of $t_2\omega_\lambda+\nu_*^{-1}\lambda$ are given in Lemma \ref{lemm:multicaseIII} below. 
}
\begin{equation}\label{eq:413}
\gamma_N=\frac{t_2(\lambda\wht g(0)+\omega_\lambda i\xi'\cdot\wht\Bh'(0))}{(t_2-\omega_\lambda)(t_2\omega_\lambda+\nu_*^{-1}\lambda)}.
\end{equation}

Finally, one has, by \eqref{eq:400_2}, \eqref{eq:406}, \eqref{eq:408}, \eqref{eq:411}, \eqref{eq:412_1}, and \eqref{eq:412_2},
\begin{align*}
\wht u_j(x_N)
&= \wht h_j(0)e^{-\omega_\lambda x_N}-\frac{i\xi_j}{t_2}\gamma_N(e^{-t_2 x_N}-e^{-\omega_\lambda x_N}), \\
\wht u_N(x_N)
&=\gamma_N(e^{-t_2 x_N}-e^{-\omega_\lambda x_N}), \\
\wht\vph(x_N)
&=\left(i\xi'\cdot \wht \Bh'(0)-\frac{|\xi'|^2}{t_2}\gamma_N+\omega_\lambda\gamma_N\right)e^{-\omega_\lambda x_N} 
-\frac{\nu_*^{-1}\lambda}{t_2}\gamma_N e^{-t_2 x_N},
\end{align*}
and sets $\wht\rho(x_N)=-\lambda^{-1}\wht \vph(x_N)$ in view of \eqref{eq:1}.
Recall the inverse partial Fourier transform given in \eqref{defi:IPFT} and
set $\rho=\CF_{\xi'}^{-1}[\wht\rho(x_N)](x')$ and $u_J=\CF_{\xi'}^{-1}[\wht u_J(x_N)](x')$.
Then $\rho$ and $\Bu=(u_1,\dots, u_N)^\SST$ solve the system \eqref{eq2:half-red}.


The last part of this subsection is devoted to the proof of the following lemma.

\begin{lemm}\label{lemm:multicaseIII}
\begin{enumerate}[$(1)$]
\item
$t_2\omega_\lambda+\nu_*^{-1}\lambda\in\BBM_{2,1}(\BC_+)$.
\item
There is a positive constant $C_{\mu_*,\nu_*,\kappa_*}$ such that
\begin{equation}\label{Jan17:eq5}
|t_2\omega_\lambda+\nu_*^{-1}\lambda|\geq C_{\mu_*,\nu_*,\kappa_*}(|\lambda|^{1/2}+|\xi'|)^2
\end{equation}
for any $(\xi',\lambda)\in\BR^{N-1}\times\BC_+$.
\item
$(t_2\omega_\lambda+\nu_*^{-1}\lambda)^{-1}\in\BBM_{-2,1}(\BC_+)$.
\end{enumerate}
\end{lemm}

\begin{proof}
(1) The required property follows from Lemmas \ref{lemm:algebra} and \ref{lemm:symbol0} immediately.

(2). First, let us prove 
\begin{equation}\label{Jan17:eq3}
\Re\left(\frac{\nu_*^{-1}\lambda}{t_2}\right)>0.
\end{equation}

By $t_2^2=|\xi'|^2+\nu_*^{-1}\lambda$,
$$
\frac{\nu_*^{-1}\lambda}{t_2}=\frac{t_2^2-|\xi'|^2}{t_2}=t_2-\frac{|\xi'|^2}{|t_2|^2}\overline{t_2},
$$
which implies
\begin{equation}\label{Jan17:eq2}
\Re\left(\frac{\nu_*^{-1}\lambda}{t_2}\right)
=(\Re t_2)\left(1-\frac{|\xi'|^2}{|t_2|^2}\right)
=\frac{(\Re t_2)(|t_2|^4-|\xi'|^4)}{|t_2|^2(|t_2|^2+|\xi'|^2)}.
\end{equation}
One here observes by $|t_2|^4=||\xi'|^2+\nu_*^{-1}\lambda|^2$ that
\begin{align*}
|t_2|^4-|\xi'|^4
&=(|\xi'|^2+\nu_*^{-1}\Re\lambda)^2+(\nu_*^{-1}\Im\lambda)^2-|\xi'|^4 \\
&=\nu_*^{-2}|\lambda|^2+2|\xi'|^2\nu_*^{-1}\Re\lambda, 
\end{align*}
which, combined with $\Re\lambda>0$, furnishes
$$
|t_2|^4-|\xi'|^4\geq \nu_*^{-2}|\lambda|^2.
$$
Since $\Re t_2>0$ by Lemma \ref{lemm:roots_P}, the last inequality and \eqref{Jan17:eq2} imply \eqref{Jan17:eq3}.

Next, we prove \eqref{Jan17:eq5}. By \eqref{Jan17:eq3},
$$
\left|\omega_\lambda+\frac{\nu_*^{-1}\lambda}{t_2}\right|\geq \Re\left(\omega_\lambda+\frac{\nu_*^{-1}\lambda}{t_2}\right)
\geq \Re\omega_\lambda,
$$
which, combined with Lemma \ref{lemm:symbol0}, furnishes
\begin{align*}
|t_2\omega_\lambda+\nu_*^{-1}\lambda|
&=|t_2|\left|\omega_\lambda+\frac{\nu_*^{-1}\lambda}{t_2}\right| \\
&\geq (\Re t_2)(\Re \omega_\lambda) \\
&\geq C_{\mu_*,\nu_*,\kappa_*}(|\lambda|^{1/2}+|\xi'|)^2.
\end{align*}
This completes the proof of \eqref{Jan17:eq5}.

(3). The required property follows from (1) and (2) 
in the same manner as one has proved Corollary \ref{coro:lopatinski} from Lemmas \ref{lemm:symbol1} and \ref{lemm:lopatinski},
so that the detailed proof may be omitted.
\end{proof}

\subsection{Proof of Theorem \ref{theo:half-red}}\label{subsec4-2}
This subsection proves Theorem \ref{theo:half-red} by means of results obtained in Subsection \ref{subsec4-1}.
Let $\CM(x_N)$ be given in \eqref{stkernel:2}.

First, one considers the formula of $\rho$, which is written as
\begin{align*}
\rho
&=\CF_{\xi'}^{-1}
\left[\left\{
-\frac{1}{\lambda}\bigg(i\xi'\cdot \wht\Bh'(0)-\frac{|\xi'|^2}{t_2}\gamma_N+\omega_\lambda\gamma_N\bigg)+\frac{\nu_*^{-1}}{ t_2}\gamma_N
\right\} e^{-\omega_\lambda x_N}\right](x') \\
&+\CF_{\xi'}^{-1}\left[\frac{\nu_*^{-1}}{t_2}\gamma_N(e^{-t_2 x_N}-e^{-\omega_\lambda x_N})\right](x') \\
&=:\rho_1+\rho_2.
\end{align*}
The relation $t_2^2=|\xi'|^2+\nu_*^{-1}\lambda$ yields
\begin{equation*}
\rho_1=\CF_{\xi'}^{-1}
\left[\frac{1}{\lambda}\left\{-i\xi'\cdot\wht\Bh'(0)+(t_2-\omega_\lambda)\gamma_N\right\}e^{-\omega_\lambda x_N}\right](x').
\end{equation*}
Since it holds by \eqref{eq:413} that
\begin{equation*}
-i\xi'\cdot\wht\Bh'(0)+(t_2-\omega_\lambda)\gamma_N 
=\lambda\left(\frac{t_2}{t_2\omega_\lambda +\nu_*^{-1}\lambda}\wht g(0)-\frac{\nu_*^{-1}}{t_2\omega_\lambda+\nu_*^{-1}\lambda}i\xi'\cdot\wht\Bh'(0)\right)
\end{equation*}
the above formula of $\rho_1$ is reduced to
\begin{align*}
\rho_1
&=\CF_{\xi'}^{-1}\left[\left(\frac{t_2}{t_2\omega_\lambda +\nu_*^{-1}\lambda}\right)e^{-\omega_\lambda x_N}\wht g(0)\right](x') \\
&-\sum_{k=1}^{N-1}\left[\left(\frac{\nu_*^{-1}i\xi_k}{t_2\omega_\lambda+\nu_*^{-1}\lambda}\right)e^{-\omega_\lambda x_N}\wht h_k(0)\right](x').
\end{align*}
On the other hand, by \eqref{eq:413},
\begin{align*}
\rho_2
&=\CF_{\xi'}^{-1}
\left[\left(\frac{\nu_*^{-1}\lambda}{t_2\omega_\lambda+\nu_*^{-1}\lambda}\right)\CM(x_N)\,\wht g(0)\right](x') \\
&+\sum_{k=1}^{N-1}\CF_{\xi'}^{-1}\left[\left(\frac{\nu_*^{-1}\omega_\lambda i\xi_k}{t_2\omega_\lambda+\nu_*^{-1}\lambda}\right)\CM(x_N)\,\wht h_k(0)\right](x').
\end{align*}
Recalling $\rho=\rho_1+\rho_2$, one obtains
\begin{align}\label{solIII:rho}
\rho
&=\CF_{\xi'}^{-1}\left[\left(\frac{t_2}{t_2\omega_\lambda +\nu_*^{-1}\lambda}\right)e^{-\omega_\lambda x_N}\wht g(0)\right](x') \\
&-\sum_{k=1}^{N-1}\left[\left(\frac{\nu_*^{-1}i\xi_k}{t_2\omega_\lambda+\nu_*^{-1}\lambda}\right)e^{-\omega_\lambda x_N}\wht h_k(0)\right](x') \notag \\
&+\CF_{\xi'}^{-1}\left[\left(\frac{\nu_*^{-1}\lambda}{t_2\omega_\lambda+\nu_*^{-1}\lambda}\right)\CM(x_N)\,\wht g(0)\right](x') \notag \\
&+\sum_{k=1}^{N-1}\CF_{\xi'}^{-1}\left[\left(\frac{\nu_*^{-1}\omega_\lambda i\xi_k}{t_2\omega_\lambda+\nu_*^{-1}\lambda}\right)\CM(x_N)\,\wht h_k(0)\right](x'). \notag
\end{align}

Next, we consider the formulas of $u_J$. By \eqref{eq:413},
\begin{align}\label{solIII:u}
u_j&=\CF_{\xi'}^{-1}\left[e^{-\omega_\lambda x_N}\wht h_j(0)\right](x') \\
&-\CF_{\xi'}^{-1}\left[\left(\frac{i\xi_j\lambda}{t_2\omega_\lambda+\nu_*^{-1}\lambda}\right)\CM(x_N)\,\wht g(0)\right](x') \notag \\
&+ \sum_{k=1}^{N-1}\CF_{\xi'}^{-1}\left[\left(\frac{\xi_j\xi_k\omega_\lambda}{t_2\omega_\lambda+\nu_*^{-1}\lambda}\right)\CM(x_N)\,\wht h_k(0)\right] (x'), \notag \\
u_N&=
\CF_{\xi'}^{-1}\left[\frac{t_2\lambda}{t_2\omega_\lambda+\nu_*^{-1}\lambda}\CM(x_N)\,\wht g(0)\right](x') \notag \\
&+\sum_{k=1}^{N-1}\CF_{\xi'}^{-1}\left[\frac{t_2\omega_\lambda i\xi_k}{t_2\omega_\lambda+\nu_*^{-1}\lambda}\CM(x_N)\,\wht h_k(0)\right](x'). \notag
\end{align}

By Lemmas \ref{lemm:algebra}, \ref{lemm:symbol0}, and \ref{lemm:multicaseIII}, the symbols of $\rho$ and $u_J$ satisfy the following condition:
For $\rho$, there hold
\begin{align}\label{symbolIII:rho}
\frac{t_2}{t_2\omega_\lambda +\nu_*^{-1}\lambda}, 
\ \frac{\nu_*^{-1}i\xi_k}{t_2\omega_\lambda+\nu_*^{-1}\lambda}
&\in\BBM_{-1,1}(\BC_+), \\
\frac{\nu_*^{-1}\lambda}{t_2\omega_\lambda+\nu_*^{-1}\lambda},
\ \frac{\nu_*^{-1}\omega_\lambda i\xi_k}{t_2\omega_\lambda+\nu_*^{-1}\lambda}
&\in\BBM_{0,1}(\BC_+); \notag 
\end{align}
For $u_J$, there hold
\begin{align}\label{symbolIII:u}
1&\in\BBM_{0,1}(\BC_+), \\
\frac{i\xi_j\lambda}{t_2\omega_\lambda+\nu_*^{-1}\lambda}, 
\ \frac{\xi_j\xi_k\omega_\lambda}{t_2\omega_\lambda+\nu_*^{-1}\lambda}
&\in \BBM_{1,1}(\BC_+), \notag \\
\frac{t_2\lambda}{t_2\omega_\lambda+\nu_*^{-1}\lambda},
\ \frac{t_2\omega_\lambda i\xi_k}{t_2\omega_\lambda+\nu_*^{-1}\lambda}
&\in \BBM_{1,1}(\BC_+). \notag
\end{align}

Finally, combining \eqref{solIII:rho}-\eqref{symbolIII:u} with Lemmas \ref{lemm:multiplier1}, \ref{lemm:multiplier3}, and \ref{prop:R}
shows the existence of solution operators $\CA^2(\lambda)$ and $\CB^2(\lambda)$ stated in Theorem \ref{theo:half-red}.
This completes the proof of Theorem \ref{theo:half-red} for Case III.

\section{Proof of Theorem \ref{theo:half-red} for Case IV}\label{sec5}
This section proves Theorem \ref{theo:half-red} for Case IV.
Throughout this section, we assume that $\mu_*$, $\nu_*$, and $\kappa_*$ are positive constants satisfying
the condition of Case IV.
One then recalls Lemmas \ref{lemm:roots_s} \eqref{lemm:roots_s4} and \ref{lemm:roots_t} \eqref{coro:roots_s4}, i.e.
\begin{equation*}
\mu_*\neq \nu_*, \quad 
t_1=t_2=\sqrt{|\xi'|^2+\left(\frac{\mu_*+\nu_*}{2\kappa_*}\right)\lambda}, \quad t_2\neq \omega_\lambda,
\end{equation*}
which are often used in the following computations.
Let $J=1,\dots,N$ and $j=1,\dots,N-1$ in this section.

\subsection{Solution formulas}\label{subsec5-1}
One first considers
\eqref{170821_1}-\eqref{160827_1} with \eqref{eq:4}, \eqref{eq:44}, and \eqref{eq:div-vph}
in order to derive solution formulas of \eqref{eq2:half-red}.
In view of \eqref{170821_1}, \eqref{160827_1}, and Lemma \ref{lemm:roots_P}, we look for solutions $\wht u_J$ and $\wht\vph$ of the forms:
\begin{align}
\wht u_J
&=\alpha_J e^{-\omega_\lambda x_N}+\beta_J(e^{-t_2 x_N}-e^{-\omega_\lambda x_N})+\gamma_J x_N e^{-t_2 x_N}, \label{eq:200} \\
\wht \vph
&=\sigma e^{-t_2 x_N}+\tau x_N e^{-t_2 x_N}. \label{eq:200_3}
\end{align}
It then holds by \eqref{eq:div-vph} that
\begin{align}
0&=i\xi'\cdot\alpha'-i\xi'\cdot\beta'-\omega_\lambda \alpha_N+\omega_\lambda \beta_N, \label{eq:203} \\
\sigma&=i\xi'\cdot\beta'-t_2\beta_N+\gamma_N, \label{eq:203_3} \\
\tau&=i\xi'\cdot\gamma'-t_2\gamma_N. \label{eq:203_5}
\end{align}
Here note that
\begin{align*}
(\pd_N^2-\omega_\lambda)^2 \wht u_J
&=\beta_J(t_2^2-\omega_\lambda^2)e^{-t_2 x_N}
+\gamma_J\{-2t_2 e^{-t_2 x_N}+(t_2^2-\omega_\lambda^2)x_N e^{-t_2 x_N}\}, \\
(\pd_N^2-|\xi'|^2)\wht\vph
&=\sigma(t_2^2-|\xi'|^2)e^{-t_2 x_N}+\tau\{-2t_2 e^{-t_2 x_N}+(t_2^2-|\xi'|^2)x_N e^{-t_2 x_N}\}, \\
\pd_N(\pd_N^2-|\xi'|^2)\wht\vph
&=-\sigma t_2(t_2^2-|\xi'|^2)e^{-t_2 x_N} \\
&+\tau\{2 t_2^2 e^{-t_2 x_N}+(t_2^2-|\xi'|^2)e^{-t_2 x_N}-t_2(t_2^2-|\xi'|^2)x_N e^{-t_2 x_N}\},
\end{align*}
and also
$$
\pd_N\wht\vph=(-\sigma t_2+\tau)e^{-t_2 x_N}-\tau t_2 x_N e^{-t_2 x_N}.
$$
Inserting \eqref{eq:200} and \eqref{eq:200_3}, together with the last four relations, into \eqref{1023_18:eq5} and \eqref{1103_18:eq5} 
furnishes that
\begin{align}
\mu_*\lambda\{(t_2^2-\omega_\lambda^2)\beta_j-2 t_2 \gamma_j\}
+i\xi_j \sigma\{\nu_*\lambda-\kappa_*(t_2^2-|\xi'|^2)\}   
+2\kappa_*i\xi_j t_2\tau&=0,  \label{eq:204} \\
\mu_*\lambda(t_2^2-\omega_\lambda^2)\gamma_j+i\xi_j\tau\{\nu_*\lambda-\kappa_*(t_2^2-|\xi'|^2)\} &=0, \label{eq:205} \\
\mu_*\lambda\{(t_2^2-\omega_\lambda^2)\beta_N-2t_2\gamma_N\}  
+(-t_2\sigma+\tau)\{\nu_*\lambda-\kappa_*(t_2^2-|\xi'|^2)\}-2\kappa_* t_2^2\tau&=0,  \label{eq:220} \\
\mu_*\lambda(t_2^2-\omega_\lambda^2)\gamma_N-t_2\tau\{\nu_*\lambda-\kappa_*(t_2^2-|\xi'|^2)\}&=0. \label{eq:221}
\end{align}
It then holds by $t_2^2=|\xi'|^2+(\mu_*+\nu_*)\lambda/(2\kappa_*)$ that
\begin{equation*}
\nu_*\lambda-\kappa_*(t_2^2-|\xi'|^2)=\frac{(\nu_*-\mu_*)\lambda}{2},
\end{equation*}
which, inserted into \eqref{eq:204}-\eqref{eq:221}, furnishes
\begin{align}
2\mu_*\lambda\{(t_2^2-\omega_\lambda^2)\beta_j-2 t_2\gamma_j \}  
+  i\xi_j \sigma(\nu_*-\mu_*)\lambda +4\kappa_*i\xi_j t_2\tau &=0, \label{eq:206} \\
2\mu_*(t_2^2-\omega_\lambda^2)\gamma_j + i\xi_j \tau(\nu_*-\mu_*) &=0, \label{eq:207} \\
2\mu_*\lambda\{(t_2^2-\omega_\lambda^2)\beta_N-2 t_2\gamma_N\}
+ (-t_2\sigma+\tau)(\nu_*-\mu_*)\lambda -4\kappa_*t_2^2\tau&=0, \label{eq:212} \\
2\mu_*(t_2^2-\omega_\lambda^2)\gamma_N-t_2\tau(\nu_*-\mu_*)&=0. \label{eq:211} 
\end{align}
By \eqref{eq:207} and \eqref{eq:211}, 
\begin{equation*}
2\mu_*(t_2^2-\omega_\lambda^2)(t_2\gamma_j+i\xi_j\gamma_N)=0,
\end{equation*}
which, combined with $t_2\neq \omega_\lambda$, furnishes
\begin{equation}\label{eq:213}
\gamma_j=-\frac{i\xi_j}{t_2}\gamma_N.
\end{equation}
This relation yields
\begin{equation*}
i\xi'\cdot\gamma' =\frac{|\xi'|^2}{t_2}\gamma_N, 
\end{equation*}
and thus by \eqref{eq:203_5}
\begin{equation}\label{eq:214_2}
\tau=i\xi'\cdot\gamma'-t_2\gamma_N=-\left(\frac{t_2^2-|\xi'|^2}{t_2}\right)\gamma_N
=-\left(\frac{\mu_*+\nu_*}{2\kappa_*}\right)\frac{\lambda}{t_2}\gamma_N. 
\end{equation}
On the other hand, one multiplies \eqref{eq:206} by $t_2$ and \eqref{eq:212} by $i\xi_j$,
and sum the resultant equations in order to obtain
$$
2\mu_* t_2 \{(t_2^2-\omega_\lambda^2)\beta_j-2t_2\gamma_j\}
+2\mu_* i\xi_j \{(t_2^2-\omega_\lambda^2)\beta_N-2t_2\gamma_N\} 
+ i\xi_j \tau(\nu_*-\mu_*)=0.
$$
Inserting \eqref{eq:213} into the last equation furnishes
\begin{equation*}
2\mu_* t_2(t_2^2-\omega_\lambda^2)\beta_j
=-2\mu_*  i \xi_j  (t_2^2-\omega_\lambda^2)\beta_N
-i\xi_j\tau(\nu_*-\mu_*),
\end{equation*}
which, combined with $t_2\neq\omega_\lambda$, yields
\begin{equation}\label{eq:230_1}
\beta_j=-\frac{i\xi}{t_2}\beta_N-\frac{i\xi_j}{t_2}
\left(\frac{\tau}{t_2^2-\omega_\lambda^2}\right)\left(\frac{\nu_*-\mu_*}{2\mu_*}\right).
\end{equation}
One here notes by the assumption $\eta_*=0$
\begin{equation}\label{eq:230_3}
\frac{\mu_*+\nu_*}{2\kappa_*}=\frac{2}{\mu_*+\nu_*},
\end{equation}
and thus one observes, by $t_2^2=|\xi'|^2+(\mu_*+\nu_*)\lambda/(2\kappa_*)$ and $\omega_\lambda^2=|\xi'|^2+\mu_*^{-1}\lambda$,
\begin{equation}\label{eq:230_2}
t_2^2-\omega_\lambda^2=\left(\frac{\mu_*+\nu_*}{2\kappa_*}-\frac{1}{\mu_*}\right)\lambda
=-\frac{(\nu_*-\mu_*)}{\mu_*(\mu_*+\nu_*)}\lambda.
\end{equation}
Then, by $\mu_*\neq\nu_*$, \eqref{eq:214_2}, \eqref{eq:230_3}, and \eqref{eq:230_2},
\begin{align*}
\frac{\tau}{t_2^2-\omega_\lambda^2}
&=\left\{-\frac{(\nu_*-\mu_*)}{\mu_*(\mu_*+\nu_*)}\lambda\right\}^{-1}
\left\{-\left(\frac{2}{\mu_*+\nu_*}\right)\frac{\lambda}{t_2}\gamma_N\right\} \\
&=\left(\frac{2\mu_*}{\nu_*-\mu_*}\right)\frac{1}{t_2}\gamma_N,
\end{align*}
which, combined with \eqref{eq:230_1}, furnishes
\begin{equation}\label{eq:230}
\beta_j=-\frac{i\xi}{t_2}\beta_N-\frac{i\xi_j}{t_2^2}\gamma_N.
\end{equation}
This relation yields
\begin{equation}\label{eq:216_2}
i\xi'\cdot\beta'=\frac{|\xi'|^2}{t_2}\beta_N+\frac{|\xi'|^2}{t_2^2}\gamma_N, 
\end{equation}
and thus by \eqref{eq:203_3} 
\begin{equation}\label{eq:216_5}
\sigma=i\xi'\cdot\beta'-t_2\beta_N+\gamma_N=-\left(\frac{t_2^2-|\xi'|^2}{t_2}\right)\beta_N+\left(\frac{t_2^2+|\xi'|^2}{t_2^2}\right)\gamma_N.
\end{equation}


Next, we consider the boundary conditions.
By \eqref{eq:4} and \eqref{eq:200},
\begin{equation}\label{eq:217}
\alpha_j=\wht h_j(0), \quad \alpha_N=0.
\end{equation}
It then holds by the first relation of \eqref{eq:217} that
\begin{equation}\label{eq:217_2}
i\xi'\cdot\alpha'=i\xi'\cdot\wht \Bh'(0), \quad \wht\Bh'(0)=(\wht h_1(0),\dots,\wht h_{N-1}(0))^\SST.
\end{equation}
On the other hand, by \eqref{eq:200_3}, 
\begin{equation*}
\pd_N\wht\vph(0)=-t_2\sigma+\tau,
\end{equation*}
which, combined with \eqref{eq:44}, \eqref{eq:214_2}, \eqref{eq:216_5}, and $t_2^2=|\xi'|^2+(\mu_*+\nu_*)\lambda/(2\kappa_*)$, furnishes that 
\begin{equation}\label{eq:219_2}
(t_2^2-|\xi'|^2)\beta_N-2 t_2\gamma_N=\lambda\wht g(0).
\end{equation}

%
%
%
%
%

From now on, we derive simultaneous equations with respect to $\beta_N$ and $\gamma_N$.
To this end, we note by \eqref{eq:230_3} and \eqref{eq:230_2} that the following relation holds: 
\begin{align}\label{eq:216_3}
&2\mu_*(t_2^2-\omega_\lambda^2)-(\nu_*-\mu_*)|\xi'|^2 \\
&=-2\mu_*\frac{(\nu_*-\mu_*)}{\mu_*(\mu_*+\nu_*)} \lambda -(\nu_*-\mu_*)|\xi'|^2 \notag \\
&=-(\nu_*-\mu_*)\left(\frac{2}{\mu_*+\nu_*}\lambda+|\xi'|^2\right) \notag \\
&=-(\nu_*-\mu_*) t_2^2. \notag
\end{align}
One multiplies \eqref{eq:203} by $-(\nu_*-\mu_*)t_2^2$ in order to obtain
\begin{align*}
-(\nu_*-\mu_*)t_2^2(i\xi'\cdot\alpha'-i\xi'\cdot\beta'-\omega_\lambda\alpha_N+\omega_\lambda \beta_N)=0,
\end{align*}
which, combined with \eqref{eq:216_2}, \eqref{eq:217_2}, \eqref{eq:216_3}, and $\alpha_N=0$ of \eqref{eq:217} furnishes
\begin{align*}
-(\nu_*-\mu_*)t_2^2i\xi'\cdot \wht \Bh'(0)+(\nu_*-\mu_*)(t_2|\xi'|^2\beta_N+|\xi'|^2\gamma_N) & \\
+\{2\mu_*(t_2^2-\omega_\lambda^2)-(\nu_*-\mu_*)|\xi'|^2\}\omega_\lambda\beta_N&=0.
\end{align*}
It thus holds that
\begin{align}\label{eq:218}
&(t_2-\omega_\lambda)\{2\mu_*\omega_\lambda(t_2+\omega_\lambda)+(\nu_*-\mu_*)|\xi'|^2\}\beta_N+(\nu_*-\mu_*)|\xi'|^2\gamma_N \\
&=(\nu_*-\mu_*)t_2^2 i\xi'\cdot\wht \Bh'(0). \notag
\end{align}
On the other hand, \eqref{eq:216_3} is written as
\begin{equation}\label{eq:219_0}
(\nu_*-\mu_*)(t_2^2-|\xi'|^2)=-2\mu_*(t_2-\omega_\lambda)(t_2+\omega_\lambda).
\end{equation}
One then multiplies \eqref{eq:219_2} by $\nu_*-\mu_*$,
and inserts the last relation into the resultant formula in order to obtain
\begin{equation*}
-2\mu_*(t_2-\omega_\lambda)(t_2+\omega_\lambda)\beta_N-2(\nu_*-\mu_*)t_2\gamma_N =(\nu_*-\mu_*)\lambda \wht g(0).
\end{equation*}
Summing up this equation and \eqref{eq:218}, we have achieved
\begin{equation}\label{eq:220}
\BM\begin{pmatrix}\beta_N \\ \gamma_N\end{pmatrix}=
\begin{pmatrix}
(\nu_*-\mu_*)\lambda \wht g(0) \\ 
(\nu_*-\mu_*)t_2^2 i\xi'\cdot\wht \Bh'(0)
\end{pmatrix},
\end{equation}
where 
\begin{equation*}
\BM=
\begin{pmatrix}
-2\mu_*(t_2-\omega_\lambda)(t_2+\omega_\lambda) & -2(\nu_*-\mu_*)t_2 \\
(t_2-\omega_\lambda)\{2\mu_*\omega_\lambda(t_2+\omega_\lambda)+(\nu_*-\mu_*)|\xi'|^2\} & (\nu_*-\mu_*)|\xi'|^2 
\end{pmatrix}.
\end{equation*}

Let us solve \eqref{eq:220}. By direct calculations,
$$
\det\BM
=(\nu_*-\mu_*)(t_2-\omega_\lambda)\Fq(\xi',\lambda), 
$$
where
$$
\Fq(\xi',\lambda)
=2\{(2\mu_*(t_2+\omega_\lambda)\omega_\lambda+(\nu_*-\mu_*)|\xi'|^2) t_2-\mu_*(t_2+\omega_\lambda)|\xi'|^2\}.
$$
One here has
\begin{lemm}\label{lemm:detL_IV}
There holds $\det\BM\neq 0$ for any $(\xi',\lambda)\in\BR^{N-1}\times(\overline{\BC_+}\setminus\{0\})$.
\end{lemm}

\begin{proof}
The proof is similar to the proof of Lemma \ref{lemm:detL}, so that the detailed proof may be omitted.
\end{proof}

Let us write the inverse matrix $\BM^{-1}$ of $\BM$ as follows:
\begin{equation*}
\BM^{-1}=\frac{1}{\det\BM}
\begin{pmatrix}
M_{11} & M_{12} \\
M_{21} & M_{22}
\end{pmatrix},
\end{equation*}
where
\begin{align}\label{eq:236}
M_{11}&=(\nu_*-\mu_*)|\xi'|^2, \quad
M_{12}=2(\nu_*-\mu_*)t_2, \\
M_{21} &=-(t_2-\omega_\lambda)\{2\mu_*(t_2+\omega_\lambda)\omega_\lambda+(\nu_*-\mu_*)|\xi'|^2\}, \notag \\
M_{22}&=-2\mu_*(t_2-\omega_\lambda)(t_2+\omega_\lambda). \notag
\end{align}
One then sees that, by solving \eqref{eq:220},
\begin{align}\label{eq:235}
\beta_N
&=\frac{(\nu_*-\mu_*) M_{11}}{\det\BM}\lambda \wht g(0)
+\frac{(\nu_*-\mu_*) M_{12}}{\det\BM}t_2^2i\xi'\cdot\wht\Bh'(0), \\
\gamma_N
&=\frac{(\nu_*-\mu_*) M_{21}}{\det\BM}\lambda \wht g(0)
+\frac{(\nu_*-\mu_*) M_{22}}{\det\BM}t_2^2 i\xi'\cdot\wht\Bh'(0). \notag
\end{align}
On the other hand, one has, by \eqref{eq:200}, \eqref{eq:200_3}, \eqref{eq:213}, \eqref{eq:214_2},  \eqref{eq:230}, \eqref{eq:216_5}, and \eqref{eq:217},
\begin{align*}
\wht u_j(x_N)
&=\wht h_j(0)e^{-\omega_\lambda x_N} \\
&-\left(\frac{i\xi_j}{t_2}\beta_N+\frac{i\xi_j}{t_2^2}\gamma_N\right)(e^{-t_2 x_N}-e^{-\omega_\lambda x_N})
-\frac{i\xi_j}{t_2}\gamma_N x_N e^{-t_2 x_N}, \\
\wht u_N(x_N)
&=\beta_N(e^{-t_2 x_N}-e^{-\omega_\lambda x_N})+\gamma_N x_N e^{-t_2 x_N}, \\
\wht\vph(x_N)
&=\left\{-\left(\frac{t_2^2-|\xi'|^2}{t_2}\right)\beta_N+\left(\frac{t_2^2+|\xi'|^2}{t_2^2}\right)\gamma_N\right\}e^{-t_2 x_N} \\
&-\left(\frac{\mu_*+\nu_*}{2\kappa_*}\right)\frac{\lambda}{t_2}\gamma_N x_N e^{-t_2 x_N},
\end{align*}
and sets $\wht\rho(x_N)=-\lambda^{-1}\wht \vph(x_N)$ in view of \eqref{eq:1}.
Recall the inverse partial Fourier transform given in \eqref{defi:IPFT} and
set $\rho=\CF_{\xi'}^{-1}[\wht\rho(x_N)](x')$ and $u_J=\CF_{\xi'}^{-1}[\wht u_J(x_N)](x')$.
Then $\rho$ and $\Bu=(u_1,\dots, u_N)^\SST$ solve the system \eqref{eq2:half-red}.

\subsection{Analysis of symbols}\label{subsec5-2}
This subsection estimates several symbols arising from the representation formulas of solutions
obtained in Subsection \ref{subsec5-1}.
We often denote $\Fq(\xi',\lambda)$ by $\Fq$ for short in what follows.

One starts with

\begin{lemm}\label{lemm1:case4}
\begin{enumerate}[$(1)$]
\item
$\Fq\in\BBM_{3,1}(\BC_+)$.
\item
$M_{11}, M_{22}\in\BBM_{2,1}(\BC_+)$, $M_{12}\in\BBM_{1,1}(\BC_+)$, and
$M_{21}\in\BBM_{3,1}(\BC_+)$.
\item
It holds that
\begin{equation*}
\frac{t_2^2-|\xi'|^2}{t_2-\omega_\lambda}\in\BBM_{1,1}(\BC_+), \quad
\frac{M_{21}}{t_2-\omega_\lambda}\in \BBM_{2,1}(\BC_+), \quad
\frac{M_{22}}{t_2-\omega_\lambda}\in\BBM_{1,1}(\BC_+).
\end{equation*}
\end{enumerate}
\end{lemm}

\begin{proof}
(1), (2).
The required properties follow from \eqref{eq:236} and Lemmas \ref{lemm:algebra} and \ref{lemm:symbol0} immediately.

(3). We here prove the first assertion only.
Since it holds by \eqref{eq:230_3} that
$$
t_2^2-|\xi'|^2=\left(\frac{\mu_*+\nu_*}{2\kappa_*}\right)\lambda=\frac{2}{\mu_*+\nu_*}\lambda,
$$
one has by \eqref{eq:230_2}
\begin{align*}
\frac{t_2^2-|\xi'|^2}{t_2-\omega_\lambda}
&=\frac{t_2^2-|\xi'|^2}{t_2^2-\omega_\lambda^2}(t_2+\omega_\lambda) \\
&=\left\{-\frac{(\nu_*-\mu_*)}{\mu_*(\mu_*+\nu_*)}\lambda\right\}^{-1}\frac{2\lambda}{\mu_*+\nu_*}(t_2+\omega_\lambda) \\
&=-\frac{2\mu_*}{\nu_*-\mu_*}(t_2+\omega_\lambda).
\end{align*}
Combining this relation with Lemmas \ref{lemm:algebra} and \ref{lemm:symbol0} furnishes
$$
\frac{t_2^2-|\xi'|^2}{t_2-\omega_\lambda}\in\BBM_{1,1}(\BC_+).
$$
This completes the proof of the lemma.
\end{proof}

Similarly to the proof of Lemma \ref{lemm:lopatinski} and Corollary \ref{coro:lopatinski},
one can prove the following lemma by using Lemmas \ref{lemm:detL_IV} and \ref{lemm1:case4} (1).

\begin{lemm}\label{lemm2:case4}
 $\Fq^{-1}\in\BBM_{-3,1}(\BC_+)$.
\end{lemm}

\subsection{Proof of Theorem \ref{theo:half-red}}\label{subsec5-3}
This subsection proves Theorem \ref{theo:half-red} by means of results obtained in Subsections \ref{subsec5-1} and \ref{subsec5-2}.
Let $\CM(x_N)$ be given in \eqref{stkernel:2}.

First, we consider the solution formula of $\rho$.
It holds by \eqref{eq:235} that
\begin{align*}
&-\left(\frac{t_2^2-|\xi'|^2}{t_2}\right)\beta_N+\left(\frac{t_2^2+|\xi'|^2}{t_2^2}\right)\gamma_N \\
&=
\left\{-\left(\frac{t_2^2-|\xi'|^2}{t_2-\omega_\lambda}\right)t_2 M_{11}+(t_2^2+|\xi'|^2)
\left(\frac{M_{21}}{t_2-\omega_\lambda}\right)\right\}\frac{\lambda \wht g(0)}{t_2^2\Fq(\xi',\lambda)} \\
&+\left\{-t_2(t_2^2-|\xi'|^2)M_{12}+(t_2^2+|\xi'|^2)M_{22}\right\}\frac{i\xi'\cdot\wht\Bh'(0)}{(t_2-\omega_\lambda)\Fq(\xi',\lambda)}.
\end{align*}
On the other hand, since 
\begin{align*}
&-t_2(t_2^2-|\xi'|^2)M_{12}+(t_2^2+|\xi'|^2)M_{22} \\
&=-2t_2^2(\nu_*-\mu_*)(t_2^2-|\xi'|^2)-2\mu_*(t_2^2+|\xi'|^2)(t_2-\omega_\lambda)(t_2+\omega_\lambda),
\end{align*}
one has by \eqref{eq:230_3} and \eqref{eq:219_0}
\begin{align*}
&-t_2(t_2^2-|\xi'|^2)M_{12}+(t_2^2+|\xi'|^2)M_{22} \\
&=2\mu_*(t_2-\omega_\lambda)(t_2+\omega_\lambda)(t_2^2-|\xi'|^2) \\
&=\left(\frac{4\mu_*}{\mu_*+\nu_*}\right)\lambda(t_2-\omega_\lambda)(t_2+\omega_\lambda).
\end{align*}
Thus we have
\begin{align}\label{sol:rho_case4}
\rho
&=
-\CF_{\xi'}^{-1}
\left[
\left(-\frac{t_2^2-|\xi'|^2}{t_2-\omega_\lambda}\frac{M_{11}}{t_2\Fq(\xi',\lambda)}
+\frac{t_2^2+|\xi'|^2}{t_2^2\Fq(\xi',\lambda)}
\frac{M_{21}}{t_2-\omega_\lambda}\right) 
e^{-t_2 x_N}\wht g(0)
\right](x')  \\
&-\left(\frac{4\mu_*}{\mu_*+\nu_*}\right)\sum_{k=1}^{N-1}\CF_{\xi'}^{-1}\left[
\frac{ i\xi_k(t_2+\omega_\lambda)}{\Fq(\xi',\lambda)}e^{-t_2 x_N}\wht h_k(0)\right](x') \notag \\
&+\left(\frac{\mu_*+\nu_*}{2\kappa_*}\right)
\CF_{\xi'}^{-1}\left[\frac{\lambda}{t_2\Fq(\xi',\lambda)}\frac{M_{21}}{t_2-\omega_\lambda}x_N e^{-t_2 x_N}\wht g(0)\right](x') \notag \\
&+\left(\frac{\mu_*+\nu_*}{2\kappa_*}\right)\sum_{k=1}^{N-1}
\CF_{\xi'}^{-1}\left[\frac{i\xi_k t_2}{\Fq(\xi',\lambda)}\frac{M_{22}}{t_2-\omega_\lambda}x_N e^{-t_2 x_N}\wht h_k(0)\right](x'). \notag
\end{align}

Next, we consider the formulas of $u_J$.
By \eqref{eq:235},
\begin{align}\label{sol:u_case4}
u_j&=\CF_{\xi'}^{-1}\left[e^{-\omega_\lambda x_N}\wht h_j(0)\right](x') \\
&-\CF_{\xi'}^{-1}
\left[\left(\frac{\lambda i\xi_j M_{11}}{t_2 \Fq(\xi',\lambda)}
+\frac{\lambda i\xi_j M_{21}}{t_2^2\Fq(\xi',\lambda)}\right)\CM(x_N)\wht g(0)\right](x') \notag \\
&+\sum_{k=1}^{N-1}\CF_{\xi'}^{-1}\left[
\left(\frac{\xi_j\xi_k t_2 M_{12}}{\Fq(\xi',\lambda)}+\frac{\xi_j\xi_k M_{22}}{\Fq(\xi',\lambda)}\right)
\CM(x_N)\wht h_k(0)\right](x') \notag \\
&-\CF_{\xi'}^{-1}\left[\frac{\lambda i\xi_j}{t_2\Fq(\xi',\lambda)}\frac{M_{21}}{t_2-\omega_\lambda}x_N e^{-t_2 x_N}\wht g(0)\right](x') \notag \\
&+\sum_{k=1}^{N-1}\CF_{\xi'}^{-1}
\left[\frac{\xi_j\xi_k t_2}{\Fq(\xi',\lambda)}\frac{M_{22}}{t_2-\omega_\lambda}x_N e^{-t_2 x_N}\wht h_k(0)\right](x'), \notag \\
u_N&=
\CF_{\xi'}^{-1}\left[\frac{\lambda M_{11}}{\Fq(\xi',\lambda)}\CM(x_N)\wht g(0)\right](x') \notag \\
&+\sum_{k=1}^{N-1}\CF_{\xi'}^{-1}
\left[\frac{i\xi_k t_2^2 M_{12}}{\Fq(\xi',\lambda)}\CM(x_N)\wht h_k(0)\right](x') \notag \\
&+\CF_{\xi'}^{-1}
\left[\frac{\lambda}{\Fq(\xi',\lambda)}\frac{M_{21}}{t_2-\omega_\lambda}x_N e^{-t_2 x_N}\wht g(0)\right](x') \notag \\
&+\sum_{k=1}^{N-1}
\CF_{\xi'}^{-1}\left[\frac{i\xi_k t_2^2}{\Fq(\xi',\lambda)}\frac{M_{22}}{t_2-\omega_\lambda}x_N e^{-t_2 x_N}\wht h_k(0)\right](x'). \notag
\end{align}

By Lemmas \ref{lemm:algebra}, \ref{lemm:symbol0}, \ref{lemm1:case4}, and \ref{lemm2:case4},
the symbols of $\rho$ and $u_J$ satisfy the following conditions:
For $\rho$, there hold
\begin{align}\label{symbols:rho_case4}
\frac{t_2^2-|\xi'|^2}{t_2-\omega_\lambda}\frac{M_{11}}{t_2\Fq(\xi',\lambda)},
\ \frac{t_2^2+|\xi'|^2}{t_2^2\Fq(\xi',\lambda)}\frac{M_{21}}{t_2-\omega_\lambda},
\ \frac{ i\xi_k(t_2+\omega_\lambda)}{\Fq(\xi',\lambda)}
&\in\BBM_{-1,1}(\BC_+), \\
\frac{\lambda}{t_2\Fq(\xi',\lambda)}\frac{M_{21}}{t_2-\omega_\lambda},
\ \frac{i\xi_k t_2}{\Fq(\xi',\lambda)}\frac{M_{22}}{t_2-\omega_\lambda}
&\in \BBM_{0,1}(\BC_+); \notag
\end{align}
For $u_J$, there hold
\begin{align}\label{symbols:u_case4}
1&\in\BBM_{0,1}(\BC_+), \\
\frac{\lambda i\xi_j M_{11}}{t_2 \Fq(\xi',\lambda)},
\ \frac{\lambda i\xi_j M_{21}}{t_2^2\Fq(\xi',\lambda)},
\ \frac{\xi_j\xi_k t_2 M_{12}}{\Fq(\xi',\lambda)},
\ \frac{\xi_j\xi_k M_{22}}{\Fq(\xi',\lambda)}
&\in\BBM_{1,1}(\BC_+), \notag \\
\frac{\lambda i\xi_j}{t_2\Fq(\xi',\lambda)}\frac{M_{21}}{t_2-\omega_\lambda}, 
\ \frac{\xi_j\xi_k t_2}{\Fq(\xi',\lambda)}\frac{M_{22}}{t_2-\omega_\lambda}
&\in \BBM_{1,1}(\BC_+), \notag \\
\frac{\lambda M_{11}}{\Fq(\xi',\lambda)},
\ \frac{i\xi_k t_2^2 M_{12}}{\Fq(\xi',\lambda)},
\ \frac{\lambda}{\Fq(\xi',\lambda)}\frac{M_{21}}{t_2-\omega_\lambda},
\ \frac{i\xi_k t_2^2}{\Fq(\xi',\lambda)}\frac{M_{22}}{t_2-\omega_\lambda}
&\in \BBM_{1,1}(\BC_+). \notag
\end{align}

At this point, we introduce the following lemma 
in order to show the existence of $\CR$-bounded solutions operator families associated with
$\rho$ and $u_J$ (cf. the appendix below for the proof).

\begin{lemm}\label{lemm:multi-inhom}
Let $q\in(1,\infty)$ and $a\in(0,\infty)$.
Assume
$$
w_l(\xi',\lambda)\in\BBM_{l-1,1}(\BC_+) \quad (l=1,2),
$$
and set for $x=(x',x_N)\in\BR_+^N$
\begin{align*}
[W_l(\lambda)f](x)
=\CF_{\xi'}^{-1}\left[w_l(\xi',\lambda)x_Ne^{-\sqrt{|\xi'|^2+a\lambda}x_N}\wht f(\xi',0)\right](x') \quad (l=1,2),
\end{align*}
with $\lambda\in\BC_+$ and $f\in H_q^2(\BR_+^N)$.
Then the following assertions hold true:
\begin{enumerate}[$(1)$]
\item\label{lemm:multi-inhom_1}
For $\lambda\in\BC_+$, there is an operator $\wtd W_1(\lambda)$, with
$$
\wtd W_1(\lambda)\in \hlm(\BC_+,\CL(L_q(\BR_+^N)^{N^2+N+1},H_q^3(\BR_+^N))),
$$
such that for any $f\in H_q^2(\BR_+^N)$
$$
W_1(\lambda)f=\wtd W_1(\lambda)(\nabla^2 f,\lambda^{1/2}\nabla f,\lambda f).
$$
In addition, for $n=0,1$,
$$
\CR_{\CL(L_q(\BR_+^N)^{N^2+N+1},\FA_q^0(\BR_+^N))}
\left(\left\{\left.\left(\lambda\frac{d}{d\lambda}\right)^n\left(\CS_\lambda^0\wtd W_1(\lambda)\right)
\right|\lambda\in\BC_+\right\}\right)
\leq C,
$$
with some positive constant $C$ depending solely on $N$, $q$, and $a$.
Here, $\FA_q^0(\BR_+^N)$ and $\CS_\lambda^0$ are given in \eqref{defi:AB} for $G=\BR_+^N$.
\item\label{lemm:multi-inhom_2}
For $\lambda\in\BC_+$, there is an operator $\wtd W_2(\lambda)$, with
$$
\wtd W_2(\lambda)\in \hlm(\BC_+,\CL(L_q(\BR_+^N)^{N^2+N+1},H_q^2(\BR_+^N))),
$$
such that for any $f\in H_q^2(\BR_+^N)$
$$
W_2(\lambda)f=\wtd W_2(\lambda)(\nabla^2 f,\lambda^{1/2}\nabla f,\lambda f).
$$
In addition, for $n=0,1$,
$$
\CR_{\CL(L_q(\BR_+^N)^{N^2+N+1})}
\left(\left\{\left.\left(\lambda\frac{d}{d\lambda}\right)^n\left(\CT_\lambda \wtd W_2(\lambda)\right)
\right|\lambda\in\BC_+\right\}\right)
\leq C,
$$
with some positive constant $C$ depending solely on $N$, $q$, and $a$.
Here, $\CT_\lambda$ is given in \eqref{defi:AB}.
\end{enumerate}
\end{lemm}

Finally, combining \eqref{sol:rho_case4}-\eqref{symbols:u_case4}
with Lemmas \ref{lemm:multiplier1}, \ref{lemm:multiplier3}, \ref{prop:R}, and \ref{lemm:multi-inhom}
shows the existence of solution operators $\CA^2(\lambda)$ and $\CB^2(\lambda)$
stated in Theorem \ref{theo:half-red}.
This completes the proof of Theorem \ref{theo:half-red} for Case IV.

\section{Proof of Theorem \ref{theo:half-red} for Case V}\label{sec6}
This section proves Theorem \ref{theo:half-red} for Case V.
Throughout this section, we assume that $\mu_*$, $\nu_*$, and $\kappa_*$ are positive constants satisfying 
the condition of Case V.
One then recalls Lemmas \ref{lemm:roots_s} \eqref{lemm:roots_s5} and \ref{lemm:roots_t} \eqref{coro:roots_s5}, i.e.
\begin{equation*}
\mu_*=\nu_*, \quad  t_1=t_2=\omega_\lambda=\sqrt{|\xi'|^2+\mu_*^{-1}\lambda},
\end{equation*}
which are often used in the following computations.
Let $J=1,\dots,N$ and $j=1,\dots,N-1$ in this section.

\subsection{Solution formulas}\label{subsec6-1}
One first considers
\eqref{170821_1}-\eqref{160827_1} with \eqref{eq:4}, \eqref{eq:44}, and \eqref{eq:div-vph}
in order to derive solution formulas of \eqref{eq2:half-red}.
In view of \eqref{170821_1}, \eqref{160827_1}, and Lemma \ref{lemm:roots_P}, we look for solutions $\wht u_J$ and $\wht\vph$ of the forms:
\begin{align}
\wht u_J
&=\alpha_J e^{-\omega_\lambda x_N}+\beta_J x_Ne^{-\omega_\lambda x_N}+\gamma_J x_N^2e^{-\omega_\lambda x_N}, \label{eq:300} \\
\wht\vph
&=\sigma e^{-\omega_\lambda x_N}+\tau x_Ne^{-\omega_\lambda x_N}. \label{eq:300_0}
\end{align}
It then holds  by \eqref{eq:div-vph} that
\begin{align}
\sigma
&=i\xi'\cdot\alpha'-\omega_\lambda \alpha_N+\beta_N, \label{eq:303_2} \\
\tau
&=i\xi'\cdot\beta'-\omega_\lambda \beta_N+2\gamma_N, \label{eq:303_1} \\
0&=i\xi'\cdot\gamma'-\omega_\lambda\gamma_N. \label{eq:303}
\end{align}
On the other hand, by the assumption $\kappa_*=\mu_*\nu_*$,
\begin{equation*}
\nu_*\lambda-\kappa_*(\pd_N^2-|\xi'|^2)=-\mu_*\nu_*(\pd_N^2-\omega_\lambda^2),
\end{equation*}
which, combined with $\mu_*=\nu_*$, yields
$$
\nu_*\lambda-\kappa_*(\pd_N^2-|\xi'|^2)=-\mu_*^2(\pd_N^2-\omega_\lambda^2).
$$
Therefore, \eqref{1023_18:eq5} and \eqref{1103_18:eq5} are respectively equivalent to 
\begin{align}\label{eq:320}
\lambda(\pd_N^2-\omega_\lambda^2)\wht u_j-\mu_* i \xi_j(\pd_N^2-\omega_\lambda^2)\wht\vph&=0, \\
\lambda(\pd_N^2-\omega_\lambda^2)\wht u_N-\mu_* \pd_N(\pd_N^2-\omega_\lambda^2)\wht\vph&=0. \notag
\end{align}
Here note that
\begin{align*}
(\pd_N^2-\omega_\lambda^2)(x_N e^{-\omega_\lambda x_N})&=-2\omega_\lambda e^{-\omega_\lambda x_N}, \\
(\pd_N^2-\omega_\lambda^2)(x_N^2 e^{-\omega_\lambda x_N})&=2 e^{-\omega_\lambda x_N}-4\omega_\lambda x_N e^{-\omega_\lambda x_N}.
\end{align*}
Inserting \eqref{eq:300} and \eqref{eq:300_0}, together with the last two relations, into \eqref{eq:320} furnishes
\begin{alignat*}{2}
-4\lambda\omega_\lambda \gamma_j&=0, \quad
&\lambda(-2\omega_\lambda\beta_j+2\gamma_j)+2\mu_*\omega_\lambda i\xi_j\tau&=0, \\
-4\lambda \omega_\lambda\gamma_N&=0, \quad
&\lambda(-2\omega_\lambda +2\gamma_N)-2\mu_*\omega_\lambda^2\tau&=0,
\end{alignat*}
which yields 
\begin{align}
\gamma_J&=0, \label{eq:307} \\
-\lambda\beta_j +\mu_*  i\xi_j\tau&=0, \label{eq:308} \\
-\lambda\beta_N -\mu_*\omega_\lambda \tau&=0.  \label{eq:309} 
\end{align}

\begin{rema}
The relations \eqref{eq:307}-\eqref{eq:309} imply \eqref{eq:303_1} and \eqref{eq:303}.
\end{rema}

One has by \eqref{eq:300} and \eqref{eq:307}
\begin{equation}\label{eq:314}
\wht u_J=\alpha_J e^{-\omega_\lambda x_N}+\beta_J x_N e^{-\omega_\lambda x_N}, 
\end{equation}
and also by \eqref{eq:308} and \eqref{eq:309}
\begin{equation}\label{eq:312}
\beta_j = -\frac{i\xi_j}{\omega_\lambda}\beta_N.
\end{equation}

Next, we consider the boundary conditions. By \eqref{eq:4} and \eqref{eq:314},
\begin{equation}\label{eq:316}
\alpha_j=\wht h_j(0), \quad \alpha_N=0. 
\end{equation}
It then holds by the first relation of \eqref{eq:316} that
\begin{equation*}
i\xi'\cdot\alpha'=i\xi'\cdot\wht\Bh'(0), \quad \wht\Bh'(0)=(\wht h_1(0),\dots,\wht h_{N-1}(0))^\SST.
\end{equation*}
Combining this relation with \eqref{eq:303_2} and $\alpha_N=0$ of \eqref{eq:316} furnishes
\begin{equation}\label{eq:321_2}
\sigma=i\xi'\cdot\wht\Bh'(0)+\beta_N,
\end{equation}
while \eqref{eq:309} implies
\begin{equation}\label{eq:321_3}
\tau=-\frac{\mu_*^{-1}\lambda}{\omega_\lambda}\beta_N.
\end{equation}
On the other hand, by \eqref{eq:44} and \eqref{eq:300_0},
$$
-\omega_\lambda\sigma+\tau=\lambda\wht g(0),
$$
which, combined with \eqref{eq:321_2} and \eqref{eq:321_3}, furnishes
$$
-\omega_\lambda (i\xi'\cdot\wht\Bh'(0)+\beta_N)-\frac{\mu_*^{-1}\lambda}{\omega_\lambda}\beta_N=\lambda\wht g(0).
$$
Solving this equation with respect to $\beta_N$, we have
\begin{equation}\label{eq:335}
\beta_N=-\frac{\omega_\lambda}{\omega_\lambda^2+\mu_*^{-1}\lambda}\left(\lambda\wht g(0)+\omega_\lambda i\xi'\cdot\wht \Bh'(0)\right).
\end{equation}

Finally, one has, by \eqref{eq:300_0}, \eqref{eq:314}, \eqref{eq:312}, \eqref{eq:316}, \eqref{eq:321_2}, and \eqref{eq:321_3}, 
\begin{align*}
\wht u_j(x_N)
&=\wht h_j(0)e^{-\omega_\lambda x_N}-\frac{i\xi_j}{\omega_\lambda}\beta_N x_N e^{-\omega_\lambda x_N}, \\
\wht u_N(x_N)
&=\beta_N x_N e^{-\omega_\lambda x_N}, \\
\wht \vph(x_N)
&=(i\xi'\cdot\wht\Bh'(0)+\beta_N)e^{-\omega_\lambda x_N}
-\frac{\mu_*^{-1}\lambda}{\omega_\lambda}\beta_Nx_N e^{-\omega_\lambda x_N},
\end{align*}
and sets $\wht\rho(x_N)=-\lambda^{-1}\wht\vph(x_N)$ in view of \eqref{eq:1}.
Recall the inverse partial Fourier transform given in \eqref{defi:IPFT} and
set $\rho=\CF_{\xi'}^{-1}[\wht\rho(x_N)](x')$ and $u_J=\CF_{\xi'}^{-1}[\wht u_J(x_N)](x')$.
Then $\rho$ and $\Bu=(u_1,\dots, u_N)^\SST$ solve the system \eqref{eq2:half-red}.

In the last part of this section, we prove the following lemma.

\begin{lemm}\label{lemm:multi_case5}
$(\omega_\lambda^2+\mu_*^{-1}\lambda)^{-1}\in\BBM_{-2,1}(\BC_+)$.
\end{lemm}

\begin{proof}
Since $\omega_\lambda^2=|\xi'|^2+\mu_*^{-1}\lambda$, it holds that
$\omega_\lambda^2+\mu_*^{-1}\lambda=|\xi'|^2+2\mu_*^{-1}\lambda$.
One thus sees by Lemma \ref{lemm:symbol0} that  
$$
(\omega_\lambda^2+\mu_*^{-1}\lambda)^{-1}=(|\xi'|^2+2\mu_*^{-1}\lambda)^{-1}\in\BBM_{-2,1}(\BC_+).
$$
This completes the proof of the lemma.
\end{proof}

\subsection{Proof of Theorem \ref{theo:half-red}}\label{subsec6-2} 
This subsection proves Theorem \ref{theo:half-red} by means of solution formulas obtained in Subsection \ref{subsec6-1}.

First, we consider the formula of $\rho$.
Note that by \eqref{eq:335}
\begin{align*}
i\xi'\cdot\wht\Bh'(0)+\beta_N
&=-\frac{\omega_\lambda}{\omega_\lambda^2+\mu_*^{-1}\lambda} \lambda\wht g(0)
+\frac{\mu_*^{-1}\lambda }{\omega_\lambda^2+\mu_*^{-1}\lambda}i\xi'\cdot\wht \Bh'(0).
\end{align*}
It thus holds that
\begin{align}\label{solrho:case5}
\rho&=
\CF_{\xi'}^{-1}\left[\frac{\omega_\lambda}{\omega_\lambda^2+\mu_*^{-1}\lambda}e^{-\omega_\lambda x_N}\wht g(0)\right](x') \\
&-\sum_{k=1}^{N-1}
\CF_{\xi'}^{-1}\left[\frac{\mu_*^{-1} i\xi_k}{\omega_\lambda^2+\mu_*^{-1}\lambda}e^{-\omega_\lambda x_N}\wht h_k(0)\right](x') \notag \\
&-\CF_{\xi'}^{-1}\left[\frac{\mu_*^{-1}\lambda}{\omega_\lambda^2+\mu_*^{-1}\lambda}x_N e^{-\omega_\lambda x_N}\wht g(0)\right](x') \notag \\
&-\sum_{k=1}^{N-1}\CF_{\xi'}^{-1}
\left[\frac{\mu_*^{-1}i\xi_k\omega_\lambda}{\omega_\lambda^2+\mu_*^{-1}\lambda}x_N e^{-\omega_\lambda x_N}\wht h_k(0)\right](x'). \notag
\end{align}

Next, we consider the formulas of $u_J$. By \eqref{eq:335},
\begin{align}\label{solu:case5}
u_j
&=\CF_{\xi'}^{-1}\left[e^{-\omega_\lambda x_N}\wht h_j(0)\right](x') \\
&+\CF_{\xi'}^{-1}\left[\frac{i\xi_j \lambda}{\omega_\lambda^2+\mu_*^{-1}\lambda}x_N e^{-\omega_\lambda x_N}\wht g(0)\right](x') \notag \\
&-\sum_{k=1}^{N-1}\CF_{\xi'}^{-1}
\left[\frac{\xi_j\xi_k\omega_\lambda}{\omega_\lambda^2+\mu_*^{-1}\lambda}x_N e^{-\omega_\lambda x_N}\wht h_k(0)\right](x'), \notag \\
u_N
&=-\CF_{\xi'}^{-1}\left[\frac{\lambda\omega_\lambda}{\omega_\lambda^2+\mu_*^{-1}\lambda}x_N e^{-\omega_\lambda x_N}\wht g(0)\right](x') \notag \\
&-\sum_{k=1}^{N-1}\CF_{\xi'}^{-1}
\left[\frac{i\xi_k\omega_\lambda^2}{\omega_\lambda^2+\mu_*^{-1}\lambda}x_N e^{-\omega_\lambda x_N}\wht h_k(0)\right](x'). \notag
\end{align}

By Lemmas \ref{lemm:algebra}, \ref{lemm:symbol0}, and \ref{lemm:multi_case5},
the symbols of $\rho$ and $u_J$ satisfy the following conditions:
For $\rho$, there hold
\begin{align}\label{multipier:rho_case5}
\frac{\omega_\lambda}{\omega_\lambda^2+\mu_*^{-1}\lambda},
\ \frac{\mu_*^{-1} i\xi_k}{\omega_\lambda^2+\mu_*^{-1}\lambda}&\in\BBM_{-1,1}(\BC_+), \\
\frac{\mu_*^{-1}\lambda}{\omega_\lambda^2+\mu_*^{-1}\lambda},
\ \frac{\mu_*^{-1}i\xi_k\omega_\lambda}{\omega_\lambda^2+\mu_*^{-1}\lambda}&\in \BBM_{0,1}(\BC_+); \notag
\end{align}
For $u_J$, there hold
\begin{align}\label{multipier:u_case5}
1&\in\BBM_{0,1}(\BC_+), \\
\frac{i\xi_j \lambda}{\omega_\lambda^2+\mu_*^{-1}\lambda},
\ \frac{\xi_j\xi_k\omega_\lambda}{\omega_\lambda^2+\mu_*^{-1}\lambda}&\in\BBM_{1,1}(\BC_+), \notag \\
\frac{\lambda\omega_\lambda}{\omega_\lambda^2+\mu_*^{-1}\lambda},
\ \frac{i\xi_k\omega_\lambda^2}{\omega_\lambda^2+\mu_*^{-1}\lambda}&\in \BBM_{1,1}(\BC_+). \notag
\end{align}

Finally, combining \eqref{solrho:case5}-\eqref{multipier:u_case5}
with Lemmas \ref{lemm:multiplier1} and \ref{lemm:multi-inhom} shows the existence
of solution operators $\CA^2(\lambda)$ and $\CB^2(\lambda)$
stated in Theorem \ref{theo:half-red}.
This completes the proof of Theorem \ref{theo:half-red} for Case V.


\def\thesection{A}
\renewcommand{\theequation}{A.\arabic{equation}}
\section{}

This appendix proves Lemma \ref{lemm:multi-inhom} for\footnote{
Lemma \ref{lemm:multi-inhom} for $a>0$ follows from the result for $a=1$ and the definition of  the $\CR$-boundedness (cf. Definition \ref{defi:R}).
}
$a=1$. Set
$$
B=\sqrt{|\xi'|^2+ \lambda} \quad \text{for $(\xi',\lambda)\in \BR^{N-1}\times\Sigma_\vps$.}
$$
One then has

\begin{lemm}\label{lemm:A0}
Let $\vps\in(0,\pi/2)$.
\begin{enumerate}[$(1)$]
\item
For $r\in\BR$, there holds $B^r\in\BBM_{r,1}(\Sigma_\vps)$.
\item
For $x_N>0$, $n=0,1$, and multi-index $\alpha'\in\BN_0^{N-1}$, there holds
$$
\left|\pd_{\xi'}^{\alpha'}\left(\lambda\frac{d}{d\lambda}\right)^n e^{-B x_N}\right| 
\leq C_{\alpha',\vps} (|\lambda|^{1/2}+|\xi'|)^{-|\alpha'|}e^{-b_\vps(|\lambda|^{1/2}+|\xi'|)x_N},
$$
where $(\xi',\lambda)\in\BR^{N-1}\times\Sigma_\vps$,
with a positive constant $C_{\alpha',\vps}$ independent of $x_N$
and a positive constant $b_\vps$ depending only on $\vps$.
\end{enumerate}
\end{lemm}

\begin{proof}
(1), The required property follows from Lemma \ref{lemm:symbol0} with $\mu_*=1$.

(2). See \cite[Lemma 5.3]{SS12}. 
\end{proof}

At this point, we introduce two lemmas. 
The first one is essentially proved in \cite[Lemma 5.4]{SS12},
while the second one can be proved similarly to \cite[Lemmas 5.4 and 5.6]{SS12} (cf. also \cite{Saito1}).


\begin{lemm}\label{lemm:A2}
Let $q\in(1,\infty)$ and $\CK_\lambda(x)$ $(\lambda\in\BC_+)$ be a function on $\BR_+^N$.
Assume
$$
\left|\left(\lambda\frac{d}{d\lambda}\right)^n\CK_\lambda(x)\right|\leq \frac{M}{|x|^N} \quad (x\in\BR_+^N,\lambda\in\BC_+)
$$
for a positive constant $M$ and $n=0,1$, and set for $x=(x',x_N)\in\BR_+^N$ and $\lambda\in\BC_+$
$$
[K(\lambda)f](x)=\int_{\BR_+^N}\CK_\lambda(x'-y',x_N+y_N)f(y)\intd y.
$$
Then $K(\lambda)\in\hlm(\BC_+,\CL(L_q(\BR_+^N)))$,
and also for $n=0,1$
$$
\CR_{\CL(L_q(\BR_+^N))}
\left(\left\{\left.\left(\lambda\frac{d}{d\lambda}\right)^n K(\lambda)\right|\lambda\in\BC_+ \right\}\right) \leq C_{N,q}M,
$$
with some positive constant $C_{N,q}$.
\end{lemm}

\begin{lemm}\label{lemm:A2_2}
Let $q\in(1,\infty)$. Assume 
$$
k(\xi',\lambda)\in\BBM_{1,1}(\BC_+), \quad l_m(\xi',\lambda)\in\BBM_{m-3}(\BC_+) \quad (m=1,2),
$$
and set for $x=(x',x_N)\in\BR_+^N$ and $\lambda\in\BC_+$
\begin{align*}
[K(\lambda)f](x)&=\int_0^\infty
\CF_{\xi'}^{-1}\left[k(\xi',\lambda)e^{-B(x_N+y_N)}\wht f(\xi',y_N)\right](x')\intd y_N, \\
[L_m(\lambda)f](x)&=\int_0^\infty
\CF_{\xi'}^{-1}\left[l_m(\xi',\lambda)e^{-B(x_N+y_N)}\wht f(\xi',y_N)\right](x')\intd y_N \quad (m=1,2).
\end{align*}
Then the following assertions hold true:
\begin{enumerate}[$(1)$]
\item
There holds $K(\lambda)\in\hlm(\BC_+,\CL(L_q(\BR_+^N)))$, and also for $n=0,1$
$$
\CR_{\CL(L_q(\BR_+^N))}
\left(\left\{\left.\left(\lambda\frac{d}{d\lambda}\right)^n K(\lambda) \right|\lambda\in\BC_+\right\}\right)\leq C_{N,q},
$$
with some positive constant $C_{N,q}$.
\item
There holds $L_1(\lambda)\in\hlm(\BC_+,\CL(L_q(\BR_+^N),H_q^3(\BR_+^N)))$, and also for $n=0,1$
\begin{align*}
\CR_{\CL(L_q(\BR_+^N),\FA_q^0(\BR_+^N))}
\left(\left\{\left.\left(\lambda\frac{d}{d\lambda}\right)^n\left(\CS_\lambda^0 L_1(\lambda)\right) \right|\lambda\in\BC_+\right\}\right)\leq C_{N,q},
\end{align*}
with some positive constant $C_{N,q}$.
\item
There holds $L_2(\lambda)\in\hlm(\BC_+,\CL(L_q(\BR_+^N),H_q^2(\BR_+^N)))$, and also for $n=0,1$
\begin{align*}
\CR_{\CL(L_q(\BR_+^N),L_q(\BR_+^N)^{N^2+N+1})}
\left(\left\{\left.\left(\lambda\frac{d}{d\lambda}\right)^n\left(\CT_\lambda L_2(\lambda)\right) \right|\lambda\in\BC_+\right\}\right)\leq C_{N,q},
\end{align*}
with some positive constant $C_{N,q}$.
\end{enumerate}
\end{lemm}

One now proves

\begin{lemm}\label{lemm:A3}
Let $q\in(1,\infty)$. 
Assume
$$
k(\xi',\lambda)\in \BBM_{1,1}(\BC_+), \quad
l(\xi',\lambda) \in\BBM_{1,2}(\BC_+), \quad
m(\xi',\lambda)\in \BBM_{2,1}(\BC_+),
$$
and set for $x=(x',x_N)\in\BR_+^N$ and $\lambda\in\BC_+$
\begin{align*}
[K(\lambda) f](x) &=\int_0^\infty
\CF_{\xi'}^{-1}\left[\lambda^{1/2}k(\xi',\lambda) (x_N+y_N) e^{-B(x_N+y_N)}\wht f(\xi',y_N)\right](x')\intd y_N, \\
[L(\lambda)f](x) &=\int_0^\infty
\CF_{\xi'}^{-1}\left[|\xi'|l(\xi',\lambda)(x_N+y_N)e^{-B(x_N+y_N)}\wht f(\xi',y_N)\right](x')\intd y_N, \\
[M(\lambda)f](x) &= \int_0^\infty
\CF_{\xi'}^{-1}\left[m(\xi',\lambda)(x_N+y_N)e^{-B(x_N+y_N)}\wht f(\xi',y_N)\right](x')\intd y_N.
\end{align*}
Then $K(\lambda), L(\lambda), M(\lambda) \in\hlm(\BC_+,\CL(L_q(\BR_+^N)))$,
and also for $n=0,1$
\begin{align*}
\CR_{\CL(L_q(\BR_+^N))}
\left(\left\{\left(\left.\lambda\frac{d}{d\lambda}\right)^n K(\lambda)\right| \lambda\in\BC_+\right\}\right)
&\leq C_{N,q}, \\
\CR_{\CL(L_q(\BR_+^N))}
\left(\left\{\left(\left.\lambda\frac{d}{d\lambda}\right)^n L(\lambda)\right| \lambda\in\BC_+\right\}\right)
&\leq C_{N,q}, \\
\CR_{\CL(L_q(\BR_+^N))}
\left(\left\{\left(\left.\lambda\frac{d}{d\lambda}\right)^n M(\lambda)\right| \lambda\in\BC_+\right\}\right)
&\leq C_{N,q},
\end{align*}
with some positive constant $C_{N,q}$
\end{lemm}

\begin{proof}
Let $n=0,1$ and $\lambda\in\BC_+$ in this proof.

{\bf Case 1}. This case considers $K(\lambda)$.
Let $\CK_\lambda(x)$ be given by 
$$
\CK_\lambda(x)=\CK_\lambda(x',x_N)
=\CF_{\xi'}^{-1}\left[\lambda^{1/2}k(\xi',\lambda) x_N e^{-Bx_N}\right](x').
$$
Then $K(\lambda)f$ is written as
\begin{equation*}
[K(\lambda)f](x)
=\int_{\BR_+^N}\CK_\lambda(x'-y',x_N+y_N)f(y)\intd y.
\end{equation*}
Since $k(\xi',\lambda)\in\BBM_{1,1}(\BC_+)$, it holds by the Leibniz formula and Lemma \ref{lemm:A0} that
\begin{align}\label{eq:1_appendix}
&\left|\pd_{\xi'}^{\alpha'}\left(\lambda\frac{d}{d\lambda}\right)^n\left(\lambda^{1/2} k(\xi',\lambda)x_N e^{-Bx_N}\right)\right| \\
&\leq C_{\alpha'}|\lambda|^{1/2}(|\lambda|^{1/2}+|\xi'|)^{1-|\alpha'|}x_N e^{-b(|\lambda|^{1/2}+|\xi'|)x_N} \notag \\
&\leq C_{\alpha'}|\lambda|^{1/2} (|\lambda|^{1/2}+|\xi'|)^{-|\alpha'|}e^{-(b/2)(|\lambda|^{1/2}+|\xi'|)x_N}, \notag
\end{align}
where $b=b_\epsilon$ of Lemma \ref{lemm:A0}.
On the other hand, using the identity:
$$
e^{ix'\cdot\xi'}=\sum_{|\alpha'|=j}\left(\frac{-ix'}{|x'|^2}\right)^{\alpha'} \pd_{\xi'}^{\alpha'}e^{ix'\cdot\xi'} \quad (j\in\BN_0),
$$
we have, by integration by parts,
\begin{align*}
\left(\lambda\frac{d}{d\lambda}\right)^n \CK_\lambda(x)
&=\frac{1}{(2\pi)^{N-1}}\int_{\BR^{N-1}}e^{ix'\cdot\xi'}\left(\lambda\frac{d}{d\lambda}\right)^n
\left(\lambda^{1/2}k(\xi',\lambda) x_N e^{-Bx_N}\right) \intd \xi' \\
&=\frac{1}{(2\pi)^{N-1}}\sum_{|\alpha'|=N}
\left(\frac{ix'}{|x'|^2}\right)^{\alpha'}  \\
&\cdot\int_{\BR^{N-1}}e^{ix'\cdot\xi'}\pd_{\xi'}^{\alpha'}\left\{
\left(\lambda\frac{d}{d\lambda}\right)^n
\left(\lambda^{1/2}k(\xi',\lambda) x_N e^{-Bx_N}\right)
\right\}\intd\xi'.
\end{align*}
Combining this relation with \eqref{eq:1_appendix} yields
\begin{align*}
\left|\left(\lambda\frac{d}{d\lambda}\right)^n \CK_\lambda(x)\right|
\leq \frac{C}{|x'|^N} |\lambda|^{1/2}\int_{\BR^{N-1}}(|\lambda|^{1/2}+|\xi'|)^{-N}  \intd \xi',
\end{align*}
which, combined with
$$
\int_{\BR^{N-1}}(|\lambda|^{1/2}+|\xi'|)^{-N}  \intd \xi'=|\lambda|^{-1/2}\int_{\BR^{N-1}}(1+|\eta'|)^{-N}\intd \eta',
$$
furnishes
\begin{equation}\label{eq:2_appendix}
\left|\left(\lambda\frac{d}{d\lambda}\right)^n \CK_\lambda(x)\right|
\leq \frac{C}{|x'|^N} \quad (x\in\BR_+^N,\lambda\in\BC_+).
\end{equation}
By direct calculations, we have by \eqref{eq:1_appendix} with $|\alpha'|=0$
\begin{align*}
\left|\left(\lambda\frac{d}{d\lambda}\right)^n\CK_\lambda(x)\right|
&\leq C|\lambda|^{1/2}\int_{\BR^{N-1}}e^{-(b/2)(|\lambda|^{1/2}+|\xi'|)x_N} \intd \xi' \\
&\leq C (x_N)^{-1}\int_{\BR^{N-1}}e^{-(b/2)|\xi'|x_N} \intd \xi' \\
&\leq C(x_N)^{-N}.
\end{align*}
Combining this inequality with \eqref{eq:2_appendix} yields
$$
\left|\left(\lambda \frac{d}{d\lambda}\right)^n\CK_\lambda(x)\right|\leq \frac{C}{|x|^N} \quad (x\in\BR_+^N,\lambda\in\BC_+),
$$
which, combined with Lemma \ref{lemm:A2}, proves
that $K(\lambda)$ satisfies the required properties of Lemma \ref{lemm:A3}.

{\bf Case 2}. This case considers $L(\lambda)$.
Let $\CL_\lambda(x)$ be given by
$$
\CL_\lambda(x)=\CL_\lambda(x',x_N)=\CF_{\xi'}^{-1}\left[|\xi'| l(\xi',\lambda) x_N e^{-B x_N}\right](x').
$$
Then $L(\lambda)f$ is written as 
$$
[L(\lambda)f](x)=\int_{\BR_+^N}\CL_\lambda(x'-y',x_N+y_N) f(y)\intd y.
$$
Since $l(\xi',\lambda)\in\BBM_{1,2}(\BC_+)$, it holds by the Leibniz formula and Lemma \ref{lemm:A0} that
\begin{align}\label{eq:4_appendix}
&\left|\pd_{\xi'}^{\alpha'}\left(\lambda\frac{d}{d\lambda}\right)^n\left(|\xi'| l(\xi',\lambda)x_N e^{-Bx_N}\right)\right| \\
&\leq C_{\alpha'}|\xi'|^{1-|\alpha'|}(|\lambda|^{1/2}+|\xi'|)x_N e^{-b(|\lambda|^{1/2}+|\xi'|)x_N} \notag \\
&\leq C_{\alpha'}|\xi'|^{1-|\alpha'|}e^{-(b/2)(|\lambda|^{1/2}+|\xi'|)x_N}. \notag
\end{align}
Applying \cite[Theorem 2.3]{SS01} to $\lambda^n(d/d\lambda)^n\CL_\lambda(x)$ thus furnishes
\begin{equation}\label{eq:6_appendix}
\left|\left(\lambda\frac{d}{d\lambda}\right)^n\CL_\lambda(x)\right|\leq \frac{C}{|x'|^N} \quad (x\in\BR_+^N,\lambda\in\BC_+).
\end{equation}
Similarly to Case 1, we also have by \eqref{eq:4_appendix} with $|\alpha'|=0$
$$
\left|\left(\lambda\frac{d}{d\lambda}\right)^n\CL_\lambda(x)\right|\leq C(x_N)^{-N} \quad (x\in\BR_+^N,\lambda\in\BC_+).
$$
Combining this inequality with \eqref{eq:6_appendix} yields
$$
\left|\left(\lambda\frac{d}{d\lambda}\right)^n \CL_\lambda(x)\right| \leq \frac{C}{|x|^N} \quad (x\in\BR_+^N,\lambda\in\BC_+),
$$
which, combined with Lemma \ref{lemm:A2}, proves
that $L(\lambda)$ satisfies the required properties of Lemma \ref{lemm:A3}.

{\bf Case 3}. This case considers $M(\lambda)$.
Since
$$
m(\xi',\lambda)=\frac{B^2 m(\xi',\lambda)}{B^2}
=\lambda^{1/2}\frac{\lambda^{1/2}m(\xi',\lambda)}{B^2}+\sum_{j=1}^{N-1}|\xi'|\frac{\xi_j}{|\xi'|}\frac{\xi_j m(\xi',\lambda)}{B^2},
$$
one observes that 
\begin{align*}
&[M(\lambda)f](x) \\
&=\int_0^\infty
\CF_{\xi'}^{-1}\left[\lambda^{1/2}\frac{\lambda^{1/2}m(\xi',\lambda)}{B^2}(x_N+y_N)e^{-B(x_N+y_N)}\wht f(\xi',y_N)\right](x')\intd y_N \\
&+\sum_{j=1}^{N-1}\int_0^\infty
\CF_{\xi'}^{-1}\left[|\xi'|\frac{\xi_j}{|\xi'|}\frac{\xi_j m(\xi',\lambda)}{B^2}(x_N+y_N)e^{-B(x_N+y_N)}\wht f(\xi',y_N)\right](x')\intd y_N.
\end{align*}
It then holds by Lemmas \ref{lemm:algebra} and \ref{lemm:A0}  that
$$
\frac{\lambda^{1/2}m(\xi',\lambda)}{B^2}\in\BBM_{1,1}(\BC_+), \quad
\frac{\xi_j}{|\xi'|}\frac{\xi_j m(\xi',\lambda)}{B^2}\in\BBM_{1,2}(\BC_+),
$$
which, combined with Lemma \ref{prop:R} and the results obtained in Cases 1 and 2, proves that 
$M(\lambda)$ satisfies the required properties of Lemma \ref{lemm:A3}.
This completes the proof of the lemma.
\end{proof}

From now on, we prove Lemma \ref{lemm:multi-inhom}\footnote{
We here prove Lemma \ref{lemm:multi-inhom} \eqref{lemm:multi-inhom_2} only.
The proof of (1) is similar to one of (2).
}.
For functions $g(s)$ and $h(s)$ $(s\geq 0)$ with $\lim_{y_N\to \infty}g(x_N+y_N)h(y_N)=0$, 
one notes that
\begin{align*}
g(x_N)h(0)
&=-\int_0^\infty\frac{\pd}{\pd y_N}\left(g(x_N+y_N)h(y_N)\right) \intd y_N\\
&=-\int_0^\infty g'(x_N+y_N)h(y_N) \intd y_N \\
&-\int_0^\infty g(x_N+y_N)h'(y_N) \intd y_N,
\end{align*}
where $g'(s)=(dg/ds)(s)$ and $h'(s)=(dh/ds)(s)$.
By this relation, 
\begin{align*}
[W_2(\lambda)f](x)
&=-\int_0^\infty\CF_{\xi'}^{-1}\left[w_2(\xi',\lambda)e^{-B(x_N+y_N)}\wht f(\xi',y_N)\right](x') \\
&+\int_0^\infty\CF_{\xi'}^{-1}\left[w_2(\xi',\lambda)B(x_N+y_N) e^{-B(x_N+y_N)}\wht f(\xi',y_N)\right](x') \\
&-\int_0^\infty\CF_{\xi'}^{-1}\left[w_2(\xi',\lambda)(x_N+y_N)e^{-B(x_N+y_N)}\wht{\pd_N f}(\xi',y_N)\right](x'),
\end{align*}
which, combined with
$$
1=\frac{B^2}{B^2}=\frac{\lambda}{B^2}-\sum_{j=1}^{N-1}\frac{(i\xi_j)^2}{B^2}, \quad
B=\frac{B^2}{B}=\frac{\lambda}{B}-\sum_{j=1}^{N-1}\frac{(i\xi_j)^2}{B},
$$
furnishes that
\begin{align*}
&[W_2(\lambda)f](x)
=-\int_0^\infty\CF_{\xi'}^{-1}\left[\frac{w_2(\xi',\lambda)}{B^2}e^{-B(x_N+y_N)}\wht{\lambda f}(\xi',y_N)\right](x') \\
&\quad +\sum_{j=1}^{N-1}
\int_0^\infty\CF_{\xi'}^{-1}\left[\frac{w_2(\xi',\lambda)}{B^2}e^{-B(x_N+y_N)}\wht{\pd_j^2 f}(\xi',y_N)\right](x') \\
&\quad +\int_0^\infty\CF_{\xi'}^{-1}\left[\frac{w_2(\xi',\lambda)}{B}(x_N+y_N) e^{-B(x_N+y_N)}\wht{\lambda f}(\xi',y_N)\right](x') \\
&\quad -\sum_{j=1}^{N-1}\int_0^\infty\CF_{\xi'}^{-1}
\left[\frac{w_2(\xi',\lambda)}{B}(x_N+y_N) e^{-B(x_N+y_N)}\wht{\pd_j^2 f}(\xi',y_N)\right](x') \\
&\quad-\int_0^\infty\CF_{\xi'}^{-1}\left[\frac{\lambda^{1/2}w_2(\xi',\lambda)}{B^2}(x_N+y_N)e^{-B(x_N+y_N)}\wht{\lambda^{\frac{1}{2}}\pd_N f}(\xi',y_N)\right](x') \\
&\quad +\sum_{j=1}^{N-1}\int_0^\infty\CF_{\xi'}^{-1}\left[\frac{i\xi_j w_2(\xi',\lambda)}{B^2}(x_N+y_N)e^{-B(x_N+y_N)}\wht{\pd_j\pd_N f}(\xi',y_N)\right](x'). 
\end{align*}
From this viewpoint, for
$$
\BF=(\{F_{jk}\}_{j,k=1}^N,\{G_j\}_{j=1}^N,H)\in L_q(\BR_+^N)^{N^2}\times L_q(\BR_+^N)^N\times L_q(\BR_+^N),
$$
let us define
\begin{align*}
Z_1(\lambda)\BF
&=-\int_0^\infty\CF_{\xi'}^{-1}\left[\frac{w_2(\xi',\lambda)}{B^2}e^{-B(x_N+y_N)}\wht{H}(\xi',y_N)\right](x'), \\
Z_2(\lambda)\BF
&=
\sum_{j=1}^{N-1}
\int_0^\infty\CF_{\xi'}^{-1}\left[\frac{w_2(\xi',\lambda)}{B^2}e^{-B(x_N+y_N)}\wht{F_{jj}}(\xi',y_N)\right](x'), \\
Z_3(\lambda)\BF
&=\int_0^\infty\CF_{\xi'}^{-1}\left[\frac{w_2(\xi',\lambda)}{B}(x_N+y_N) e^{-B(x_N+y_N)}\wht{H}(\xi',y_N)\right](x'), \\
Z_4(\lambda)\BF
&=-\sum_{j=1}^{N-1}\int_0^\infty\CF_{\xi'}^{-1}
\left[\frac{w_2(\xi',\lambda)}{B}(x_N+y_N) e^{-B(x_N+y_N)}\wht{F_{jj}}(\xi',y_N)\right](x'), \\
Z_5(\lambda)\BF
&=-\int_0^\infty\CF_{\xi'}^{-1}\left[\frac{\lambda^{1/2}w_2(\xi',\lambda)}{B^2}(x_N+y_N)e^{-B(x_N+y_N)}\wht{G_N}(\xi',y_N)\right](x'), \\
Z_6(\lambda)\BF
&=\sum_{j=1}^{N-1}\int_0^\infty\CF_{\xi'}^{-1}\left[\frac{i\xi_j w_2(\xi',\lambda)}{B^2}(x_N+y_N)e^{-B(x_N+y_N)}\wht{F_{jN}}(\xi',y_N)\right](x').
\end{align*}
It is then clear that
\begin{align*}
W_2(\lambda) f = \sum_{k=1}^6 Z_k(\lambda)(\nabla^2 f,\lambda^{1/2}f,\lambda f).
\end{align*}
In addition, one has by $w_2(\xi',\lambda)\in\BBM_{1,1}(\BC_+)$ and Lemmas \ref{lemm:algebra} and \ref{lemm:A0} 
\begin{equation*}
\frac{w_2(\xi',\lambda)}{B^2} \in \BBM_{-1,1}(\BC_+),
\end{equation*}
which, combined with Lemma \ref{lemm:A2_2}, furnishes that for $k=1,2$ and $n=0,1$ 
\begin{align*}
Z_k(\lambda)\in\hlm(\BC_+,\CL(L_q(\BR_+^N)^{N^2+N+1},H_q^2(\BR_+^N))), \\
\CR_{\CL(L_q(\BR_+^N)^{N^2+N+1})}
\left(\left\{\left.\left(\lambda\frac{d}{d\lambda}\right)^n
\left(\CT_\lambda Z_k(\lambda)\right)\right|\lambda\in\BC_+\right\}\right)\leq C_{N,q},
\end{align*}
with some positive constant $C_{N,q}$.

Next, we consider $Z_3(\lambda)$. By direct calculations, we see for $j,k=1,\dots,N-1$
\begin{align*}
\pd_N^2 Z_3(\lambda)\BF
&=
-2\int_0^\infty\CF_{\xi'}^{-1}\left[w_2(\xi',\lambda) e^{-B(x_N+y_N)}\wht{H}(\xi',y_N)\right](x')  \\
&+\int_0^\infty\CF_{\xi'}^{-1}\left[w_2(\xi',\lambda)B(x_N+y_N) e^{-B(x_N+y_N)}\wht{H}(\xi',y_N)\right](x'),  \\
\pd_j\pd_N Z_3(\lambda)\BF
&=
\int_0^\infty\CF_{\xi'}^{-1}\left[\frac{i\xi_j w_2(\xi',\lambda)}{B} e^{-B(x_N+y_N)}\wht{H}(\xi',y_N)\right](x')  \\
&-\int_0^\infty\CF_{\xi'}^{-1}\left[i\xi_jw_2(\xi',\lambda)(x_N+y_N) e^{-B(x_N+y_N)}\wht{H}(\xi',y_N)\right](x'),  \\
\pd_j\pd_k Z_3(\lambda)\BF
&=
-\int_0^\infty\CF_{\xi'}^{-1}\left[\frac{\xi_j\xi_kw_2(\xi',\lambda)}{B}(x_N+y_N) e^{-B(x_N+y_N)}\wht{H}(\xi',y_N)\right](x'),
\\
\lambda^{1/2}\pd_N Z_3(\lambda)\BF
&=
\int_0^\infty\CF_{\xi'}^{-1}\left[\frac{\lambda^{1/2}w_2(\xi',\lambda)}{B} e^{-B(x_N+y_N)}\wht{H}(\xi',y_N)\right](x') \\
&-\int_0^\infty\CF_{\xi'}^{-1}\left[\lambda^{1/2}w_2(\xi',\lambda)(x_N+y_N) e^{-B(x_N+y_N)}\wht{H}(\xi',y_N)\right](x'),  \\
\lambda^{1/2}\pd_j Z_3(\lambda)\BF
&=\int_0^\infty\CF_{\xi'}^{-1}\left[\frac{\lambda^{1/2}i\xi_j w_2(\xi',\lambda)}{B}(x_N+y_N) e^{-B(x_N+y_N)}\wht{H}(\xi',y_N)\right](x'), \\
\lambda Z_3(\lambda)\BF
&=
\int_0^\infty\CF_{\xi'}^{-1}\left[\frac{\lambda w_2(\xi',\lambda)}{B}(x_N+y_N) e^{-B(x_N+y_N)}\wht{H}(\xi',y_N)\right](x').
\end{align*}
Note that by $w_2(\xi',\lambda)\in\BBM_{1,1}(\BC_+)$ and Lemmas \ref{lemm:algebra} and \ref{lemm:A0}
\begin{align*}
\frac{i\xi_j w_2(\xi',\lambda)}{B},\ \frac{\lambda^{1/2}w_2(\xi',\lambda)}{B}&\in\BBM_{1,1}(\BC_+); \\
w_2(\xi',\lambda)B,\ i\xi_j w_2(\xi',\lambda),\ \frac{\xi_j\xi_k w_2(\xi',\lambda)}{B}, 
\ \lambda^{1/2}w_2(\xi',\lambda), &\\
\frac{\lambda^{1/2}i\xi_j w_2(\xi',\lambda)}{B},
\ \frac{\lambda w_2(\xi',\lambda)}{B} &\in\BBM_{2,1}(\BC_+).
\end{align*}
One thus observes by Lemmas \ref{lemm:A2_2} and \ref{lemm:A3} that for $n=0,1$
\begin{align*}
Z_3(\lambda)\in\hlm(\BC_+,\CL(L_q(\BR_+^N)^{N^2+N+1},H_q^2(\BR_+^N))), \\
\CR_{\CL(L_q(\BR_+^N)^{N^2+N+1})}
\left(\left\{\left.\left(\lambda\frac{d}{d\lambda}\right)^n
\left(\CT_\lambda Z_3(\lambda)\right)\right|\lambda\in\BC_+\right\}\right)\leq C_{N,q}.
\end{align*}
Similarly to $Z_3(\lambda)$, it holds that for $k=4,5,6$ and $n=0,1$
\begin{align*}
Z_k(\lambda)\in\hlm(\BC_+,\CL(L_q(\BR_+^N)^{N^2+N+1},H_q^2(\BR_+^N))), \\
\CR_{\CL(L_q(\BR_+^N)^{N^2+N+1})}
\left(\left\{\left.\left(\lambda\frac{d}{d\lambda}\right)^n
\left(\CT_\lambda Z_k(\lambda)\right)\right|\lambda\in\BC_+\right\}\right)\leq C_{N,q}.
\end{align*}
This completes the proof of Lemma \ref{lemm:multi-inhom}.



\end{document}